\documentclass[oneside,reqno,a4paper,11pt]{amsart}
\usepackage{amsthm,amssymb,amsmath,soul,bropd,float,bookmark,amsaddr}
\usepackage{color,verbatim}
\allowdisplaybreaks[4]
\usepackage{graphicx}
\usepackage{hyperref}
\usepackage{esint}
\usepackage{tikz}
\newtheorem {theorem}{Theorem}[section]
\newtheorem {lemma}[theorem]{{\bf Lemma}}
\newtheorem {corollary}[theorem]{{\bf Corollary}}
\newtheorem {proposition}[theorem]{{\bf Proposition}}
\theoremstyle{remark}
\newtheorem{remark}{{\bf Remark}}[section]
\theoremstyle{definition}
\newtheorem {definition}{{\bf Definition}}[section]
\theoremstyle{plain} \numberwithin {equation}{section}
\newcommand{\hrefemail}[1]{\href{mailto:#1}{#1}}

\title[]{The free boundary of steady axisymmetric inviscid flow with vorticity $\uppercase\expandafter{\romannumeral2}$: near the non-degenerate points}

\author[Du, Yang]{Lili Du$^{\lowercase {1},\lowercase{2},\dagger}$,  Chunlei Yang$^{\lowercase{2},\ddagger}$}
\email{$^{\dagger}$\hrefemail{ dulili@szu.edu.cn}}
\thanks{$^{\dagger}$ The first author is supported by National Nature Science Foundation of China under Grant 12125102.}
\email{$^{\ddagger}$\hrefemail{yang\_chunlei@stu.scu.edu.cn}}
\thanks{$^{\ddagger}$ Corresponding author}
\thanks{Accepted for publication in Comm.Math.Phys.}
\begin{document}
\maketitle
    \begin{center}
		$^{1}$ School of Mathematical Science, Shenzhen University,
	
	    Shenzhen 518061, P.~R.~China.
	
	    $^{2}$ Department of Mathematics, Sichuan University, 
	
	    Chengdu 610064, P.~R.~China.
	\end{center}
	\begin{center}
	\end{center}
    \begin{abstract}
    	This is the sequel of the recent work (Du, Huang, Pu, Commun. Math. Phys, 2023, 400, 2137-2179) on axially symmetric gravity water waves with general vorticities, which has investigated the singular wave profile of the free boundary near the degenerate points. In this companion paper, we are interested in the regularity of the free surface of the water wave near the non-degenerate points. Precisely, we showed that the free boundary is $C^{1,\gamma}$ smooth for some $\gamma\in(0,1)$ near all non-degenerate points. The problem is intrinsically intertwined with the regularity theory of the semilinear Bernoulli-type free boundary problem. Our approach is closely related to the monotonicity formula developed by Weiss in his celebrated work (Weiss, J. Geom. Anal. 9: 317-326, 1999), and to a partial boundary Harnack inequality for the one-phase free boundary problem, which is due to De Silva (De Silva, Interfaces Free Bound. 13, 223-238, 2011). Mathematically, we associate the existence of viscosity solutions with the Weiss boundary-adjusted energy. Compared to the classical approach of Caffarelli (Caffarelli, Ann. Scoula Norm. Sup. Pisa, 15, 583-602, 1988), we provide an alternative proof of the existence of viscosity solutions for a large class of semilinear free boundary problems.
    \end{abstract}
\tableofcontents
\section{Introduction and Main Results}
In this paper we study the regularity of the free boundary of the following semilinear non-homogeneous elliptic equation with the Bernoulli-type boundary condition
\begin{align}\label{Formula: governed equation}
	\begin{cases}
		\begin{alignedat}{2}
			\operatorname{div}\left(\frac{1}{x}\nabla\psi\right)&=-xf(\psi)\quad&&\text{ in }\quad\Omega\cap\{\psi>0\},\\
			\frac{1}{x^{2}}|\nabla\psi|^{2}&=-y\quad&&\text{ on }\quad\Omega\cap\partial\{\psi>0\},
		\end{alignedat}
	\end{cases}
\end{align}
where $\psi$ is the stream function for 3D axisymmetric incompressible fluid and $\Omega$ is a connected open subset relative to the right half-plane $\mathbb{R}_{+}^{2}=\{(x,y)\in\mathbb{R}^{2}\colon x\geqslant0\}$. Here $f(\psi)$ is the vorticity function and $\Omega\cap\partial\{\psi>0\}$ is the free boundary of problem \eqref{Formula: governed equation}. The problem \eqref{Formula: governed equation} arises in the free boundary problem of incompressible inviscid flow in the gravity field, please see Section 2 in \cite{DHP2022} for derivation of the physical problem.

It is worth noting that the boundary condition $|\nabla\psi|^{2}=-x^{2}y$ on the free boundary suggests that the gradient of the stream function degenerates (in the sense that $|\nabla\psi|=0$ on the free boundary) at the symmetry axis and the $x$-axis (please see Fig \ref{Fig: NonvsD} below). The classification of all possible singular wave profiles for various \emph{degenerate points} has been completed very recently \cite{DHP2022}. As we know, the behavior of the free boundary away from the degenerate points is distinguished from the one near the degenerate points. The purpose of this paper is to examine the regularity of the free boundary near points that are away from the symmetry axis and the $x$-axis. We call such points \emph{non-degenerate free boundary points} (in the sense that $|\nabla\psi|\geqslant c_{0}$ on the free boundary for some $c_{0}>0$) or \emph{topological free boundary points} (please see Fig \ref{Fig: NonvsD} below). Additionally, we will demonstrate that the free boundary is $C^{1,\gamma}$ ($0<\gamma<1$) smooth near all such points.
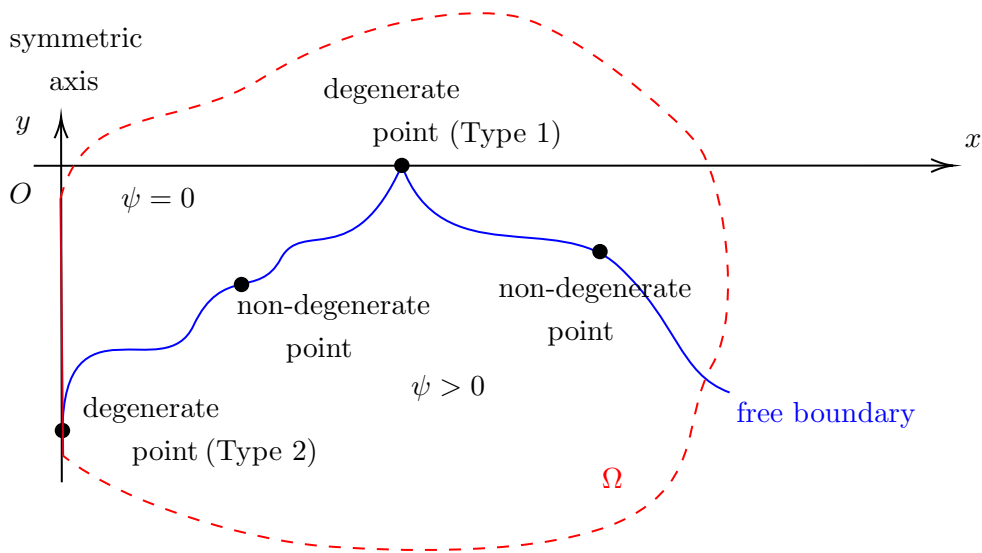
\begin{figure}[ht]
\tikzset{every picture/.style={line width=0.75pt}} 

\begin{tikzpicture}[x=0.75pt,y=0.75pt,yscale=-0.95,xscale=0.95]

\draw [line width=0.75]    (30.33,248) -- (29.97,65.95) ;
\draw [shift={(29.97,63.95)}, rotate = 89.89] [color={rgb, 255:red, 0; green, 0; blue, 0 }  ][line width=0.75]    (10.93,-3.29) .. controls (6.95,-1.4) and (3.31,-0.3) .. (0,0) .. controls (3.31,0.3) and (6.95,1.4) .. (10.93,3.29)   ;
\draw [line width=0.75]    (16.31,88.77) -- (473,88.67) ;
\draw [shift={(475,88.67)}, rotate = 179.99] [color={rgb, 255:red, 0; green, 0; blue, 0 }  ][line width=0.75]    (10.93,-3.29) .. controls (6.95,-1.4) and (3.31,-0.3) .. (0,0) .. controls (3.31,0.3) and (6.95,1.4) .. (10.93,3.29)   ;
\draw [color={rgb, 255:red, 0; green, 0; blue, 255 }  ,draw opacity=1 ][line width=0.75]    (30.65,221.96) .. controls (31.5,153.06) and (81.6,200.08) .. (96,169.23) .. controls (110.4,138.37) and (128.64,156.18) .. (139.52,135.61) .. controls (150.4,115.04) and (175.68,143.29) .. (200,88.6) ;
\draw [color={rgb, 255:red, 0; green, 0; blue, 255 }  ,draw opacity=1 ][line width=0.75]    (200,88.6) .. controls (219.73,143.2) and (277.76,111.03) .. (307.84,139.13) .. controls (337.93,167.22) and (337.93,193.31) .. (363.53,202.84) ;
\draw [line width=0.75]    (200,88.6) ;
\draw [shift={(200,88.6)}, rotate = 0] [color={rgb, 255:red, 0; green, 0; blue, 0 }  ][fill={rgb, 255:red, 0; green, 0; blue, 0 }  ][line width=0.75]      (0, 0) circle [x radius= 3.35, y radius= 3.35]   ;
\draw [line width=0.75]    (298.85,131.95) ;
\draw [shift={(298.85,131.95)}, rotate = 0] [color={rgb, 255:red, 0; green, 0; blue, 0 }  ][fill={rgb, 255:red, 0; green, 0; blue, 0 }  ][line width=0.75]      (0, 0) circle [x radius= 3.35, y radius= 3.35]   ;
\draw [line width=0.75]    (30.65,221.96) ;
\draw [shift={(30.65,221.96)}, rotate = 0] [color={rgb, 255:red, 0; green, 0; blue, 0 }  ][fill={rgb, 255:red, 0; green, 0; blue, 0 }  ][line width=0.75]      (0, 0) circle [x radius= 3.35, y radius= 3.35]   ;
\draw [color={rgb, 255:red, 255; green, 0; blue, 0 }  ,draw opacity=1 ][line width=0.75]  [dash pattern={on 4.5pt off 4.5pt}]  (29.67,105.33) .. controls (42,56) and (87,72) .. (130.33,44.67) .. controls (173.67,17.33) and (231,7.33) .. (258.33,14) .. controls (285.67,20.67) and (339.67,48.67) .. (353,65.33) .. controls (366.33,82) and (401.67,208) .. (386.33,234.67) .. controls (371,261.33) and (354.33,263.17) .. (315.67,283.33) .. controls (277,303.5) and (192.67,282.67) .. (159.67,280) .. controls (126.67,277.33) and (62.33,260) .. (31,234.67) ;
\draw [color={rgb, 255:red, 208; green, 2; blue, 27 }  ,draw opacity=1 ][line width=0.75]    (29.67,105.33) -- (31,234.67) ;
\draw [line width=0.75]    (120,148.5) ;
\draw [shift={(120,148.5)}, rotate = 0] [color={rgb, 255:red, 0; green, 0; blue, 0 }  ][fill={rgb, 255:red, 0; green, 0; blue, 0 }  ][line width=0.75]      (0, 0) circle [x radius= 3.35, y radius= 3.35]   ;

\draw (2.87,17.81) node [anchor=north west][inner sep=0.75pt]   [align=left] {symmetric\\ \ \ \ \ axis};
\draw (2.63,95.82) node [anchor=north west][inner sep=0.75pt]    {$O$};
\draw (479.45,71.45) node [anchor=north west][inner sep=0.75pt]    {$x$};
\draw (5.8,62.52) node [anchor=north west][inner sep=0.75pt]    {$y$};
\draw (298.32,239.55) node [anchor=north west][inner sep=0.75pt]  [color={rgb, 255:red, 208; green, 2; blue, 27 }  ,opacity=1 ]  {$\textcolor[rgb]{1,0,0}{\Omega }$};
\draw (203.13,192.21) node [anchor=north west][inner sep=0.75pt]    {$\psi  >0$};
\draw (58.52,97.7) node [anchor=north west][inner sep=0.75pt]    {$\psi =0$};
\draw (365.53,205.84) node [anchor=north west][inner sep=0.75pt]   [align=left] {\textcolor[rgb]{0,0,1}{free boundary}};
\draw (159.36,43.02) node [anchor=north west][inner sep=0.75pt]   [align=left] {degenerate\\ \ \ \ \ \ point};
\draw (246.67,143.33) node [anchor=north west][inner sep=0.75pt]   [align=left] {non-degenerate\\ \ \ \ \ \ point};
\draw (39.36,203.69) node [anchor=north west][inner sep=0.75pt]   [align=left] {degenerate\\ \ \ \ \ \ point};
\draw (222,64.33) node [anchor=north west][inner sep=0.75pt]   [align=left] {(Type 1)};
\draw (100.17,224.17) node [anchor=north west][inner sep=0.75pt]   [align=left] {(Type 2)};
\draw (116.17,151.83) node [anchor=north west][inner sep=0.75pt]   [align=left] {non-degenerate\\ \ \ \ \ \ point};

\end{tikzpicture}
   \caption{Degenerate points v.s. Non-degenerate points}
   \label{Fig: NonvsD}
\end{figure}

The rest of this section begins with a brief introduction to some background knowledge concerning the problem \eqref{Formula: governed equation}. The background concerns the semilinear one-phase Bernoulli-type problem (Subsection \ref{Subsection: one-phase Bernoulli-type problem}) and the axially symmetric water wave problem (Subsection \ref{Subsection: axially symmetric water wave problem}), respectively. The regularity of the free boundary in both lower and higher dimensions is included, as well as the \emph{Stokes conjecture} associated with the water wave problem. In Subsection \ref{Subsection: assumptions, notations and definitions}, we line up our basic assumptions, notations, and definitions, and finally state our main theorem in Subsection \ref{Subsection: main results}.
\subsection{Regularity of the one-phase Bernoulli-type problem}\label{Subsection: one-phase Bernoulli-type problem}
From the perspective of mathematics, \eqref{Formula: governed equation} can be regarded as a semilinear one-phase Bernoulli-type problem with a non-zero right-hand side. One-phase Bernoulli-type problems involve solving a PDE with unknown interfaces or surfaces, called free boundaries, and determining how smooth they are. The initial research focused on the homogeneous case, and the classical homogeneous one-phase Bernoulli-type problem was aimed at the minimization problem
\begin{align}\label{Formula: ACproblem}
    \min_{u}\left\lbrace\int_{\Omega}|\nabla u|^{2}+Q^{2}(X)I_{\{u>0\}}\:dX\colon u\in H^{1}(\Omega)\quad u=u^{0}\text{ on }S\subset\partial\Omega\right\rbrace.
\end{align}
Here $u^{0}\in H^{1}(\Omega)$ denotes the boundary function and is assumed to be non-negative, $I_{\{u>0\}}$ denotes the characteristic function of the set $\{u>0\}$. \caps{Alt} and \caps{Caffarelli} first systematically investigated minimizers of \eqref{Formula: ACproblem}. In their pioneering work \cite{AC1981}, they demonstrated that the free boundary is $C^{1,\gamma}$ ($0<\gamma<1$) smooth in two dimensions under the non-degenerate assumption 
\begin{align}\label{Formula: ACproblem1}
    0<Q_{\mathrm{min}}\leqslant Q(X)\leqslant Q_{\mathrm{max}}<+\infty\quad\text{ in }\Omega,
\end{align}
where $Q_{\mathrm{min}}$ and $Q_{\mathrm{max}}$ are two positive constants. The underlying ideas and mathematical tools that originated from this artwork were universal and could be applied to numerous, more complicated, homogeneous one-phase Bernoulli-type problems. See e.g. \cite{ACF1984,DP2005,OY1990}. The primary studies of the one-phase non-homogeneous Bernoulli-type problem are inadequate and limited. \caps{Friedman} first considered the axially symmetric finite cavity problem \cite{F1983} using a variational approach. He considered the minima of the problem
\begin{align}\label{Formula: F1983}
    \min_{u}\left\lbrace\int_{\Omega}\frac{1}{y}|\nabla u|^{2}-2P(X)yu+Q^{2}yI_{\{u>0\}}\:dX\colon u\in H^{1}(\Omega),u=u^{0}\text{ on }S\subset\partial\Omega\right\rbrace,
\end{align}
where $Q>0$ is a given constant. He showed that the free boundary is $C^{1,\gamma}$ smooth locally in $\Omega\subset\mathbb{R}^{2}\cap\{y>0\}$, provided that \eqref{Formula: ACproblem1} is satisfied and that $P(x)$ is a non-negative Lipschitz function in $\bar{\Omega}$. This result was further analyzed by the first author and his collaborator in \cite{CD2020}. They investigated the minima of the semilinear variational problem 
\begin{align}\label{Formula: CD2020}
    \min_{u}\left\lbrace\int_{\Omega}|\nabla u|^{2}+F(u)+Q^{2}(X)I_{\{u>0\}}\colon u\in H^{1}(\Omega), u=u^{0}\text{ on }S\subset\partial\Omega\right\rbrace,
\end{align}
and their results also indicated that the free boundary is $C^{1,\gamma}$ smooth locally in $\Omega\subset\mathbb{R}^{2}$. Based on that, they established the well-posedness of incompressible inviscid jet and cavitation flow with general vorticity. Comparatively, to \eqref{Formula: F1983}, which treats incompressible fluid with the positive constant vorticity, \eqref{Formula: CD2020} generalizes the case to the general vorticity.  

Furthermore, the structure of the free boundary in higher dimensions ($d\geqslant2$) associated with the problem \eqref{Formula: ACproblem} also has a wealthy result and has been a significant line of research in the PDE community in the last 30 years. It was due to \caps{Weiss} \cite{W1999}, who first realized that there exists a critical dimension $d^{*}\geqslant3$, so that the free boundary is smooth if $d<d^{*}$. In fact, it is now acknowledged that $d^{*}\in\{5,6,7\}$ and we refer to \cite{CJK2004,JS2015,SJ2009} for detailed research. In a recent work \cite{DY2023}, we studied the existence of $d^{*}$ for a semilinear minima problem and proved a similar result to the homogeneous case. The final twist is the precise value of $d^{*}$. To the best of our knowledge, this is still a big open problem, even to the homogeneous case.

The classical viscosity approach of \caps{Caffarelli} \cite{C1987,C1989,C1988} was developed by \caps{De Silva} \cite{S2011}, she studied the \emph{viscosity solutions} to the one-phase non-homogeneous Bernoulli-type free boundary problem
\begin{align}\label{Formula: S2011}
\begin{cases}
\begin{alignedat}{2}
\sum_{i,j=1}^{d}a_{ij}(x)D_{ij}u(x)&=f(x)\quad&&\text{ in }\varOmega^{+}(u):=\Omega\cap\{u>0\},\\
|\nabla u|&=Q(x)\quad&&\text{ on }\varGamma(u):=\Omega\cap\partial\{u>0\}.
\end{alignedat}
\end{cases}
\end{align}
She proved that \emph{flat} or \emph{Lipschitz} free boundaries are smooth. Compared to \caps{Caffarelli}'s approach, her method was quite flexible and accessible and applied to more general operators, even nonlinear ones. The core and innovative points of the argument were to bring out a partial boundary Harnack inequality and an improvement of flatness procedure for viscosity solutions to \eqref{Formula: S2011}. These techniques were then proven to be useful in the context of two-phase Bernoulli-type problems, and the existence and regularity of some two-phase Bernoulli-type problems were successfully established in a series of works \cite{SFS2015,SFS2014,SFS2015JMPA,SFS2016,SFS2017}. Notably, the regularity of these two-phase problems was obtained for more complex governing operators than the Laplacian, and for the more general free boundary condition of the form $u_{\nu}^{+}=G(u_{\nu}^{-},x)$. 

Inspired by \caps{De Silva}'s work, we will investigate viscosity solutions to our problem \eqref{Formula: governed equation}. The key ingredient in this paper is a Weiss-type monotonicity formula and a partial boundary Harnack inequality to problem \eqref{Formula: governed equation}, which is the main difference to the classical regularity work \cite{CD2020,F1983}. Another contribution of our work is to prove that local minimizers are Lipschitz viscosity solutions to problem \eqref{Formula: governed equation}. On the one hand, we relate the existence of viscosity solutions with the Weiss boundary-adjusted energy. On the other hand, this observation provides an alternative proof of the existence of viscosity solutions. In fact, \caps{De Silva}, \caps{Ferrari} and \caps{Salsa} have put forward in \cite{SFS2015} Section 6 an open problem:
\begin{quotation}
	\emph{... In general, we need to assume Lipschitz regularity of our solution. Indeed, in this generality, the existence of Lipschitz viscosity solutions with proper measure theoretical properties of the free boundary is an open problem and it will be object of future investigations...}
\end{quotation}
This problem was first answered in the context of two-phase problems in  \cite{SFS2015nonl} via Perron's method. However, the method in this paper depends on the Weiss boundary-adjusted energy, and we believe that techniques developed in our proof of existence of viscosity solutions are of independent interest and can be used on other types of free boundary problems.
\subsection{Axially symmetric water wave problem}\label{Subsection: axially symmetric water wave problem}
As we all know, three-dimensional waves are difficult to tackle, even for the irrotational flows. One way to recover this obstacle is to consider axially symmetric motions. As suggested in \cite{T1996} Chapter 16:
\begin{quotation}
	\emph{In proceeding to consider motion in three dimensions, we can no longer have recourse to the complex potential. The simplest case is that in which the motion is the same in every plane through a certain line called the axis}.
\end{quotation}
In the late 20th century, \caps{Caffarelli} and \caps{Friedman} \cite{ACF1982,ACF1983} investigated axially symmetric jet with gravity for irrotational incompressible waves ($f\equiv0$). Their approach was then fully developed in the study of compressible subsonic impinging jet \cite{CDZ2021} and incompressible impinging jet with gravity \cite{CDX2023}. These mentioned studies aimed to construct some global axially symmetric flows under physical boundary assumptions and can be understood as global well-posedness results. Locally, mathematicians are interested in those free boundary points at which the relative velocity vector field is zero. Such points are called \emph{stagnation points}, and waves with stagnation points are called \emph{extreme waves}. 

The study of extreme waves has a rich history and the primary research was to consider the existence and the profile of it in two dimensions. The most famous conjecture, which was due to \emph{Stokes}, predicted the existence of extreme waves in two dimensions and that the profile of extreme waves has corners with an included angle of $120^{\circ}$ at the crests. Literatures \cite{CS2004,CS2007,H2006,V2008,V2009} related to the Stokes conjecture aims to prove the existence of extreme waves under the least hypotheses on the structure of waves. In more recent developments \cite{VW2011}, \caps{V\u{a}rv\u{a}ruc\u{a}} and \caps{Weiss} proved the generalized Stokes conjecture for two-dimensional irrotational waves under the growth assumption near the free boundary, and this method was successfully applied to rotational waves in \cite{VW2012}.

However, the study of Stokes conjecture in the context of $3$-D axially symmetric case is very meager. The main drawback here is that the first equation in \eqref{Formula: governed equation} will completely break down near the axis of symmetry. This observation suggests that there are new degenerate points (or singularities) in $3$-D axisymmetric problems. In \cite{VW2014}, \caps{V\u{a}rv\u{a}ruc\u{a}} and \caps{Weiss} studied the possible singular profiles of the free boundary near the stagnation points of \eqref{Formula: governed equation} when $f\equiv0$, and they found that the free boundary will behave like a Stokes wave near those points on the $x$-axis (except the origin). In addition, they analyzed the stagnation points on the $y$-axis (except the origin) and the origin. Compared to the $2$-D case, new phenomena, such as vertical cusps and Garabedian pointed bubbles, have appeared. This result was then generalized to the case $f\neq 0$ in \cite{DHP2022} by the first author and his collaborators. 

Furthermore, it should be noted that the regularity of the free boundary away from stagnation points is also of particular interest. In his celebrated work \cite{V2009} on the Stokes conjecture for rotational waves, \caps{V\u{a}rv\u{a}ruc\u{a}} mentioned that:
\begin{quotation}
    \emph{Some problems left open by the present article are: ... the regularity of the wave profile away from stagnation points...}
\end{quotation}
In \cite{S2011}, \caps{De Silva} demonstrated that  the free boundary away from the stagnation points are smooth for two-dimensional rotational waves. Our paper shows that free boundaries are smooth away from the stagnation points in the $3$-D axisymmetric case. We therefore believe that for rotational waves, the profile of the free boundary away from stagnation points is smooth. Finally, we point out that another motivating factor for studying the regularity of the free boundary away from the degenerate points is that the notion of variational solutions and weak solutions defined in the previous work \cite{DHP2022} will coincide if the free boundary is smooth away from stagnation points.
\subsection{Assumptions, notations and definitions}\label{Subsection: assumptions, notations and definitions}
We now list some basic assumptions, definitions and notations that are helpful for our future discussion.
\subsubsection{Basic assumptions}
We assume that $f(z)\in C^{1,\beta}(\mathbb{R})$ so that
\begin{align}\label{Formula: assumptions on f}
	0\leqslant f(z)\leqslant F_{0}\quad\text{ for }z\leqslant0\quad\text{ and }\quad-F_{0}\leqslant f'(z)\leqslant 0\quad\text{ for }z\in\mathbb{R},
\end{align}
where $F_{0}\geqslant0$ is a given constant. Define $F(z):=\int_{0}^{z}f(s)\:ds$ and it follows from \eqref{Formula: assumptions on f} that $F''(z)\leqslant0$ for all $z\in\mathbb{R}$. Therefore, for any $p$, $q\in\mathbb{R}$, 
\begin{align*}
	F(q)-F(p)\geqslant F'(q)\cdot(q-p).
\end{align*}
\begin{remark}
    It should be noted that if $X_{0}\in\partial\{\psi>0\}\cap\Omega$, then $f(\psi(X_{0}))=f(0)\geqslant0$ and this implies that the vorticity is non-negative on the free boundary.
\end{remark}
In order to remain consistent with the previous work \cite{DHP2022}, we assume that the free boundary $\partial\{\psi>0\}$ is contained in the quarter plane $\{(x,y)\colon x\geqslant 0, y\leqslant0\}$ throughout the paper. However, we must \emph{address} that our regularity results depend only on the fact that near non-degenerate points, the gradient of solutions does not vanish. In fact, if we do not make such an assumption, then all results hold for the following problem
\begin{align*}
    \begin{cases}
        \begin{alignedat}{2}
			\operatorname{div}\left(\frac{1}{x}\nabla\psi\right)&=-xf(\psi)\quad&&\text{ in }\Omega\cap\{\psi>0\},\\
			\frac{1}{x^{2}}|\nabla\psi|^{2}&=\max\{y,0\}\quad&&\text{ on }\Omega\cap\partial\{\psi>0\}.
        \end{alignedat}
     \end{cases}
\end{align*}
\subsubsection{Notations}
We denote by $X=(x,y)$ a point in the \emph{quarter plane} $\mathbb{R}_{+,-}^{2}:=\{(x,y)\in \mathbb{R} ^{2}\colon x\geqslant 0, y\leqslant0\}$, by $dX=dxdy$ the \emph{volume element}, by $\mathcal{L}^{2}$ the two-dimensional \emph{Lebesgue measure} and by $B_{r}(X_{0}):=\{X\in\mathbb{R}^{2}\colon|X-X_{0}|<r\}$ the \emph{solid ball} of center $X_{0}$ and radius $r$. For any $t\in\mathbb{R}$, we denote by $t^{\pm}=\max\{\pm t,0\}$, which represents the \emph{positive} and \emph{negative} parts of $t$, respectively. Therefore $t=t^{+}-t^{-}$ and $|t|=t^{+}+t^{-}$. We will use functions in  \emph{weighted Lebesgue} $L_{\mathrm{w}}^{2}(E)$ and \emph{Sobolev space} $W_{\mathrm{w}}^{1,2}(E)$, which are defined by 
\begin{align*}
    L_{\mathrm{w}}^{2}(E):=\left\lbrace g\colon E\to\mathbb{R}\text{ is measurable and }\int_{E}\frac{1}{x}|g|^{2}\:dX<+\infty\right\rbrace,
\end{align*}
and
\begin{align*}
    W_{\mathrm{w}}^{1,2}(E):=\left\lbrace g\in L_{\mathrm{w}}^{2}(E)\colon\pd{g}{x}\in L_{\mathrm{w}}^{2}(E)\text{ and }\pd{g}{y}\in L_{\mathrm{w}}^{2}(E)\right\rbrace,
\end{align*}
respectively. 

The unit sphere in two dimensions will be denoted by $\mathbb{S}^{1}$ and we identify $\mathbb{S}^{1}$ with $\partial B_{1}$. In particular, for a given function $u\colon B_{r}(X_{0})\to\mathbb{R}$ with $X_{0}=(x_{0},y_{0})$, we will express it in polar coordinates by $u=u(r,\theta)$ and we obtain 
\begin{align*}
	|\nabla u|^{2}=(\partial_{r}u)^{2}+\frac{1}{r^{2}}(\partial_{\theta}u)^{2}\qquad\text{ and }\qquad\Delta u=\partial_{rr}u+\frac{1}{r}\partial_{r}u+\frac{1}{r^{2}}\partial_{\theta\theta}u.
\end{align*}
We denote by $N_{\psi}$ the set of all non-degenerate points (please see Fig \ref{Fig: Npsi} (a)) and it is defined to be the set
\begin{align*}
	N_{\psi}:=\{(x_{0},y_{0})\in\Omega\cap\partial\{\psi>0\}\colon x_{0}>0,y_{0}<0\}.
\end{align*}
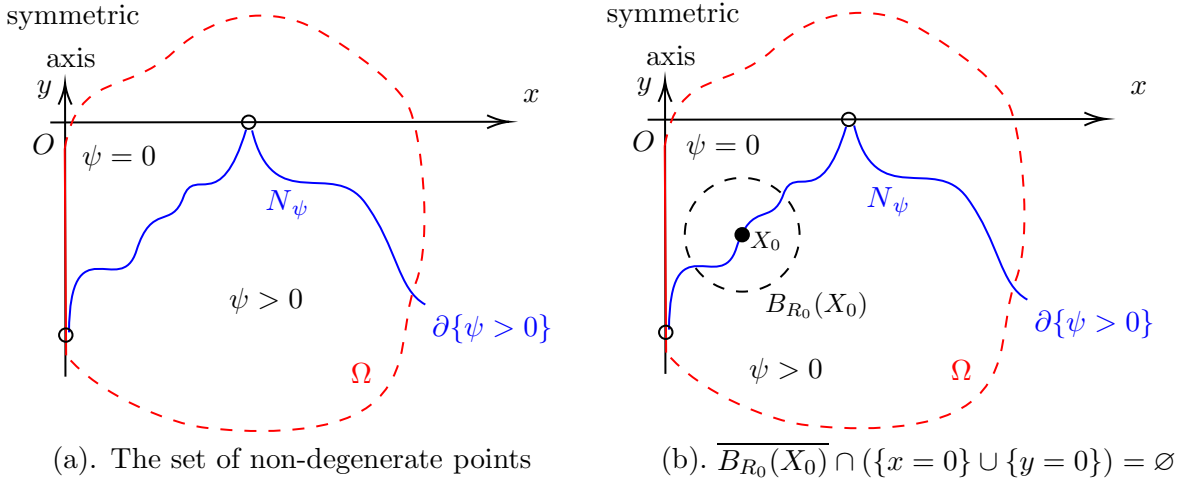
\begin{figure}[ht]

\tikzset{every picture/.style={line width=0.75pt}} 

\begin{tikzpicture}[x=0.75pt,y=0.75pt,yscale=-0.95,xscale=0.95]

\draw [line width=0.75]    (34.36,192.9) -- (34.16,47.09) ;
\draw [shift={(34.16,45.09)}, rotate = 89.92] [color={rgb, 255:red, 0; green, 0; blue, 0 }  ][line width=0.75]    (10.93,-3.29) .. controls (6.95,-1.4) and (3.31,-0.3) .. (0,0) .. controls (3.31,0.3) and (6.95,1.4) .. (10.93,3.29)   ;
\draw [line width=0.75]    (26.8,65.02) -- (253.67,64.67) ;
\draw [shift={(255.67,64.67)}, rotate = 179.91] [color={rgb, 255:red, 0; green, 0; blue, 0 }  ][line width=0.75]    (10.93,-3.29) .. controls (6.95,-1.4) and (3.31,-0.3) .. (0,0) .. controls (3.31,0.3) and (6.95,1.4) .. (10.93,3.29)   ;
\draw [color={rgb, 255:red, 0; green, 0; blue, 255 }  ,draw opacity=1 ][line width=0.75]    (36.13,170.61) .. controls (36.59,115.28) and (61.98,154.42) .. (69.74,129.64) .. controls (77.5,104.86) and (87.33,119.16) .. (93.19,102.64) .. controls (99.06,86.12) and (111.4,112.54) .. (124.51,68.62) ;
\draw [color={rgb, 255:red, 0; green, 0; blue, 255 }  ,draw opacity=1 ][line width=0.75]    (127.47,69.02) .. controls (138.11,112.87) and (167.69,82.9) .. (183.9,105.46) .. controls (200.11,128.02) and (200.11,148.98) .. (213.91,156.63) ;
\draw [color={rgb, 255:red, 255; green, 0; blue, 0 }  ,draw opacity=1 ][line width=0.75]  [dash pattern={on 4.5pt off 4.5pt}]  (34,78.32) .. controls (40.64,38.7) and (64.89,51.55) .. (88.24,29.6) .. controls (111.6,7.65) and (136.94,9.31) .. (151.67,14.67) .. controls (166.4,20.02) and (197.82,42.45) .. (205,55.84) .. controls (212.19,69.22) and (249.76,149.83) .. (241.5,171.25) .. controls (233.24,192.67) and (211.34,197.55) .. (190.5,213.75) .. controls (169.66,229.95) and (121.84,220.74) .. (104.05,218.59) .. controls (86.27,216.45) and (51.6,202.53) .. (34.71,182.19) ;
\draw [color={rgb, 255:red, 255; green, 0; blue, 0 }  ,draw opacity=1 ][line width=0.75]    (34,78.32) -- (34.71,182.19) ;
\draw [line width=0.75]    (34.53,171.98) ;
\draw [shift={(34.53,171.98)}, rotate = 0] [color={rgb, 255:red, 0; green, 0; blue, 0 }  ][line width=0.75]      (0, 0) circle [x radius= 3.35, y radius= 3.35]   ;
\draw [line width=0.75]    (125.79,64.89) ;
\draw [shift={(125.79,64.89)}, rotate = 0] [color={rgb, 255:red, 0; green, 0; blue, 0 }  ][line width=0.75]      (0, 0) circle [x radius= 3.35, y radius= 3.35]   ;
\draw [line width=0.75]    (333.71,191.5) -- (333.51,45.7) ;
\draw [shift={(333.51,43.7)}, rotate = 89.92] [color={rgb, 255:red, 0; green, 0; blue, 0 }  ][line width=0.75]    (10.93,-3.29) .. controls (6.95,-1.4) and (3.31,-0.3) .. (0,0) .. controls (3.31,0.3) and (6.95,1.4) .. (10.93,3.29)   ;
\draw [line width=0.75]    (324.82,63.63) -- (551.68,63.28) ;
\draw [shift={(553.68,63.27)}, rotate = 179.91] [color={rgb, 255:red, 0; green, 0; blue, 0 }  ][line width=0.75]    (10.93,-3.29) .. controls (6.95,-1.4) and (3.31,-0.3) .. (0,0) .. controls (3.31,0.3) and (6.95,1.4) .. (10.93,3.29)   ;
\draw [color={rgb, 255:red, 0; green, 0; blue, 255 }  ,draw opacity=1 ][line width=0.75]    (335.48,169.22) .. controls (335.94,113.88) and (361.33,153.02) .. (369.09,128.24) .. controls (376.85,103.47) and (386.68,117.77) .. (392.55,101.25) .. controls (398.41,84.73) and (410.75,111.14) .. (423.86,67.23) ;
\draw [color={rgb, 255:red, 0; green, 0; blue, 255 }  ,draw opacity=1 ][line width=0.75]    (427.86,66.56) .. controls (438.49,110.41) and (468.08,80.44) .. (484.29,103) .. controls (500.5,125.56) and (500.5,146.52) .. (514.3,154.17) ;
\draw [color={rgb, 255:red, 255; green, 0; blue, 0 }  ,draw opacity=1 ][line width=0.75]  [dash pattern={on 4.5pt off 4.5pt}]  (333.35,76.93) .. controls (339.99,37.31) and (364.24,50.16) .. (387.6,28.21) .. controls (410.95,6.26) and (436.29,7.92) .. (451.02,13.27) .. controls (465.75,18.63) and (497.17,41.06) .. (504.36,54.44) .. controls (511.54,67.83) and (545.26,143.83) .. (537,165.25) .. controls (528.74,186.67) and (512.84,195.55) .. (492,211.75) .. controls (471.16,227.95) and (421.19,219.34) .. (403.4,217.2) .. controls (385.62,215.06) and (350.95,201.14) .. (334.07,180.8) ;
\draw [color={rgb, 255:red, 255; green, 0; blue, 0 }  ,draw opacity=1 ][line width=0.75]    (333.35,76.93) -- (334.07,180.8) ;
\draw [line width=0.75]    (333.88,170.59) ;
\draw [shift={(333.88,170.59)}, rotate = 0] [color={rgb, 255:red, 0; green, 0; blue, 0 }  ][line width=0.75]      (0, 0) circle [x radius= 3.35, y radius= 3.35]   ;
\draw [line width=0.75]    (425.14,63.49) ;
\draw [shift={(425.14,63.49)}, rotate = 0] [color={rgb, 255:red, 0; green, 0; blue, 0 }  ][line width=0.75]      (0, 0) circle [x radius= 3.35, y radius= 3.35]   ;
\draw [line width=0.75]    (372,121.5) ;
\draw [shift={(372,121.5)}, rotate = 0] [color={rgb, 255:red, 0; green, 0; blue, 0 }  ][fill={rgb, 255:red, 0; green, 0; blue, 0 }  ][line width=0.75]      (0, 0) circle [x radius= 3.35, y radius= 3.35]   ;
\draw  [color={rgb, 255:red, 0; green, 0; blue, 0 }  ,draw opacity=1 ][dash pattern={on 4.5pt off 4.5pt}][line width=0.75]  (343.38,121.5) .. controls (343.38,105.69) and (356.19,92.88) .. (372,92.88) .. controls (387.81,92.88) and (400.63,105.69) .. (400.63,121.5) .. controls (400.63,137.31) and (387.81,150.13) .. (372,150.13) .. controls (356.19,150.13) and (343.38,137.31) .. (343.38,121.5) -- cycle ;

\draw (26.17,226.17) node [anchor=north west][inner sep=0.75pt]   [align=left] {(a). The set of non-degenerate points};
\draw (3.65,4.39) node [anchor=north west][inner sep=0.75pt]   [align=left] {symmetric\\ \ \ \ \ axis};
\draw (16.2,69.29) node [anchor=north west][inner sep=0.75pt]    {$O$};
\draw (261.18,47.05) node [anchor=north west][inner sep=0.75pt]    {$x$};
\draw (18.6,42.54) node [anchor=north west][inner sep=0.75pt]    {$y$};
\draw (175.31,184.71) node [anchor=north west][inner sep=0.75pt]  [color={rgb, 255:red, 208; green, 2; blue, 27 }  ,opacity=1 ]  {$\textcolor[rgb]{1,0,0}{\Omega }$};
\draw (88.56,143.2) node [anchor=north west][inner sep=0.75pt]    {$\psi  >0$};
\draw (41.22,74.66) node [anchor=north west][inner sep=0.75pt]    {$\psi =0$};
\draw (132.72,97.66) node [anchor=north west][inner sep=0.75pt]  [color={rgb, 255:red, 248; green, 231; blue, 28 }  ,opacity=1 ]  {$\textcolor[rgb]{0,0,1}{N}\textcolor[rgb]{0,0,1}{_{\psi }}$};
\draw (329.85,226.33) node [anchor=north west][inner sep=0.75pt]   [align=left] {(b). };
\draw (358,224.4) node [anchor=north west][inner sep=0.75pt]    {$\overline{B_{R_{0}}( X_{0})} \cap (\{x=0\} \cup \{y=0\}) =\varnothing $};
\draw (303,3) node [anchor=north west][inner sep=0.75pt]   [align=left] {symmetric\\ \ \ \ \ axis};
\draw (315.55,67.89) node [anchor=north west][inner sep=0.75pt]    {$O$};
\draw (317.95,41.15) node [anchor=north west][inner sep=0.75pt]    {$y$};
\draw (474.66,183.32) node [anchor=north west][inner sep=0.75pt]  [color={rgb, 255:red, 208; green, 2; blue, 27 }  ,opacity=1 ]  {$\textcolor[rgb]{1,0,0}{\Omega }$};
\draw (373.91,181.81) node [anchor=north west][inner sep=0.75pt]    {$\psi  >0$};
\draw (342.57,69.27) node [anchor=north west][inner sep=0.75pt]    {$\psi =0$};
\draw (432.57,96.27) node [anchor=north west][inner sep=0.75pt]  [color={rgb, 255:red, 248; green, 231; blue, 28 }  ,opacity=1 ]  {$\textcolor[rgb]{0,0,1}{N}\textcolor[rgb]{0,0,1}{_{\psi }}$};
\draw (382,148.9) node [anchor=north west][inner sep=0.75pt]  [font=\small]  {$B_{R_{0}}( X_{0})$};
\draw (374.5,118.4) node [anchor=north west][inner sep=0.75pt]  [font=\footnotesize]  {$X_{0}$};
\draw (564.18,43.05) node [anchor=north west][inner sep=0.75pt]    {$x$};
\draw (472.3,152.57) node [anchor=north west][inner sep=0.75pt]    {$\textcolor[rgb]{0,0,1}{\partial }\textcolor[rgb]{0,0,1}{\{}\textcolor[rgb]{0,0,1}{\psi  >0}\textcolor[rgb]{0,0,1}{\}}$};
\draw (176.41,154.53) node [anchor=north west][inner sep=0.75pt]    {$\textcolor[rgb]{0,0,1}{\partial }\textcolor[rgb]{0,0,1}{\{}\textcolor[rgb]{0,0,1}{\psi  >0}\textcolor[rgb]{0,0,1}{\}}$};

\end{tikzpicture}  
    \caption{The set of non-degenerate points $N_{\psi}$ and $B_{R_{0}}(X_{0})$.}
    \label{Fig: Npsi}
\end{figure}
Typically, one has the inclusion
\begin{align*}
	\partial\{\psi>0\}\subset\underbrace{\left(S_{\psi}^{s}\cup S_{\psi}^{a}\cup\{(0,0)\}\right)}_{\text{ Set of degenerate points}}\cup N_{\psi}.
\end{align*}
Here $S_{\psi}^{s}$ and $S_{\psi}^{a}$ are the set of Type 1 and Type 2 degenerate points, respectively (See \cite{DHP2022} for the precise definition). For any $X_{0}=(x_{0},y_{0})\in N_{\psi}$, we define  
\begin{align}\label{Formula: R0}
	R_{0}:=\frac{1}{2}\min\left\{ \frac{x_{0}}{2},-\frac{y_{0}}{2} \right\}.
\end{align}
It can be checked easily that (please see Fig \ref{Fig: Npsi} (b)) 
\begin{align*}
	\overline{B_{R_{0}}(X_{0})}\cap(\{x=0\}\cup\{y=0\})=\varnothing.
\end{align*}
Note that $R_{0}$ exists (depending on $X_{0}$ and the definition of $N_{\psi}$).

Let $\rho_{n}>0$ be a vanishing sequence as $n\to\infty$, we define the \emph{blow-up sequence}
\begin{align}\label{Formula: Type4}
	\psi_{n}(X):=\frac{\psi(X_{0}+\rho_{n}X)}{\rho_{n}}\quad\text{ for any }X_{0}\in N_{\psi}.
\end{align}
For any $X_{0}\in N_{\psi}$ and for any $r<\frac{R_{0}}{2}$, we define
\begin{align}\label{Formula: D(1)}
	\mathcal{D}_{1,X_{0},\psi}(r)=\int_{B_{r}(X_{0})}\left(\frac{1}{x}|\nabla\psi|^{2}-xyI_{\{\psi>0\}}-x\psi f(\psi)\right)\:dX,
\end{align}
and the \emph{boundary} $L^{2}$ energy
\begin{align}\label{Formula: D(2)}
	\mathcal{D}_{2,X_{0},\psi}(r):=\int_{\partial B_{r}(X_{0})}\frac{1}{x}\psi^{2}\:d\mathcal{H}^{1},
\end{align}
respectively. Furthermore, we define the so-called \emph{Weiss boundary-adjusted energy}
\begin{align}\label{Formula: D(3)}
	\mathcal{D}_{X_{0},\psi}(r):=r^{-2}\mathcal{D}_{1,X_{0},\psi}(r)-r^{-3}\mathcal{D}_{2,X_{0},\psi}(r).
\end{align}
In this paper, we study the following minimization problem
\begin{align*}
    \min_{\psi}\left\lbrace J(\psi)\colon\psi\in W_{\mathrm{w}}^{1,2}(B_{R_{0}}(X_{0})),X_{0}\in N_{\psi}\right\rbrace,
\end{align*}
where
\begin{align*}
	J(\psi):=\int_{B_{R_{0}}(X_{0})}\left(\frac{1}{x}|\nabla\psi|^{2}-2xF(\psi)-xyI_{\{\psi>0\}}\right)\:dX,
\end{align*}
and $F(\psi):=\int_{0}^{\psi}f(z)\:dz$.
\begin{remark}
    The idea of subtracting the boundary $L^{2}$ energy \eqref{Formula: D(2)} from \eqref{Formula: D(1)} was due to \caps{Weiss} \cite{W1999}, who investigated for the first time the one-homogeneous Lipschitz cones of the classical Alt-Caffarelli functional. His fundamental idea is that the Dirichlet energy for a homogeneous Bernoulli-type problem will gain some monotonicity after subtracting the boundary $L^{2}$ energy, which will eventually lead to the homogeneity of blow-up limits. Since our problem involves the axis of symmetry, the classical $|\nabla u|^{2}$ and $u^{2}$ included in the \caps{Weiss}'s monotonicity formula (we refer to \cite{W1999} Theorem 1.2 for the precise definition) are adapted to $\frac{1}{x}|\nabla\psi|^{2}$ and $\frac{1}{x}\psi^{2}$, respectively. Moreover, we will prove in Section \ref{Section: homogeneity of minimizers} that such an adjustment does not affect the monotonicity of $\mathcal{D}_{X_{0},\psi}$ defined in \eqref{Formula: D(3)}. And based on this observation, we successfully obtain the desired homogeneity of blow-up limits to our problem.
\end{remark}
\subsubsection{Definitions}
In this subsection, we introduce the concept of \emph{local minimizers} and \emph{viscosity solutions}. Both two objects are fundamental in the research of free boundary problems. Technically, minimizers of a given functional lead to a PDE that holds locally, whereas the information related to the regularity of the free boundary are often  hidden in viscosity solutions of a free boundary problem. We begin with the definition of local minimizers.
\begin{definition}[Local minimizers]
	Let $X_{0}\in N_{\psi}$, and we say that $\psi\in W_{\mathrm{w}}^{1,2}(B_{R_{0}}(X_{0}))$ is a \emph{local minimizer} of $J(\psi)$ if and only if 
	\begin{align*}
		J(\psi)\leqslant J(\phi)
	\end{align*}
	for every $\phi\in W_{\mathrm{w}}^{1,2}(B_{R_{0}}(X_{0}))$ with $\phi=\psi$ on $\partial B_{R_{0}}(X_{0})$. 
\end{definition}
\begin{remark}
If $X_{0}=(x_{0},y_{0})\in N_{\psi}$ and $r\leqslant\min\{\tfrac{x_{0}}{2},\tfrac{R_{0}}{2}\}$, then $\|\psi\|_{W_{\mathrm{w}}^{1,2}(B_{r}(X_{0}))}$ is equivalent to the usual $W^{1,2}$ norm of $\psi$. Indeed, via a direct computation, one has
\begin{align*}
	\frac{2}{3x_{0}}\int_{B_{r}(X_{0})}\psi^{2}\:dX\leqslant\int_{B_{r}(X_{0})}\frac{1}{x}\psi^{2}\:dX\leqslant\frac{2}{x_{0}}\int_{B_{r}(X_{0})}\psi^{2}\:dX,
\end{align*}
and
\begin{align*}
	\frac{2}{3x_{0}}\int_{B_{r}(X_{0})}|\nabla\psi|^{2}\:dX\leqslant\int_{B_{r}(X_{0})}\frac{1}{x}|\nabla\psi|^{2}\:dX\leqslant\frac{2}{x_{0}}\int_{B_{r}(X_{0})}|\nabla\psi|^{2}\:dX.
\end{align*}
In the rest of the paper, we will apply the classical norm $\|\cdot\|_{W^{1,2}(B_{r}(X_{0}))}$ for the abuse of notations. Moreover, the equality $\phi=\psi$ on $\partial B_{R_{0}}(X_{0})$ in the above definition is understood in the sense of traces for classical $W_{0}^{1,2}$ norms.
\end{remark}
\begin{remark}
	The existence of local minimizers of $J(\psi)$ in $B_{R_{0}}(X_{0})$ can be obtained by specifying a given Sobolev trace on $\partial B_{R_{0}}(X_{0})$. In other words, given any nonnegative $\Psi_{0}\in W^{1,2}(B_{R_{0}}(X_{0}))$ we can find a minimizer $\psi$ to the functional $J(\psi; B_{R_{0}}(X_{0}))$ which satisfies $\psi-\Psi_{0}\in W_{0}^{1,2}(B_{R_{0}}(X_{0}))$.
\end{remark}
In Section \ref{Section: regularity of the non-degenerate points}, we will use the notion of  viscosity solutions to the problem \eqref{Formula: governed equation}. (See also \cite{CIL1992} for a comprehensive research on viscosity solutions.) To this end, we first need the notion of "touch" .
\begin{definition}[Touch]
	Let $X_{0}\in N_{\psi}$ and let $\psi$, $\phi\in C^{0}(B_{R_{0}}(X_{0}))$, we say that $\phi$ \emph{touches} $\psi$ from below at $X_{1}\in B_{R_{0}}(X_{0})$ if $\psi(X_{1})=\phi(X_{1})$ and 
	\begin{align*}
	    \psi(X)\geqslant\phi(X)\quad\text{ in a small neighborhood } B_{r}(X_{1})\text{ of }X_{1}.
	\end{align*}
	If this inequality is strict in $B_{r}(X_{1})$, we say that $\phi$ touches $\psi$ \emph{strictly} from below. Touching (strictly) from above can be defined in a similar way, replacing $\geqslant$ by $\leqslant$.
\end{definition}
We are now in a position to define viscosity solutions to the problem \eqref{Formula: governed equation}.
\begin{definition}[Viscosity solutions]\label{Definition: viscosity solutions}
	Let $X_{0}\in N_{\psi}$ and let $\psi$ be a nonnegative continuous function in $B_{R_{0}}(X_{0})$. We say that $\psi$ is a \emph{viscosity solution} to \eqref{Formula: governed equation} in $B_{R_{0}}(X_{0})$, if for every point $X_{1}=(x_{1},y_{1})\in\overline{B_{R_{0}}(X_{0})\cap\{\psi>0\}}$, the following does hold:
	\begin{enumerate}
		\item If $X_{1}=(x_{1},y_{1})\in B_{R_{0}}(X_{0})\cap\{\psi>0\}$ and
		\begin{itemize}
			\item if $\phi\in C^{2}(B_{R_{0}}(X_{0})\cap\{\psi>0\})$ touches $\psi$ from below at $X_{1}$, then 
			\begin{align*}
				\operatorname{div}\left(\frac{1}{x}\nabla\phi\right)(X_{1})+x_{1}f(\phi(X_{1}))\leqslant0,
			\end{align*}
			\item and if $\phi\in C^{2}(B_{R_{0}}(X_{0})\cap\{\psi>0\})$ touches $\psi$ from above at $X_{1}=(x_{1},y_{1})$, then
			\begin{align*}
				\operatorname{div}\left(\frac{1}{x}\nabla\phi\right)(X_{1})+x_{1}f(\phi(X_{1}))\geqslant0.
			\end{align*}
		\end{itemize}
		\item If $X_{1}=(x_{1},y_{1})\in B_{R_{0}}(X_{0})\cap\partial\{\psi>0\}$ and
		\begin{itemize}
			\item if $\phi\in C^{2}(B_{R}(X_{0}))$ touches $\psi$ from below at $X_{1}$, then $|\nabla\phi(X_{1})|\leqslant x_{1}\sqrt{-y_{1}}$,
			\item if $\phi\in C^{2}(B_{R}(X_{0}))$ touches $\psi$ from above at $X_{1}$, then $|\nabla\phi(X_{1})|\geqslant x_{1}\sqrt{-y_{1}}$.
		\end{itemize}
	\end{enumerate}
\end{definition}
\subsection{Plan of the paper and main results}\label{Subsection: main results}
Our paper is composed as follows: 

In Section \ref{Section: preliminaries}, we provide some primary properties of local minimizers in Proposition \ref{Proposition: properties for local minimizers}, including the non-negativity, the H\"{o}lder continuity, and the governing equation (weak sense) of $\psi$ in $B_{R_{0}}(X_{0})$. In Section \ref{Section: optimal regularity and non-degeneracy of local minimizers}, we focus on the optimal regularity for $\psi$ in $B_{R_{0}}(X_{0})$. We prove that $\psi\in C^{0,1}(B_{R_{0}}(X_{0}))$ in Proposition \ref{Proposition: Lipschitz regularity for local minimizers}. As a direct consequence, minimizers satisfy an optimal linear growth property near the free boundary (Lemma \ref{Lemma: Optimal linear growth}). Later, we establish a non-degeneracy lemma for minimizers (Lemma \ref{Lemma: Non-degeneracy of minimizers}). In Section \ref{Section: homogeneity of minimizers}, we aim to investigate the convergence of blow-up limits (recalling \eqref{Formula: Type4}) as $n\to\infty$, and we show that there exists a so-called \emph{blow-up limit} $\psi_{0}$ so that $\psi_{n}$ converges uniformly to $\psi_{0}$ and $I_{\{\psi_{n}>0\}}$ converges to $I_{\{\psi_{0}>0\}}$ locally in $L_{\mathrm{loc}}^{1}(\mathbb{R}^{2})$ (Lemma \ref{Lemma: Structure of the blow-up limits}). As a direct corollary, we obtain some basic properties of $\psi_{0}$ in Corollary \ref{Corollary: Properties of blow-up limits}. Next, we have to deal with the homogeneity of $\psi_{0}$, and our tool is the functional $\mathcal{D}_{X_{0},\psi}(r)$ introduced in \eqref{Formula: D(3)}. We obtain the monotonicity of $\mathcal{D}_{X_{0},\psi}(r)$ with respect to $r$ in Lemma \ref{Lemma: monotonicity formula}, and in Lemma \ref{Lemma: Homogeneity of minimizers} the homogeneity of $\psi_{0}$. Moreover, we obtain an $\mathcal{L}^{2}$ density estimate with respect to all non-degenerate points in Corollary \ref{Corollary: Density}. As a further application of the monotonicity formula, we prove in Proposition \ref{Proposition: minimizers are viscosity solutions} that local minimizers are viscosity solutions to the problem \eqref{Formula: governed equation}. In Section \ref{Section: regularity of the non-degenerate points}, we consider a more general free boundary problem \eqref{Formula: R(1)} and our main result in this section is Proposition \ref{Proposition: Flatness implies C1alpha}, which indicates that for viscosity solutions to the general free boundary problem \eqref{Formula: R(1)}, flat free boundary points are locally $C^{1,\gamma}$ for some $\gamma\in(0,1)$. Last, we obtain the uniqueness of blow-up limits in Lemma \ref{Lemma: Uniqueness of blow-up limits}, and we immediately obtain that free boundaries are Lipschitz functions near non-degenerate points. (Corollary \ref{Corollary: free boundary Lipschitz}). This, together with Proposition \ref{Proposition: Flatness implies C1alpha}, yields that free boundaries are $C^{1,\gamma}$ near all non-degenerate points for some $\gamma<\beta$. 

In summary, our main results in this article can be concluded as follows. 
\begin{theorem}\label{Theorem: main(1)}
	Let $\psi$ be a local minimizer of $J$ in $B_{R_{0}}(X_{0})$, then $\psi$ is a Lipschitz viscosity solution to the problem \eqref{Formula: governed equation} in $B_{R_{0}}(X_{0})$. Furthermore, for every non-degenerate free boundary point $X_{0}\in N_{\psi}$, we have
	\begin{enumerate}
	\item the limit $\mathcal{D}_{X_{0},\psi}(0^{+}):=\lim_{r\to0^{+}}\mathcal{D}_{X_{0},\psi}(r)$ exists and 
	\begin{align}\label{Formula: maintheorem(1)}
		\mathcal{D}_{X_{0},\psi}(0^{+})=-x_{0}y_{0}\frac{\omega_{2}}{2},
	\end{align}
	where $\omega_{2}=\mathcal{L}^{2}(B_{1})$.
	\item There exists a small neighborhood $B_{\rho}$ of $X_{0}$, the size of which depends on $X_{0}$ and $F_{0}$ ($L^{\infty}$-norm of $f$), so that $\partial\{\psi>0\}\cap B_{\rho}(X_{0})$ is locally the graph of a $C^{1,\gamma}$ function for some $\gamma\in(0,\beta)$.
	\end{enumerate}
\end{theorem}
\section{Preliminaries}\label{Section: preliminaries}
In this section, we list some fundamental properties of local minimizers, which are used throughout the paper.
\begin{proposition}[Properties for local minimizers]\label{Proposition: properties for local minimizers}
	Let $X_{0}\in N_{\psi}$ and let $\psi$ be a local minimizer of $J$ in $B_{R_{0}}(X_{0})$, then
	\begin{enumerate}
		\item $\psi\in C^{0,\alpha}(B_{r}(X_{0}))$ for some $\alpha\in(0,1)$ and any $r<R_{0}$.
		\item $\psi\geqslant0$ in $B_{R_{0}}(X_{0})$.
		\item $\psi$ satisfies $\operatorname{div}\left(\frac{1}{x}\nabla\psi\right)+xf(\psi)\geqslant0$ in $B_{R_{0}}(X_{0})$ in the weak sense.
		\item $\psi$ satisfies $\operatorname{div}\left(\frac{1}{x}\nabla\psi\right)+xf(\psi)=0$ in $B_{R_{0}}(X_{0})\cap\{\psi>0\}$ in the classical sense.
	\end{enumerate}
\end{proposition}
The proof of the above Proposition is rather technical, so we give it in Appendix A in order not to disrupt the flow of this paper. We only remark here that if $X_{0}\in N_{\psi}$ and if $\psi$ is a local minimizer of $J(\psi)$ in $B_{R_{0}}(X_{0})$, then the first property above tells us that $\psi$ is a continuous function in $B_{R_{0}}(X_{0})$. This means that one can choose any continuous functions to “touch” $\psi$ from below(above) at any point $X_{1}\in B_{R_{0}}(X_{0})$. Besides, the continuity of $\psi$ in $B_{r}(X_{0})$ implies that the set $\{\psi>0\}\cap B_{r}(X_{0})$ is open. In addition, the second property in the above Proposition suggests us that we only need to consider $f(z)$ defined for $z\geqslant0$. In fact, minimizers of $J$ in $B_{R_{0}}(X_{0})$ possess higher regularity in the phase $B_{R_{0}}(X_{0})\cap\{\psi>0\}$.
\begin{corollary}\label{Corollary: higher regularity in the positive phase}
    Let $X_{0}\in N_{\psi}$ and let $\psi$ be a local minimizer of $J$ in $B_{R_{0}}(X_{0})$, then $\psi$ satisfies the equation
    \begin{align*}
        \operatorname{div}\left(\frac{1}{x}\nabla\psi\right)+xf(\psi)=0\quad\text{ in }B_{R_{0}}(X_{0})\cap\{\psi>0\}
    \end{align*}
    in the classical sense. Moreover, $\psi\in C^{2,\alpha}(B_{R_{0}}(X_{0})\cap\{\psi>0\})$.
\end{corollary}
\begin{proof}
    Since $\psi$ is H\"{o}lder continuous in $B_{R_{0}}(X_{0})\cap\{\psi>0\}$, we have that $f\circ\psi$ is also a H\"{o}lder continuous function in $B_{R_{0}}(X_{0})\cap\{\psi>0\}$. On the other hand, since $X_{0}\in N_{\psi}$, it is easy to deduce that $\dfrac{1}{x}$ is uniformly bounded from below and above in $B_{R_{0}}(X_{0})$ by a constant that depends only on $x_{0}$. Thus, we are in a position to apply the standard Schauder estimates in \cite{GT1998} to conclude that $\psi\in C^{2,\alpha}(B_{R_{0}}(X_{0})\cap\{\psi>0\})$.
\end{proof}
\section{Optimal regularity and non-degeneracy of local minimizers}\label{Section: optimal regularity and non-degeneracy of local minimizers}
This section aims to derive the optimal regularity for local minimizers. Firstly, we have proved in Corollary \ref{Corollary: higher regularity in the positive phase} that if $X_{0}\in N_{\psi}$ is a non-degenerate point, then $\psi$ is smooth in $B_{R_{0}}(X_{0})\cap\{\psi>0\}$. However, we do not expect $\psi$ to retain such high regularity in the whole $B_{R_{0}}(X_{0})$. Therefore, it is important to understand the behavior of solutions as they approach the zero phase $B_{r}(X_{0})\cap\{\psi=0\}$. Here we mainly borrow the idea originated from \caps{Alt} and \caps{Caffarelli} \cite{AC1981,ACF1984}. They found that if one stays a fixed amount away from the set $\{\psi=0\}$, then the minimizer $\psi$ starts growing linearly. Based on this observation, it seems that Lipschitz continuity should be the best regularity we hope for the minimizer. We adopt such an idea to our setting and for the abuse of notation, we set
\begin{align*}
	\rho(X)=\operatorname{dist}(X,\{\psi=0\}).
\end{align*}
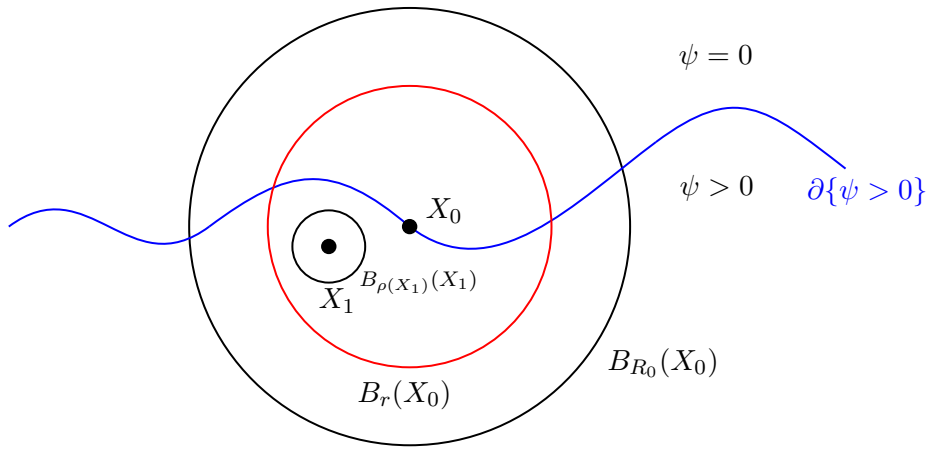
\begin{figure}[ht]
	\tikzset{every picture/.style={line width=0.75pt}} 
	\begin{tikzpicture}[x=0.75pt,y=0.75pt,yscale=-0.9,xscale=0.9]
		\draw [color={rgb, 255:red, 0; green, 0; blue, 255 }  ,draw opacity=1 ][line width=0.75]    (170.17,150.5) .. controls (210.17,120.5) and (231.17,116.83) .. (270.17,150.5) ;
		
		\draw  [line width=0.75]  (160.17,150.5) .. controls (160.17,89.75) and (209.42,40.5) .. (270.17,40.5) .. controls (330.92,40.5) and (380.17,89.75) .. (380.17,150.5) .. controls (380.17,211.25) and (330.92,260.5) .. (270.17,260.5) .. controls (209.42,260.5) and (160.17,211.25) .. (160.17,150.5) -- cycle ;
		 
		\draw [color={rgb, 255:red, 0; green, 0; blue, 255 }  ,draw opacity=1 ][line width=0.75]    (270.17,150.5) .. controls (311.67,182.67) and (353,138) .. (393,108) ;
		
		\draw [color={rgb, 255:red, 0; green, 0; blue, 255 }  ,draw opacity=1 ][line width=0.75]    (70.17,150.5) .. controls (110.17,120.5) and (130.17,180.5) .. (170.17,150.5) ;
		
		\draw [color={rgb, 255:red, 0; green, 0; blue, 255 }  ,draw opacity=1 ][line width=0.75]    (393,108) .. controls (433,78) and (447.67,90) .. (487.67,121.33) ;
		 
		\draw [line width=0.75]    (270.17,150.5) -- (270.17,150.5) ;
		\draw [shift={(270.17,150.5)}, rotate = 90] [color={rgb, 255:red, 0; green, 0; blue, 0 }  ][fill={rgb, 255:red, 0; green, 0; blue, 0 }  ][line width=0.75]      (0, 0) circle [x radius= 3.35, y radius= 3.35]   ;
		 
		\draw  [color={rgb, 255:red, 255; green, 0; blue, 0 }  ,draw opacity=1 ][line width=0.75]  (199.39,150.5) .. controls (199.39,111.41) and (231.08,79.72) .. (270.17,79.72) .. controls (309.26,79.72) and (340.95,111.41) .. (340.95,150.5) .. controls (340.95,189.59) and (309.26,221.28) .. (270.17,221.28) .. controls (231.08,221.28) and (199.39,189.59) .. (199.39,150.5) -- cycle ;
		
		\draw  [color={rgb, 255:red, 0; green, 0; blue, 0 }  ,draw opacity=1 ][line width=0.75]  (211.67,160.5) .. controls (211.67,150.47) and (219.8,142.33) .. (229.83,142.33) .. controls (239.87,142.33) and (248,150.47) .. (248,160.5) .. controls (248,170.53) and (239.87,178.67) .. (229.83,178.67) .. controls (219.8,178.67) and (211.67,170.53) .. (211.67,160.5) -- cycle ;
		
		\draw [line width=0.75]    (229.83,160.5) ;
		\draw [shift={(229.83,160.5)}, rotate = 0] [color={rgb, 255:red, 0; green, 0; blue, 0 }  ][fill={rgb, 255:red, 0; green, 0; blue, 0 }  ][line width=0.75]      (0, 0) circle [x radius= 3.35, y radius= 3.35]   ;
		
		\draw (367.33,210.07) node [anchor=north west][inner sep=0.75pt]    {$B_{R_{0}}( X_{0})$};
		\draw (276.67,134.73) node [anchor=north west][inner sep=0.75pt]    {$X_{0}$};
		\draw (242.67,226.07) node [anchor=north west][inner sep=0.75pt]    {$B_{r}( X_{0})$};
		\draw (223.33,180.07) node [anchor=north west][inner sep=0.75pt]    {$X_{1}$};
		\draw (243.67,170.73) node [anchor=north west][inner sep=0.75pt]  [font=\scriptsize]  {$B_{\rho ( X_{1})}( X_{1})$};
		\draw (403.5,122.9) node [anchor=north west][inner sep=0.75pt]    {$\psi  >0$};
		\draw (403,56.9) node [anchor=north west][inner sep=0.75pt]    {$\psi =0$};
		\draw (467,124.4) node [anchor=north west][inner sep=0.75pt]    {$\textcolor[rgb]{0,0,1}{\partial }\textcolor[rgb]{0,0,1}{\{}\textcolor[rgb]{0,0,1}{\psi  >0}\textcolor[rgb]{0,0,1}{\}}$};
	\end{tikzpicture}
    \caption{$B_{R_{0}}(X_{0})$ and $B_{\rho(X_{1})}(X_{1})$.}
    \label{Fig: balls}
\end{figure}
\begin{proposition}[Lipschitz regularity]\label{Proposition: Lipschitz regularity for local minimizers}
	Let $X_{0}=(x_{0},y_{0})\in N_{\psi}$, let $\psi$ be a local minimizer of $J$ in $B_{R_{0}}(X_{0})$, and let $B_{r}(X_{0})\subset B_{R_{0}}(X_{0})$ with $r<R_{0}$. Then 
	\begin{align}\label{Formula: Lipschitz estimate for minimizer}
		\frac{|\nabla\psi(X)|}{x}\leqslant C\sqrt{-y}\quad\text{ in }B_{r}(X_{0}).
	\end{align}
	Here the constant $C$ depends only  on $X_{0}$ and $F_{0}$ and $\left(1-\tfrac{r}{R_{0}}\right)^{-1}$.
\end{proposition}
\begin{proof}
	\textbf{Step I.} We begin by proving that
	\begin{align}\label{Formula: Lip(1)}
		\frac{\psi(X)}{\rho(X)}\leqslant C\sqrt{-y}\quad\text{ for any }X\in B_{r}(X_{0}),
	\end{align}
	where $C$ is a constant depending on $X_{0}$, $F_{0}$ and $\left(1-\frac{r}{R_{0}}\right)^{-1}$. To this end, let $\rho_{0}=R_{0}-r$ and let $X_{1}=(x_{1},y_{1})\in B_{r}(X_{0})\cap\{\psi>0\}$ with $\rho(X_{1})<\rho_{0}$. It is easy to check that  $B_{\rho(X_{1})}(X_{1})\subset B_{R_{0}}(X_{0})\cap\{\psi>0\}$ (please see Fig \ref{Fig: balls}). Assume now that
	\begin{align}\label{Formula: Lip(2)}
		\frac{\psi(X_{1})}{\rho(X_{1})}>Mx_{1},
	\end{align}
	and we aim to derive an upper bound for $M$. Consider the rescaling 
	\begin{align}\label{Formula: Lip(3)}
		\phi_{\rho}(X):=\frac{\psi(X_{1}+\rho X)}{\rho x_{1}}\quad\text{ with }\rho=\rho(X_{1}),
	\end{align}
    and one has
    \begin{align*}
    	\operatorname{div}\left(\frac{x_{1}\nabla\phi_{\rho}}{x_{1}+\rho x}\right)+\rho(x_{1}+\rho x)f(\rho x_{1}\phi_{\rho})=0\quad\text{ in }B_{1}(0).
    \end{align*}
    It follows from \eqref{Formula: Lip(2)} and \eqref{Formula: Lip(3)} that
    \begin{align*}
    	\phi_{\rho}(0)=\frac{\psi(X_{1})}{\rho x_{1}}>M.
    \end{align*}
    A straightforward computation gives
    \begin{align*}
    	|\rho(x_{1}+rx)f(\rho x_{1}\phi_{\rho})|\leqslant C\rho_{0} x_{0}|f(0)|+C\rho_{0}^{2}x_{0}^{2}\max_{s\geqslant0}|f'(s)|\phi_{\rho},
    \end{align*}
    where the constant $C$ depends only on $x_{0}$. Thanks to the Harnack inequality in \cite{T1967}, we have 
    \begin{align}\label{Formula: Lip(3(1))}
    	\phi_{\rho}(X)\geqslant cM-CF_{0}\quad\text{ in }B_{3/4}(0),
    \end{align} 
    where $c$ and $C$ are constants depending only on $X_{0}$ and $F_{0}$. On the other hand, there exists $\tilde{X}=(\tilde{x},\tilde{y})\in\partial B_{1}(0)\cap\{\phi_{\rho}=0\}$. Let $\Psi$ be a function which satisfies
    \begin{align*}
    	\begin{cases}
    		\operatorname{div}\left(\frac{x_{1}\nabla\Psi}{x_{1}+\rho x}\right)+\rho(x_{1}+\rho x)f(\rho x_{1}\Psi)=0\quad&\text{ in }B_{1}(\tilde{X}),\\
    		\Psi=\phi_{\rho}\quad&\text{ outside }B_{1}(\tilde{X}).
    	\end{cases}
    \end{align*}
    Since $\operatorname{div}\left(\frac{x_{1}\nabla\phi_{\rho}}{x_{1}+\rho x}\right)+\rho(x_{1}+\rho x)f(\rho x_{1}\phi_{\rho})\geqslant0$ in $B_{1}(\tilde{X})$. It follows from the maximum principle that
    \begin{align}\label{Formula: Lip(4)}
    	\Psi(X)\geqslant\phi_{\rho}(X)\quad\text{ in }B_{1}(\tilde{X}).
    \end{align}
    Then we have
    \begin{align*}
    	0&\leqslant J(\Psi)-J(\phi_{\rho})\\
    	&\leqslant\int_{B_{1}(\tilde{X})}\frac{-x_{1}^{2}|\nabla(\phi_{\rho}-\Psi)|^{2}}{x_{1}+\rho x}+\frac{2x_{1}\nabla\Psi}{x_{1}+\rho x}\nabla(x_{1}(\Psi-\phi_{\rho}))\:dX\\
    	&\quad+\int_{B_{1}(\tilde{X})}-2\rho(x_{1}+\rho x)F'(\rho x_{1}\Psi)(x_{1}(\Psi-\phi_{\rho}))\:dX\\
    	&\quad+\int_{B_{1}(\tilde{X})}-(x_{1}+\rho x)(y_{1}+\rho y)\left(I_{\{\Psi>0\}}-I_{\{\phi_{\rho}>0\}}\right)\:dX\\
    	&\leqslant\int_{B_{1}(\tilde{X})}\frac{-x_{1}^{2}|\nabla(\phi_{\rho}-\Psi)|^{2}}{x_{1}+\rho x}\:dX+\int_{B_{1}(\tilde{X})}-(x_{1}+\rho x)(y_{1}+\rho y)I_{\{\phi_{\rho}=0\}}\:dX,
    \end{align*}
    where we have used the equation that $\Psi$ satisfied in the second inequality. This implies that
    \begin{align}\label{Formula: Lip(4(1))}
    	\begin{alignedat}{3}
    		\int_{B_{1}(\tilde{X})}|\nabla(\Psi-\phi_{\rho})|^{2}\:dX&\leqslant C\int_{B_{1}(\tilde{X})}\frac{x_{1}|\nabla(\Psi-\phi_{\rho})|^{2}}{x_{1}+\rho x}\:dX\\
    		&\leqslant\frac{C}{x_{1}}\int_{B_{1}(\tilde{X})}(x_{1}+\rho x)(-y_{1}-\rho y)I_{\{\phi_{\rho}=0\}}\:dX\\
    		&\leqslant C\int_{B_{1}(\tilde{X})}(-y_{1}-\rho y)I_{\{\phi_{\rho}=0\}}\:dX\\
    		&\leqslant C(-y_{1})\int_{B_{1}(\tilde{X})}I_{\{\phi_{\rho}=0\}}\:dX,
    	\end{alignedat}
    \end{align}
    where $C$ depends only on $X_{0}$ and $\left(1-\tfrac{r}{R_{0}}\right)^{-1}$. It follows then from \eqref{Formula: Lip(3(1))} and \eqref{Formula: Lip(4)} that
    \begin{align*}
    	\Psi(X)\geqslant\phi_{\rho}(X)\geqslant cM-CF_{0}\quad\text{ in }B_{3/4}(0)\cap B_{1}(\tilde{X}).
    \end{align*}
    Applying Harnack's inequality for $\Psi$ in $B_{1}(\tilde{X})$ gives
    \begin{align}\label{Formula: Lip(5)}
    	\Psi(X)\geqslant C_{0}\quad\text{ in }B_{1/2}(\tilde{X}),
    \end{align}
    where $C_{0}:=cM-CF_{0}$. Define $\varphi(X)=C_{0}(e^{-\mu|X-\tilde{X}|^{2}}-e^{-\mu})$, and after a straightforward computation, one has
    \begin{align*}
    	\operatorname{div}\left(\frac{x_{1}\nabla\varphi}{x_{1}+\rho x}\right)+\rho(x_{1}+\rho  x)f(\rho x_{1}\varphi)>0\quad\text{ in }B_{1}(\tilde{X})\setminus B_{1/2}(\tilde{X}),
    \end{align*}
    provided that $\mu$ is sufficiently large. Since $\Psi\geqslant\varphi$ on $\partial(B_{1}(\tilde{X})\setminus B_{1/2}(\tilde{X}))$, it is easy to deduce from the maximum principle that
    \begin{align*}
    	\Psi(X)\geqslant\varphi(X)\geqslant C_{0}(e^{-\mu|X-\tilde{X}|^{2}}-e^{-\mu})\geqslant cC_{0}(1-|X-\tilde{X}|)\quad\text{ in }B_{1}(\tilde{X})\setminus B_{1/2}(\tilde{X}),
    \end{align*}
    which together with \eqref{Formula: Lip(5)} gives that
    \begin{align}\label{Formula: Lip(6)}
    	\Psi(X)\geqslant(cM-CF_{0})(1-|X-\tilde{X}|)\quad\text{ in }B_{1}(\tilde{X}).
    \end{align}
    With the aid of \eqref{Formula: Lip(4(1))} and \eqref{Formula: Lip(6)}, proceeding as in Lemma 3.2 in \cite{AC1981} and Lemma 2.2 in \cite{ACF1984}, we obtain
    \begin{align*}
    	(cM-CF_{0})^{2}\leqslant C(-y_{1}),
    \end{align*}
    that is,
    \begin{align*}
    	M\leqslant C\left(\sqrt{-y_{1}}+F_{0}\right)\leqslant C\sqrt{-y_{1}},
    \end{align*}
    where $C$ depends on $X_{0}$, $F_{0}$ and $\left(1-\tfrac{r}{R_{0}}\right)^{-1}$. This implies that
    \begin{align*}
    	\frac{\psi(X_{1})}{\rho(X_{1})}\leqslant Cx_{1}\sqrt{-y_{1}}.
    \end{align*}
    Now take a point $X_{2}=(x_{2},y_{2})\in B_{r}(X_{0})$ with $\rho(X_{2})>\rho_{0}$, then there must exist a point $X_{1}\in B_{\rho_{0}/2}(X_{2})$ with $\rho(X_{1})<\rho_{0}$. Thanks to Harnack's inequality for $\psi$ in $B_{\rho_{0}}(X_{2})$, one has
    \begin{align*}
    	\psi(X_{2})\leqslant C(\psi(X_{1})+\rho_{0}F_{0})\leqslant C\left(\rho(X_{1})x_{1}\sqrt{-y_{1}}\right)\leqslant C\left(\rho(X_{2})x_{2}\sqrt{-y_{2}}\right).
    \end{align*}
    Notice that for any $X\in B_{r}(X_{0})$, one can repeat this argument step by step, and after a finite steps, one has
    \begin{align*}
    	\frac{\psi(X)}{\rho(X)}\leqslant Cx\sqrt{-y},
    \end{align*}
    where $C$ depends only on $X_{0}$, $F_{0}$ and $\left(1-\tfrac{r}{R_{0}}\right)^{-1}$. Hence, we complete the proof of \eqref{Formula: Lip(1)}.
    \item \textbf{Step II.} In this step, we will finish the proof of Proposition \ref{Proposition: Lipschitz regularity for local minimizers}. For any $X_{1}\in B_{r}(X_{0})$, denote by $\rho_{0}=R_{0}-r$, and we consider the following two cases.
    
    \textbf{Case 1}. $\rho(X_{1})<\rho_{0}$. It then follows from \eqref{Formula: Lip(1)}
    \begin{align*}
    	\operatorname{div}\left(\frac{x_{1}\nabla\phi_{\rho}}{x_{1}+\rho x}\right)+\rho(x_{1}+\rho x)f(\rho x_{1}\phi_{\rho})=0\quad\text{ in }B_{1}(0),
    \end{align*}
    and
    \begin{align*}
    	0\leqslant\phi_{\rho}\leqslant\frac{C(x_{1}+\rho x)\sqrt{(-y_{1}-\rho y)}\rho(X_{1}+\rho X)}{\rho x_{1}}\leqslant C\sqrt{-y_{1}}\quad\text{ in }B_{1}(0),
    \end{align*}
    where the constant $C$ depends only on $X_{0}$, $F_{0}$ and $\left(1-\tfrac{r}{R_{0}}\right)^{-1}$. Applying Schauder estimates in \cite{GT1998}, one has
    \begin{align*}
    	|\nabla\phi_{\rho}(0)|\leqslant C\sqrt{-y_{1}}.
    \end{align*}
    Combined with the definition of $\phi_{\rho}$ in \eqref{Formula: Lip(3)}, one has
    \begin{align*}
    	|\nabla\psi(X_{1})|=x_{1}|\nabla\phi_{\rho}(0)|\leqslant Cx_{1}\sqrt{-y_{1}}.
    \end{align*}
    
    \textbf{Case 2}. $\rho(X_{1})\geqslant\rho_{0}$. Since $B_{\rho_{0}}(X_{1})\subset B_{R_{0}}(X_{0})\cap\{\psi>0\}$. Denote by $\phi_{0}:=\frac{\psi(X_{1}+\rho_{0}X)}{\rho_{0}x_{1}}$, it follows from \eqref{Formula: Lip(1)} that
    \begin{align*}
    	\operatorname{div}\left(\frac{x_{1}\nabla\phi_{0}}{x_{1}+\rho x}\right)+\rho(x_{1}+\rho x)f(\rho x_{1}\phi_{0})=0\quad\text{ in }B_{1}(0),
    \end{align*}
    and
    \begin{align*}
    	0\leqslant\phi_{0}\leqslant\frac{C(x_{1}+\rho_{0}x)\sqrt{(-y_{1}-\rho_{0}y)}\rho(X_{1}+\rho_{0}X)}{\rho_{0}x_{1}}\leqslant C\sqrt{-y_{1}-\rho_{0}y}\quad\text{ in }B_{1}(0).
    \end{align*}
    Using the Schauder estimates in \cite{GT1998}, one has
    \begin{align*}
    	|\nabla\phi_{0}(0)|\leqslant C\sqrt{-y_{1}},
    \end{align*}
    and
    \begin{align*}
    	|\nabla\psi|=x_{1}|\nabla\phi_{0}(0)|\leqslant Cx_{1}\sqrt{-y_{1}}.
    \end{align*}
\end{proof}
\begin{remark}\label{Remark: Growth estimates of minimizers}
	Let $X_{0}\in N_{\psi}$ be a non-degenerate point and let $\psi$ be a local minimizer of $J$ in $B_{R_{0}}(X_{0})$, then the estimate \eqref{Formula: Lipschitz estimate for minimizer} implies that
	\begin{align*}
		\psi^{2}(x,y)\leqslant Cx^{2}(-y)|X-X_{0}|^{2}\quad\text{ in }B_{R_{0}}(X_{0}).
	\end{align*}
    It should be noted that this estimate with respect to $\psi^{2}$ will be used in our proof of monotonicity formula.
\end{remark}
With the aid of Proposition \ref{Proposition: Lipschitz regularity for local minimizers}, we obtain the following Lemma immediately.
\begin{lemma}[Optimal linear growth]\label{Lemma: Optimal linear growth}
	Let $X_{0}=(x_{0},y_{0})\in N_{\psi}$ and let $\psi$ be a local minimizer of $J$ in $B_{R_{0}}(X_{0})$, then there exists a positive constant $C^{*}$ such that for any $B_{r}(X)\subset\subset B_{R_{0}}(X_{0})$ with center $X=(x,y)$, such that 
	\begin{align}\label{Formula: Lingro(1)}
		\frac{1}{r}\fint_{\partial B_{r}(X)}\psi\:d\mathcal{H}^{1}\geqslant C^{*}x\sqrt{-y},
	\end{align}
	implies 
	\begin{align*}
		\psi(x,y)>0\quad\text{ in }B_{r}(X).
	\end{align*}
\end{lemma}
\begin{proof}
	Assume for the sake of contradiction that there exists a point $X_{1}=(x_{1},y_{1})\in B_{R_{0}}(X_{0})\cap\{\psi=0\}$. Then
	\begin{align*}
		\psi(X_{1})=\psi(X_{1})-\psi(X_{0})\leqslant Cx_{1}\sqrt{-y_{1}}r,
	\end{align*}
	which yields a contradiction to \eqref{Formula: Lingro(1)}.
\end{proof}
We now establish a non-degeneracy lemma.
\begin{lemma}[Non-degeneracy]\label{Lemma: Non-degeneracy of minimizers}
	Let $X_{0}=(x_{0},y_{0})\in N_{\psi}$ and let $\psi$ be a local minimizer of $J$ in $B_{R_{0}}(X_{0})$, then for any $\kappa\in(0,1)$ there exists a constant $c_{\kappa}^{*}$ such that for any small $B_{r}(X)\subset B_{R_{0}}(X_{0})$ with center $X=(x,y)$ and $r\leqslant\min\{\tfrac{x}{2},\tfrac{R_{0}}{2},\tfrac{c_{\kappa}^{*}\sqrt{-y-\kappa}}{F_{0}},\tfrac{1}{2}\}$, such that 
	\begin{align}\label{Formula: ND(0)}
		\frac{1}{r}\left(\fint_{B_{r}(X)}\psi^{2}\:dX\right)^{1/2}\leqslant c_{\kappa}^{*}x\sqrt{-y-\kappa r},
	\end{align}
    implies
    \begin{align*}
    	\psi(x,y)=0\quad\text{ in }B_{\kappa r}(X).
    \end{align*}
\end{lemma}
\begin{proof}
	Let $X_{1}\in B_{R_{0}}(X_{0})$ and let us consider the function
	\begin{align*}
		\psi_{r}(X):=\frac{\psi(X_{1}+rX)}{r x_{1}},
	\end{align*}
	for some $r>0$ so that $\{X_{1}+rX\colon X\in B_{1}(0)\}\subset B_{R_{0}}(X_{0})$. Then $\psi_{r}$ satisfies
	\begin{align*}
		\operatorname{div}\left(\frac{x_{1}\nabla\psi_{r}}{x_{1}+rx}\right)+r(x_{1}+rx)f(rx_{1}\psi_{r})\geqslant0\quad\text{ in }B_{1}(0).
	\end{align*}
    Denote
    \begin{align}\label{Formula: ND(1)}
    	\varepsilon:=\sup_{B_{\sqrt{\kappa}}(0)}\psi_{r},
    \end{align}
    then it follows from Theorem 8.17 in \cite{GT1998} that
    \begin{align*}
    	\varepsilon\leqslant C\left[\left(\fint_{B_{1}(0)}\psi_{r}^{2}\:dX\right)^{1/2}+rF_{0}\right].
    \end{align*}
    It follows from \eqref{Formula: ND(0)} that
    \begin{align}\label{Formula: ND(2(1))}
    	\left(\fint_{B_{1}(0)}\psi_{r}^{2}\:dX\right)^{1/2}\leqslant c_{\kappa}^{*}(x_{1}+rx)\sqrt{-y_{1}-\kappa r}\leqslant c_{\kappa}^{*}C(x_{0})\sqrt{-y_{1}-\kappa r}.
    \end{align}
    Here $C(x_{0})$ is a constant depending only on the point $X_{0}$. Let $\phi$ be a function defined by
    \begin{align}\label{Formula: ND(3)}
    	\begin{cases}
    		\begin{alignedat}{3}
    			\operatorname{div}\left(\frac{x_{1}\nabla\phi}{x_{1}+rx}\right)+r(x_{1}+rx)f(rx_{1}\phi)&=0\quad&&\text{ in }B_{\sqrt{\kappa}}(0)\setminus B_{\kappa}(0),\\
    			\phi&=0\quad&&\text{ in }B_{\kappa}(0),\\
    			\phi&=\varepsilon\quad&&\text{ outside }B_{\sqrt{\kappa}}(0).
    		\end{alignedat}
    	\end{cases}
    \end{align}
    Set $\Phi:=\min\{\psi_{r},\phi\}$ and by the minimality of $\psi$, one has
    \begin{align*}
    	&\int_{B_{\sqrt{\kappa}}(0)}\frac{|x_{1}\nabla\psi_{r}|^{2}}{x_{1}+rx}-2(x_{1}+rx)F(rx_{1}\psi_{r})+(x_{1}+rx)(-y_{1}-ry)I_{\{\psi_{r}>0\}}\:dX\\
    	&\leqslant\int_{B_{\sqrt{\kappa}}(0)}\frac{|x_{1}\nabla\Phi|^{2}}{x_{1}+rx}-2(x_{1}+rx)F(rx_{1}\Phi)+(x_{1}+rx)(-y_{1}-ry)I_{\{\Phi>0\}}\:dX.
    \end{align*}
    Since $\phi=0$ in $B_{\kappa}(0)$, we have $\Phi=0$ in $B_{\kappa}(0)$. Therefore,
    \begin{align*}
    	&\int_{B_{\kappa}(0)}\frac{|x_{1}\nabla\psi_{r}|^{2}}{x_{1}+rx}-2(x_{1}+rx)F(r x_{1}\psi_{r})+(x_{1}+r x)(-y_{1}-r y)I_{\{\psi_{r}>0\}}\:dX\\
    	&\leqslant\int_{B_{\sqrt{\kappa}}(0)\setminus B_{\kappa}(0)}\left(\frac{|x_{1}\nabla\Phi|^{2}}{x_{1}+rx}-\frac{|x_{1}\nabla\psi_{r}|^{2}}{x_{1}+rx}\right)-2(x_{1}+rx)(F(rx_{1}\Phi)-F(rx_{1}\psi_{r}))\:dX.
    \end{align*}
    Here we were using the fact $I_{\{\Phi>0\}}-I_{\{\psi_{r}>0\}}\leqslant0$, since outside the ball $B_{\kappa}(0)$ and $\Phi=0$ whenever $\psi_{r}=0$. Since $\int_{B_{\sqrt{\kappa}}(0)\setminus B_{\kappa}(0)}\frac{|x_{1}\nabla(\Phi-\psi_{r})|^{2})}{x_{1}+\rho x}\:dX\geqslant0$, we have
    \begin{align}\label{Formula: ND(3(1))}
    	\begin{alignedat}{6}
    		&\int_{B_{\sqrt{\kappa}}(0)\setminus B_{\kappa}(0)}\left(\frac{|x_{1}\nabla\Phi|^{2}}{x_{1}+rx}-\frac{|x_{1}\nabla\psi_{r}|^{2}}{x_{1}+rx}\right)-2(x_{1}+rx)(F(rx_{1}\Phi)-F(rx_{1}\psi_{r}))\:dX\\
    		&\leqslant\int_{B_{\sqrt{\kappa}}(0)\setminus B_{\kappa}(0)}\frac{2x_{1}\nabla\Phi}{x_{1}+rx}(x_{1}\nabla(\Phi-\psi_{r}))-2r(x_{1}+rx)F'(rx_{1}\Phi)(x_{1}(\Phi-\psi_{r}))\:dX\\
    		&=\int_{B_{\sqrt{\kappa}}(0)\setminus B_{\kappa}(0)}\frac{2x_{1}\nabla\phi}{x_{1}+rx}(x_{1}\nabla(\Phi-\psi_{r}))-2r(x_{1}+rx)F'(rx_{1}\phi)(x_{1}(\Phi-\psi_{r}))\:dX\\
    		&\leqslant\int_{\partial(B_{\sqrt{\kappa}}(0)\setminus B_{\kappa}(0))}\frac{2x_{1}^{2}}{x_{1}+rx}(\Phi-\psi_{r})\pd{\phi}{\nu}\:d\mathcal{H}^{1},
    	\end{alignedat}
    \end{align}
    where $\nu$ is the unit normal vector. Recalling \eqref{Formula: ND(3)}, it is easy to deduce from the elliptic estimates that
    \begin{align*}
    	\sup_{\partial B_{\kappa}(0)}|\nabla\phi|\leqslant C(x_{0},\kappa,F_{0})(\varepsilon+rF_{0}).
    \end{align*}
    On the other hand, it follows from  $F''(z)\leqslant0$ for all $z\in\mathbb{R}$ that
    \begin{align*}
    	F(rx_{1}\psi_{r})\leqslant F'(0)rx_{1}\psi_{r}\leqslant F_{0}rx_{1}\psi_{r},
    \end{align*}
    which gives
    \begin{align}\label{Formula: ND(3(2))}
    	\begin{alignedat}{2}
    		&\int_{B_{\kappa}(0)}\frac{|x_{1}\nabla\psi_{r}|^{2}}{x_{1}+rx}-2(x_{1}+rx)F_{0}rx_{1}\psi_{r}+(x_{1}+rx)(-y_{1}-ry)I_{\{\psi_{r}>0\}}\:dX\\
    		&\leqslant\int_{B_{\kappa}(0)}\frac{|x_{1}\nabla\psi_{r}|^{2}}{x_{1}+rx}-2(x_{1}+rx)F(rx_{1}\psi_{r})+(x_{1}+rx)(-y_{1}-ry)I_{\{\psi_{r}>0\}}\:dX.
    	\end{alignedat}
    \end{align}
    Consequently, with the aid of trace theorem, it follows from \eqref{Formula: ND(3(1))} and \eqref{Formula: ND(3(2))} that 
    \begin{align}\label{Formula: ND(4)}
    	\begin{alignedat}{7}
    		&\int_{B_{\kappa}(0)}|\nabla\psi_{r}|^{2}+I_{\{\psi_{r}>0\}}(x_{1}+rx)(-y_{1}-ry)\:dX\\
    		&\leqslant C(x_{0})\int_{B_{\kappa}(0)}\frac{|x_{1}\nabla\psi_{r}|^{2}}{x_{1}+rx}-2(x_{1}+rx)F(rx_{1}\psi_{r})+(x_{1}+rx)(-y_{1}-ry)I_{\{\psi_{r}>0\}}\:dX\\
    		&\qquad+C(x_{0})\int_{B_{\kappa}(0)}F_{0}r\psi_{r}\:dX\\
    		&\leqslant C(x_{0},F_{0})\left(\int_{\partial B_{\kappa}(0)}(\Phi-\psi_{r})\pd{\phi}{\nu}\:d\mathcal{H}^{1}+\int_{B_{\kappa}(0)}\psi_{r}\:dX\right)\\
    		&\leqslant C(x_{0},\kappa,F_{0})(\varepsilon+rF_{0})\left(\int_{\partial B_{\kappa}(0)}\psi_{r}\:d\mathcal{H}^{1}+\int_{B_{\kappa}(0)}\psi_{r}\:dX\right)\\
    		&\leqslant C(x_{0},\kappa,F_{0})\varepsilon_{0}\left(\int_{B_{\kappa}(0)}\psi_{r}I_{\{\psi_{r}>0\}}\:dX+C\int_{B_{\kappa}(0)}|\nabla\psi_{r}|I_{\{\psi_{r}>0\}}\:dX\right),
    	\end{alignedat}
    \end{align}
    where $\varepsilon_{0}:=\varepsilon+\rho F_{0}$. Therefore,
    \begin{align*}
    	\begin{alignedat}{4}
    		&\int_{B_{\kappa}(0)}\psi_{r}I_{\{\psi_{r}>0\}}\:dX+C\int_{B_{\kappa}(0)}|\nabla\psi_{r}|I_{\{\psi_{r}>0\}}\:dX\\
    		&\leqslant\frac{\varepsilon_{0}}{C(x_{0})(-y_{1}-\kappa r)}\int_{B_{\kappa}(0)}(x_{1}+rx)(-y_{1}-ry)I_{\{\psi_{r}>0\}}\:dX\\
    		&\qquad+\frac{C}{C(x_{0})\sqrt{-y_{1}-\kappa r}}\int_{B_{\kappa}(0)}|\nabla\psi_{r}|^{2}+(x_{1}+rx)(-y_{1}-ry)I_{\{\psi_{r}>0\}}\:dX\\
    		&\leqslant\bar{C}\int_{B_{\kappa}(0)}|\nabla\psi_{r}|^{2}+(x_{1}+rx)(-y_{1}-ry)I_{\{\psi_{r}>0\}}\:dX,
    	\end{alignedat}
    \end{align*}
    where
    \begin{align*}
    	\bar{C}=C(x_{0})\frac{1}{\sqrt{(-y_{1}-\kappa r)}}\left(\frac{\varepsilon_{0}}{\sqrt{(-y_{1}-\kappa r)}}+1\right).
    \end{align*}
    This together with \eqref{Formula: ND(4)} gives
    \begin{align}\label{Formula: ND(6)}
    	\begin{alignedat}{2}
    		&\int_{B_{\kappa}(0)}|\nabla\psi_{r}|^{2}+(x_{1}+rx)(-y_{1}-ry)I_{\{\psi_{r}>0\}}\:dX\\
    		&\leqslant\frac{C(x_{0},\kappa,F_{0})\varepsilon_{0}}{\sqrt{(-y_{1}-\kappa r)}}\left(\frac{\varepsilon_{0}}{\sqrt{-y_{1}-\kappa r}}+1\right)\int_{B_{\kappa}(0)}|\nabla\psi_{r}|^{2}+(x_{1}+rx)(-y_{1}-ry)I_{\{\psi_{r}>0\}}\:dX.
    	\end{alignedat}
    \end{align}
    With the aid of \eqref{Formula: ND(1)} and \eqref{Formula: ND(2(1))}, one has
    \begin{align*}
    	&\frac{\varepsilon_{0}}{\sqrt{(-y_{1}-\kappa r)}}=\frac{\varepsilon+rF_{0}}{\sqrt{(-y_{1}-\kappa r)}}\\
    	&\leqslant C(x_{0},\kappa,F_{0})\left[\frac{1}{\sqrt{(-y_{1}-\kappa r)}}\left(\fint_{B_{1}(0)}\psi^{2}\:dX\right)^{1/2}+\frac{r F_{0}}{\sqrt{(-y_{1}-\kappa r)}}\right]\\
    	&\leqslant C(x_{0},\kappa,F_{0})c_{\kappa}^{*}.
    \end{align*}
    This implies that $\tfrac{\varepsilon_{0}}{\sqrt{-y_{1}-\kappa r}}$ is small enough, provided that $c_{\kappa}^{*}$ is small enough. This together with \eqref{Formula: ND(6)} implies that
    \begin{align*}
    	\int_{B_{\kappa}(0)}|\nabla\psi_{r}|^{2}+(x_{1}+rx)(-y_{1}-ry)I_{\{\psi_{r}>0\}}\:dX=0,
    \end{align*}
    for sufficiently small $c_{\kappa}^{*}$. This implies directly that $\psi_{r}=0$ in $B_{\kappa}(0)$, if $c_{\kappa}^{*}$ is small enough.
\end{proof}
\begin{remark}
	 Despite our considerations in this paper being focused on free boundaries that are away from the axis of symmetry, we cannot just ignore the axis of symmetry and deal the problem like the two-dimensional case (such as \cite{CD2020}). Actually, the axis still makes for a drastic difference to the classical two-dimensional problems. For instance, the non-degeneracy property we have established in the previous lemma is entirely different from that in \cite{CD2020}, based on a simple fact that we do not have the mean value property (which holds only for the Laplacian) for a general linear elliptic operator.
\end{remark}
\section{Homogeneity of blow-up limits}\label{Section: homogeneity of minimizers}
This section aims to the local behavior of local minimizers near the non-degenerate points. Let $X_{0}=(x_{0},y_{0})\in N_{\psi}$ be a non-degenerate point, let $\psi$ be a local minimizer of $J$ in $B_{R_{0}}(X_{0})$, and we consider the following blow-up sequence, 
\begin{align}\label{Formula: BA(1)}
	\psi_{n}(X):=\frac{\psi(X_{0}+\rho_{n}X)}{\rho_{n}},
\end{align}
where $\rho_{n}\to0$ as $n\to\infty$. It follows from Proposition \eqref{Proposition: Lipschitz regularity for local minimizers} that
\begin{align*}
	|\nabla\psi_{n}(X)|=|\nabla\psi(X_{0}+\rho_{n}X)|\leqslant C(x_{0}+\rho_{n}x)\sqrt{-(y_{0}+\rho_{n}y)}\leqslant C(X_{0}).
\end{align*}
Since $\psi_{n}(0)=0$, we have by a diagonal argument and the theorem of Ascoli-Arzela that there exists a non-relabel subsequence, still denoted by $\psi_{n}$, and a function
\begin{align}\label{Formula: BA(2)}
	\psi_{0}\in W_{\mathrm{loc}}^{1,\infty}(\mathbb{R}^{2}),
\end{align}
such that 
\begin{align}\label{Formula: BA(2(1))}
	\psi_{n}\to\psi_{0}\quad\text{ in }\quad C_{\mathrm{loc}}^{\alpha}(\mathbb{R}^{2})\quad\text{ for some }\quad\alpha\in(0,1),
\end{align}
and
\begin{align*}
	\nabla\psi_{n}\stackrel{*}{\rightharpoonup}\nabla\psi_{0}\quad\text{ in }\quad L_{\mathrm{loc}}^{\infty}(\mathbb{R}^{2}).
\end{align*}
\begin{figure}[ht]

\tikzset{every picture/.style={line width=0.75pt}} 

\begin{tikzpicture}[x=0.75pt,y=0.75pt,yscale=-1,xscale=1]

\draw  [line width=0.75]  (146.79,153.22) .. controls (146.79,138.72) and (159.18,126.97) .. (174.46,126.97) .. controls (189.74,126.97) and (202.13,138.72) .. (202.13,153.22) .. controls (202.13,167.72) and (189.74,179.47) .. (174.46,179.47) .. controls (159.18,179.47) and (146.79,167.72) .. (146.79,153.22) -- cycle ;
\draw [line width=0.75]    (161.02,194.87) .. controls (200.68,187.01) and (151.83,127.08) .. (182.75,110.93) ;
\draw [line width=0.75]    (174.46,153.22) ;
\draw [shift={(174.46,153.22)}, rotate = 0] [color={rgb, 255:red, 0; green, 0; blue, 0 }  ][fill={rgb, 255:red, 0; green, 0; blue, 0 }  ][line width=0.75]      (0, 0) circle [x radius= 3.35, y radius= 3.35]   ;
\draw [line width=0.75]    (155.41,172.13) -- (174.46,153.22) ;
\draw  [line width=0.75]  (309.58,152.05) .. controls (309.58,129.1) and (329.19,110.5) .. (353.39,110.5) .. controls (377.58,110.5) and (397.19,129.1) .. (397.19,152.05) .. controls (397.19,175) and (377.58,193.6) .. (353.39,193.6) .. controls (329.19,193.6) and (309.58,175) .. (309.58,152.05) -- cycle ;
\draw [line width=0.75]    (338.71,208.26) .. controls (372.32,173.41) and (334.23,140.25) .. (371.42,101.15) ;
\draw [line width=0.75]    (353.39,152.05) ;
\draw [shift={(353.39,152.05)}, rotate = 0] [color={rgb, 255:red, 0; green, 0; blue, 0 }  ][fill={rgb, 255:red, 0; green, 0; blue, 0 }  ][line width=0.75]      (0, 0) circle [x radius= 3.35, y radius= 3.35]   ;
\draw [line width=0.75]    (323.02,182.34) -- (353.39,152.05) ;
\draw [line width=0.75]    (223.53,151.94) -- (288.76,152.15) ;
\draw [shift={(290.76,152.16)}, rotate = 180.18] [color={rgb, 255:red, 0; green, 0; blue, 0 }  ][line width=0.75]    (10.93,-3.29) .. controls (6.95,-1.4) and (3.31,-0.3) .. (0,0) .. controls (3.31,0.3) and (6.95,1.4) .. (10.93,3.29)   ;

\draw (178.75,143.77) node [anchor=north west][inner sep=0.75pt]    {$X_{0}$};
\draw (157.7,149.44) node [anchor=north west][inner sep=0.75pt]    {$r$};
\draw (188.35,97.16) node [anchor=north west][inner sep=0.75pt]    {$\psi $};
\draw (374.01,85.46) node [anchor=north west][inner sep=0.75pt]    {$\psi _{X_{0}}$};
\draw (324.19,154.55) node [anchor=north west][inner sep=0.75pt]    {$1$};
\draw (360.94,140.52) node [anchor=north west][inner sep=0.75pt]    {$0$};
\draw (121.37,174.95) node [anchor=north west][inner sep=0.75pt]    {$\psi  >0$};
\draw (179.18,176.65) node [anchor=north west][inner sep=0.75pt]    {$\psi =0$};
\draw (278.29,187.9) node [anchor=north west][inner sep=0.75pt]    {$\psi _{X_{0}}  >0$};
\draw (369.26,188.75) node [anchor=north west][inner sep=0.75pt]    {$\psi _{X_{0}} =0$};
\draw (226.9,130.77) node [anchor=north west][inner sep=0.75pt]   [align=left] {Blow-up};
\draw (131.14,207.53) node [anchor=north west][inner sep=0.75pt]    {$B_{r}( X_{0})$};
\draw (338.48,209.87) node [anchor=north west][inner sep=0.75pt]    {$B_{1}( 0)$};

\end{tikzpicture}
    \caption{The method of blow-up analysis.}
    \label{Fig: blow-up}
\end{figure}
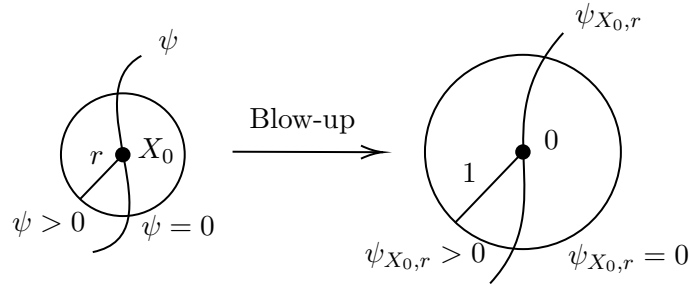
\begin{remark}
	The technique we employed above is called a \emph{blow-up analysis}. The idea is to employ a sequence of balls with vanishing radii to zoom in on any point so that the local behavior of such points can be grasped (please see Fig \ref{Fig: blow-up}). \caps{Alt and Caffarelli} first brought this approach \cite{AC1981} to scrutinize local minimizers.
\end{remark}
\begin{lemma}[Structure of the blow-up limits]\label{Lemma: Structure of the blow-up limits}
	Let $\psi_{n}$ and $\psi_{0}$ be defined as in \eqref{Formula: BA(1)} and \eqref{Formula: BA(2)}, then 
	\begin{enumerate}
		\item $\partial\{\psi_{n}>0\}\to\partial\{\psi_{0}>0\}$ locally in the Hausdorff distance of $\mathbb{R}^{2}$,
		\item $I_{\{\psi_{n}>0\}}\to I_{\{\psi_{0}>0\}}$ in $L_{\mathrm{loc}}^{1}(\mathbb{R}^{2})$,
		\item Every blow-up limit $\psi_{0}$ of $\psi_{n}$ is a global minimizer of
		\begin{align*}
			\mathcal{J}(\phi):=\int_{B_{1}(0)}\frac{|\nabla\phi|^{2}}{x_{0}}-x_{0}y_{0}I_{\{\phi>0\}}\:dX,\quad X=(x,y)\in\mathbb{R}^{2},
		\end{align*}
	    defined for all $\phi\in W^{1,2}(B_{1}(0))$ so that $\phi=\psi_{0}$ on $\partial B_{1}(0)$.
	    \item $\psi_{n}$ converges to $\psi_{0}$ strongly in $W_{\mathrm{loc}}^{1,2}(\mathbb{R}^{2})$.
	\end{enumerate}
\end{lemma}
\begin{proof}
	(1). Let $Z_{0}=(\eta_{0},\zeta_{0})\in\mathbb{R}^{2}$ and let $B_{r}(Z_{0})$ be such that $B_{r}(Z_{0})\cap\partial\{\psi_{0}>0\}=\varnothing$. Then there are only two possibilities: either $\psi_{0}>0$ in $B_{r}(Z_{0})$ or $\psi_{0}\equiv0$ in $B_{r}(Z_{0})$. In the first case, since $\psi_{n}\to\psi_{0}$ in $C^{\alpha}(\bar{B}_{r/2}(Z_{0}))$ and $\inf_{\bar{B}_{r/2}(Z_{0})}\psi_{0}>0$, we have that $\psi_{n}>0$ in $\bar{B}_{r/2}(Z_{0})$ for all $n$ sufficiently large, consequently, $B_{r/2}(Z_{0})\cap\partial\{\psi_{n}>0\}=\varnothing$ for all $n$ sufficiently large. In the second case $\psi_{0}\equiv0$ in $B_{r}(Z_{0})$, since $\psi_{n}\to\psi_{0}$ in $C^{\alpha}(\bar{B}_{3r/4}(Z_{0}))$, we have that for all $\delta>0$ there exists $N_{\delta}\in\mathbb{N}$ such that
	\begin{align*}
		|\psi_{n}(Z)|\leqslant\delta,
	\end{align*}
    for all $Z\in\bar{B}_{3r/4}(Z_{0})$ and for all $n\geqslant N_{\delta}$. Hence, if $\delta=\delta(r)$ is sufficiently small, we have by rescaling that
    \begin{align*}
    	\frac{1}{\tfrac{3}{4}r\rho_{n}}\left(\fint_{B_{\tfrac{3}{4}r\rho_{n}}(X_{0}+\rho_{n}Z_{0})}\psi^{2}\:dX\right)^{1/2}\leqslant\frac{4\delta}{3r}\leqslant c_{\tfrac{1}{2}}^{*}x_{0}\sqrt{-y_{0}-\tfrac{1}{2}r},
    \end{align*}
    for all $n\geqslant N_{\delta}$. Here $c^{*}_{\tfrac{1}{2}}$ is the constant in Lemma \ref{Lemma: Non-degeneracy of minimizers}. It then follows from Lemma \ref{Lemma: Non-degeneracy of minimizers} that $\psi_{n}\equiv0$ in $B_{\tfrac{3}{8}r}(Z_{0})$. Hence, also in this case,
    \begin{align*}
    	B_{\tfrac{1}{4}r}(Z_{0})\cap\partial\{\psi_{n}>0\}=\varnothing,
    \end{align*}
    for all $n$ sufficiently large.
    
    On the other hand, if $B_{r}(Z_{0})\subset\mathbb{R}^{2}$ such that $B_{r}(Z_{0})\cap\partial\{\psi_{n}>0\}=\varnothing$, then up to a further subsequence, we may assume that either $\psi_{n}>0$ in $B_{r}(Z_{0})$ for all $n\in\mathbb{N}$ or $\psi_{n}\equiv0$ in $B_{r}(Z_{0})$. In the first case, we have that 
    \begin{align*}
    	\operatorname{div}\left(\frac{\nabla\psi_{n}}{x_{0}+\rho_{n}x}\right)+\rho_{n}(x_{0}+\rho_{n}x)f(\rho_{n}\psi_{n})=0\quad\text{ in }B_{r}(Z_{0}),
    \end{align*}
    which implies that
    \begin{align}\label{Formula: BA(4)}
    	\Delta\psi_{0}=0\quad\text{ in }B_{r/2}(Z_{0}),\qquad\psi_{0}\geqslant0\quad\text{ in }B_{r/2}(Z_{0}).
    \end{align} 
    The strong maximum principle gives that $\psi_{0}\equiv0$ or $\psi_{0}>0$ in $B_{r/2}(Z_{0})$, which implies that
    \begin{align}\label{Formula: BA(5)}
    	B_{r/4}(Z_{0})\cap\partial\{\psi_{0}>0\}=\varnothing.
    \end{align}
    Finally, if $\psi_{n}\equiv0$ in $B_{r}(Z_{0})$ for all $n\in\mathbb{N}$, then $\psi_{0}\equiv0$ in $B_{r}(Z_{0})$ and \eqref{Formula: BA(5)} continuous to hold. 
    
    A standard compactness argument shows that $\partial\{\psi_{n}>0\}\to\partial\{\psi_{0}>0\}$ locally in the Hausdorff distance in $\mathbb{R}^{2}$.
    
    (2). Let $Z_{0}=(\eta_{0},\zeta_{0})\in\partial\{\psi_{0}>0\}$. Since $\partial\{\psi_{n}>0\}\to\partial\{\psi_{0}>0\}$ locally in the Hausdorff distance, there exists a sequence $Z_{n}\in\partial\{\psi_{n}>0\}$, $Z_{n}=(\eta_{n},\zeta_{n})\in\mathbb{R}^{2}$ so that $Z_{n}\to Z_{0}$ and $X_{0}+\rho_{n}Z_{n}\in \partial\{\psi>0\}$. The non-degeneracy (Proposition \ref{Proposition: Lipschitz regularity for local minimizers}) gives
    \begin{align*}
    	\frac{1}{s}\left(\fint_{B_{s}(X_{0}+\rho_{n}Z_{n})}\psi^{2}dX\right)^{1/2}\geqslant c_{\tfrac{1}{2}}^{*}(x_{0}+\rho_{n}\eta_{n})\sqrt{-y_{0}-\rho_{n}\zeta_{n}-\frac{1}{2}s},
    \end{align*}
    for all $s>0$ sufficiently small. Taking $s=r\rho_{n}$ and changing variables, we get
    \begin{align*}
    	\frac{1}{r}\left(\fint_{B_{r}(Z_{n})}\psi_{n}^{2}dX\right)^{1/2}\geqslant c_{\tfrac{1}{2}}^{*}(x_{0}+\rho_{n}\eta_{n})\sqrt{-y_{0}-\rho_{n}\zeta_{n}-\frac{1}{2}r}.
    \end{align*}
	Hence
    \begin{align}\label{Formula: BA(7)}
    	\frac{1}{r}\left(\fint_{B_{r}(Z_{0})}\psi_{0}^{2}\:dX\right)^{1/2}\geqslant c_{\tfrac{1}{2}}^{*}x_{0}\sqrt{-y_{0}-\frac{1}{2}r}.
    \end{align}
    We now claim that 
    \begin{align*}
    	\mathcal{L}^{2}(\partial\{\psi_{0}>0\})=0.
    \end{align*}
    To this end, it suffices to prove that there exists a positive constant $c\in(0,1)$ so that
    \begin{align}\label{Formula: BA(9)}
    	c<\frac{\mathcal{L}^{2}(B_{r}(Z_{0})\cap\{\psi_{0}>0\})}{\mathcal{L}^{2}(B_{r}(Z_{0}))}<1-c.
    \end{align}
    By Lemma \ref{Lemma: Non-degeneracy of minimizers}, there exists $Y\in B_{\tfrac{1}{2}r}(Z_{0})$ such that $\psi_{n}(Y)\geqslant cr>0$. Thus $\psi_{n}>0$ in $B_{\tfrac{1}{2}\kappa r}(Y)$ for a small $\kappa>0$. Thus, $\psi_{0}>0$ in $B_{\tfrac{1}{4}\kappa r}(Y)$. This gives the lower estimate in \eqref{Formula: BA(9)} follows.
    
    We now prove the upper bound in \eqref{Formula: BA(9)}. Assume the contrary, then $\mathcal{L}^{2}(B_{r}(Z_{0})\cap\{\psi_{0}=0\})=0$ and this implies that there must exist a non-relabel subsequence so that $\mathcal{L}^{2}(B_{r}(Z_{n})\cap\{\psi_{n}=0\})\to0$ as $n\to\infty$. Let $\Phi_{n}$ be a sequence of Lipschitz functions satisfying
    \begin{align*}
    	\begin{cases}
    		\begin{alignedat}{2}
    			\operatorname{div}\left(\frac{\nabla\Phi_{n}}{x_{0}+\rho_{n}x}\right)+\rho_{n}(x_{0}+\rho_{n}x)f(\rho_{n}\Phi_{n})&=0\quad&&\text{ in }B_{r}(Z_{n}),\\
    			\Phi_{n}&=\psi_{n}\quad&&\text{ on }\partial B_{r}(Z_{n}).
    		\end{alignedat}
    	\end{cases}
    \end{align*}
    By the minimality of $\psi_{n}$ and a similar calculation as \eqref{Formula: Lip(4(1))}, one has
    \begin{align}\label{Formula: BA(10)}
    	\int_{B_{r}(Z_{n})}|\nabla(\Phi_{n}-\psi_{n})|^{2}\:dX\leqslant C(X_{0})\int_{B_{r}(Z_{n})}I_{\{\psi_{n}=0\}}\:dX\to0,\quad\text{ as }n\to\infty
    \end{align}
    Recalling \eqref{Formula: BA(2(1))}, we have $\psi_{n}\to\psi_{0}$ and $\Phi_{n}\to\Phi_{0}$ uniformly in $\bar{B}_{\tfrac{1}{4}r}(Z_{0})$. On the other hand, $\Delta\Phi_{0}=0$ in $B_{r/2}(Z_{0})$. The estimate \eqref{Formula: BA(10)} implies that $\psi_{0}=\Phi_{0}+c$. Therefore, $\Delta\psi_{0}=0$ in $B_{\tfrac{1}{4}r}(Z_{0})$. Since $\psi_{0}\geqslant0$ and $\psi_{0}(0)=0$ in $B_{\tfrac{1}{4}r}(Z_{0})$, it follows from the strong maximal principle that $\psi_{0}\equiv0$ in $B_{\tfrac{1}{4}r}(Z_{0})$, which yields a contradiction to \eqref{Formula: BA(7)}.
    
    (3). Consider a function $\eta\in C_{0}^{1}(B_{1};[0,1])$ and for every $\phi\in W^{1,2}(B_{1}(0))$ with $\phi=\psi_{0}$ on $\partial B_{1}(0)$. Define
    \begin{align*}
    	\phi_{n}(X):=\phi(X)+(1-\eta(X))(\psi_{n}(X)-\psi_{0}(X)).
    \end{align*}
    Observe that if $X\in\partial B_{1}(0)$,
    \begin{align*}
    	\phi_{n}(X)=\psi_{n}(X)=\frac{\psi(X_{0}+\rho_{n}X)}{\rho_{n}}.
    \end{align*}
    Define a new function
    \begin{align*}
    	\Psi_{n}(X):=\begin{cases}
    		\rho_{n}\phi_{n}\left(\frac{X-X_{0}}{\rho_{n}}\right)\quad&\text{ if }X\in B_{\rho_{n}}(X_{0}),\\
    		\psi(X)\quad&\text{ if }X\notin B_{\rho_{n}}(X_{0}).
    	\end{cases}
    \end{align*}
    By the optimality of $\psi$, one has
    \begin{align*}
    	\int_{B_{\rho_{n}}(X_{0})}\frac{|\nabla\psi|^{2}}{x}-2xF(\psi)-xyI_{\{\psi>0\}}\:dX\leqslant\int_{B_{\rho_{n}}(X_{0})}\frac{|\nabla\Psi_{n}|^{2}}{x}-2xF(\Psi_{n})-xyI_{\{\Psi_{n}>0\}}\:dX.
    \end{align*}
    Changing the variables gives
    \begin{align*}
    	&\int_{B_{1}(0)}\left[\frac{|\nabla\psi(X_{0}+\rho_{n}X)}{x_{0}+\rho_{n}x}-2(x_{0}+\rho_{n}x)F(\psi(X_{0}+\rho_{n}X))\right]\rho_{n}^{2}\:dX\\
    	&\qquad+\int_{B_{1}(0)}(x_{0}+\rho_{n}x)(-y_{0}-\rho_{n}y)I_{\{\psi>0\}}(X_{0}+\rho_{n}X)\rho_{n}^{2}\:dX\\
    	&\leqslant\int_{B_{1}(0)}\left[\frac{|\nabla\Psi_{n}(X_{0}+\rho_{n}X)|^{2}}{x_{0}+\rho_{n}x}-2(x_{0}+\rho_{n}x)F(\Psi_{n}(X_{0}+\rho_{n}X))\right]\rho_{n}^{2}\:dX\\
    	&\qquad+\int_{B_{1}(0)}(x_{0}+\rho_{n}x)(-y_{0}-\rho_{n}y)I_{\{\Psi_{n}>0\}}(X_{0}+\rho_{n}X)\rho_{n}^{2}\:dX\\
    	&\leqslant\int_{B_{1}(0)}\left[\frac{|\nabla\phi_{n}(X)|^{2}}{x_{0}+\rho_{n}x}-2(x_{0}+\rho_{n}x)F(\rho_{n}\phi_{n}(X))\right]\rho_{n}^{2}\:dX\\
    	&\qquad+\int_{B_{1}(0)}(x_{0}+\rho_{n}x)(-y_{0}-\rho_{n}y)I_{\{\phi_{n}>0\}}\rho_{n}^{2}\:dX.
    \end{align*}
    It follows from the definition of $\psi_{n}$ that
    \begin{align*}
    	&\int_{B_{1}(0)}\frac{|\nabla\psi_{n}(X)|^{2}}{x_{0}+\rho_{n}x}-2(x_{0}+\rho_{n}x)F(\rho_{n}\psi_{n}(X))+(x_{0}+\rho_{n}x)(-y_{0}-\rho_{n}y)I_{\{\psi_{n}>0\}}\:dX\\
    	&\leqslant\int_{B_{1}(0)}\frac{|\nabla\phi_{n}(X)|^{2}}{x_{0}+\rho_{n}x}-2(x_{0}+\rho_{n}x)F(\rho_{n}\phi_{n}(X))(x_{0}+\rho_{n}x)(-y_{0}-\rho_{n}y)I_{\{\phi_{n}>0\}}\:dX.
    \end{align*}
    A straightforward computation gives
    \begin{align*}
    	\nabla\phi_{n}=\nabla\phi+(1-\eta)(\nabla\psi_{n}-\nabla\psi_{0})-\nabla\eta(\psi_{n}-\psi_{0}).
    \end{align*}
    Therefore,
    \begin{align*}
    	|\nabla\phi_{n}|^{2}-|\nabla\psi_{n}|^{2}&=|\nabla\phi|^{2}+|\nabla\eta|^{2}|\psi_{n}-\psi_{0}|^{2}-2(\psi_{n}-\psi_{0})\nabla\eta\cdot\nabla\phi+(1-\eta)^{2}|\nabla\psi_{n}-\nabla\psi_{0}|^{2}\\
    	&\quad+2(1-\eta)(\nabla\psi_{n}-\nabla\psi_{0})\cdot(\nabla\phi-\nabla\eta(\psi_{n}-\psi_{0}))-|\nabla\psi_{n}|^{2}\\
    	&\leqslant|\nabla\phi|^{2}+|\nabla\eta|^{2}|\psi_{n}-\psi_{0}|^{2}-2(\psi_{n}-\psi_{0})\nabla\eta\nabla\phi-2\nabla\psi_{n}\nabla\psi_{0}+|\nabla\psi_{0}|^{2}\\
    	&\quad+2(1-\eta)(\nabla\psi_{n}-\nabla\psi_{0})(\nabla v-\nabla\eta(\psi_{n}-\psi_{0})).
    \end{align*}
    Consequently,
    \begin{align*}
    	&\int_{B_{1}(0)}\frac{2\nabla\psi_{n}\cdot\nabla\psi_{0}-|\nabla\psi_{0}|^{2}}{x_{0}+\rho_{n}x}+(x_{0}+\rho_{n}x)(-y_{0}-\rho_{n}y)I_{\{\psi_{n}>0\}}\:dX\\
    	&\leqslant\int_{B_{1}(0)}\frac{|\nabla\phi|^{2}}{x_{0}+\rho_{n}x}+\left(I_{\{\phi>0\}}+I_{\{\eta<1\}}\right)(x_{0}+\rho_{n}x)(-y_{0}-\rho_{n}y)\:dX\\
    	&\quad+\int_{B_{1}(0)}-2(x_{0}+\rho_{n}x)\left(F(\rho_{n}\phi_{n}(X))-F(\rho_{n}\psi_{n}(X))\right)\:dX\\
    	&\quad+\int_{B_{1}(0)}\frac{1}{x_{0}+\rho_{n}x}\left[(\nabla\psi_{n}-\nabla\psi_{0})\cdot(2(1-\eta)\nabla\phi-2(1-\eta)(\psi_{n}-\psi_{0})\nabla\eta)\right]\:dX\\
    	&\quad+\int_{B_{1}(0)}\frac{1}{x_{0}+\rho_{n}x}\left[(\psi_{n}-\psi_{0})\nabla\eta(2\nabla\phi+(\psi_{n}-\psi_{0})\nabla\eta)\right]\:dX.
    \end{align*}
    Since $\psi_{n}\rightharpoonup\psi_{0}$ in $L^{2}(\partial B_{1}(0))$ and $F(0)=0$, we have by letting $n\to\infty$ that
    \begin{align*}
    	\int_{B_{1}(0)}\frac{|\nabla\psi_{0}|^{2}}{x_{0}}-x_{0}y_{0}I_{\{\psi_{0}>0\}}\:dX\leqslant\int_{B_{1}(0)}\frac{|\nabla\phi|^{2}}{x_{0}}+\left(I_{\{\phi>0\}}+I_{\{\eta<1\}}\right)(-x_{0}y_{0})\:dX.
    \end{align*}
    The desired result follows by choosing a sequence of functions $\eta_{n}\nearrow1$.
    
    (4). Since $\psi_{n}$ converges weakly to $\psi_{0}$ in $W_{\mathrm{loc}}^{1,2}(\mathbb{R}^{2})$, it suffices to show that $\nabla\psi_{n}$ converges to $\nabla\psi_{0}$ strongly in $L_{\mathrm{loc}}^{2}(\mathbb{R}^{2})$. Equivalently, it suffices to prove that for any $R>0$,
    \begin{align*}
    	\limsup_{n\to\infty}\int_{\bar{B}_{R}}\frac{1}{x_{0}+\rho_{n}x}|\nabla\psi_{n}|^{2}\eta\:dX=\int_{\bar{B}_{R}}\frac{1}{x_{0}}|\nabla\psi_{0}|^{2}\eta\:dX\quad\text{ for all }\eta\in C_{0}^{1}(B_{R}(0)),
    \end{align*}
    where $\bar{B}_{R}:=B_{R}(0)$. Recalling \eqref{Formula: BA(4)}, one has that $\psi_{0}$ is harmonic in $\bar{B}_{R}(0)\cap\{\psi_{0}>0\}$. With the aid of the fact that $\psi_{n}$ converges to $\psi_{0}$ locally and uniformly in $\{\psi_{0}>0\}$, one has
    \begin{align*}
    	&\int_{\bar{B}_{R}}\frac{1}{x_{0}+\rho_{n}x}|\nabla\psi_{n}|^{2}\eta\:dX\\
    	&=-\int_{\bar{B}_{R}}\psi_{n}\operatorname{div}\left(\frac{\nabla\psi_{n}}{x_{0}+\rho_{n}x}\right)\eta\:dX-\int_{\bar{B}_{R}}\psi_{n}\frac{1}{x_{0}+\rho_{n}x}\nabla\psi_{n}\nabla\eta\:dX\\
    	&\to-\int_{\bar{B}_{R}}\psi_{0}\Delta\psi_{0}\eta\:dX-\int_{\bar{B}_{R}}\psi_{0}\frac{1}{x_{0}}\nabla\psi_{0}\nabla\eta\:dX\\
    	&=\int_{\bar{B}_{R}}\frac{1}{x_{0}}|\nabla\psi_{0}|^{2}\eta\:dX,
    \end{align*}
    as wanted.
\end{proof}
As a direct corollary, we list in the following Corollary some properties of blow-up limits, the one says that $\psi_{0}$ satisfies the Laplace equation and the other one says that $\psi_{0}$ is a two-dimensional half-plane solution, provided that $\psi_{0}$ is a one-homogeneous function.
\begin{corollary}[Properties of the blow-up limits]\label{Corollary: Properties of blow-up limits}
	Let $\psi_{n}$ and $\psi_{0}$ be given as in \eqref{Formula: BA(1)} and \eqref{Formula: BA(2)}, then
	\begin{enumerate}
		\item $\psi_{0}$ satisfies $\Delta\psi_{0}=0$ in $\{\psi_{0}>0\}$ and $\dfrac{1}{x_{0}^{2}}|\nabla\psi_{0}|^{2}=\sqrt{-y_{0}}$ on $\partial\{\psi_{0}>0\}$.
	    \item If $\psi_{0}$ is a one-homogeneous function in the sense that $\psi_{0}(rX)=r\psi_{0}(X)$ for every $X\in\mathbb{R}^{2}$ and every $r>0$, then there is a unit vector $\nu\in\mathbb{R}^{2}$ such that
	    \begin{align*}
	    	\psi_{0}(X)=x_{0}\sqrt{-y_{0}}(X\cdot\nu)^{+}\quad\text{ for every }X=(x,y)\in\mathbb{R}^{2}.
	    \end{align*}
	\end{enumerate}
\end{corollary}
\begin{proof}
	(1). We have already proved in the previous Lemma that $\psi_{0}$ is harmonic in $\{\psi_{0}>0\}$ (recall \eqref{Formula: BA(4)}). We now derive the boundary condition that $\psi_{0}$ is achieved on $\partial\{\psi_{0}>0\}$. Since $\psi_{0}$ is a global minimizer of the functional
	\begin{align*}
		\mathcal{J}(\phi):=\int_{B_{1}(0)}\frac{|\nabla\phi|^{2}}{x_{0}}-x_{0}y_{0}I_{\{\phi>0\}}\:dX,
	\end{align*}
    it is easy to deduce from the expression of the first domain variation formula (for instance page 1269 in \cite{VW2014}) that
    \begin{align*}
    	0&=\int_{B_{1}(0)}\left(\frac{1}{x_{0}}|\nabla\psi_{0}|^{2}-x_{0}y_{0}I_{\{\psi_{0}>0\}}\right)\operatorname{div}\eta-\frac{2}{x_{0}}\nabla\psi_{0} D\eta\nabla\psi_{0}\:dX,
    \end{align*}
    where $\eta=(\eta_{1},\eta_{2})\in C_{0}^{1}(\mathbb{R}^{2};\mathbb{R}^{2})$. Integration by parts,
    \begin{align*}
    	0&=\int_{B_{1}(0)\cap\{\psi_{0}>0\}}\operatorname{div}\left(\left(\frac{1}{x_{0}}|\nabla\psi_{0}|^{2}-x_{0}y_{0}\right)\eta-\frac{2}{x_{0}}\nabla\psi_{0}\nabla\psi_{0}\cdot\eta\right)\:dX\\
    	&=\lim_{\varepsilon\to0}\int_{\partial\{\psi_{0}>\varepsilon\}}\left(\frac{1}{x_{0}}|\nabla\psi_{0}|^{2}-x_{0}y_{0}\right)(\eta\cdot\nu)-\frac{2}{x_{0}}|\nabla\psi_{0}|^{2}(\eta\cdot\nu)\:d\mathcal{H}^{1}\\
    	&=\lim_{\varepsilon\to0}\int_{\partial\{\psi_{0}>\varepsilon\}}\left(-x_{0}y_{0}-\frac{1}{x_{0}}|\nabla\psi_{0}|^{2}\right)\:d\mathcal{H}^{1},
    \end{align*}
    as desired.
    	
    (2). We consider the Laplacian in polar coordinates, and we write $\psi_{0}(r,\theta)=rc(\theta)$ where $c(\theta)=\psi_{0}|_{\partial B_{1}}$ denotes the trace, then
    \begin{align}\label{Formula: BA(12)}
    	\begin{alignedat}{2}
    		0=\Delta\psi_{0}(r,\theta)&=\pd[2]{\psi_{0}}{r}+\frac{1}{r}\pd{\psi_{0}}{r}+\frac{1}{r^{2}}\partial_{\theta\theta}\psi_{0}(r,\theta)\\
    		&=\frac{1}{r}(c(\theta)+c''(\theta)).
    	\end{alignedat}
    \end{align}
    Since $\psi_{0}$ is continuous, we have that $\{c>0\}\subset\mathbb{S}^{1}$ is open and so it is a countable union of disjoint arcs. Recalling $\psi_{0}(0)=0$ and \eqref{Formula: BA(9)}, it is easy to deduce that  $\{c>0\}\neq\mathbb{S}^{1}$. It follows from \eqref{Formula: BA(12)} that on each arc $I\subset\{c>0\}$, the trace $c$ is a solution of the ODE
    \begin{align*}
    	c''(\theta)+c(\theta)=0\quad\text{ in }I,\qquad c>0\quad\text{ in }I,\qquad c=0\quad\text{ on }\partial I.
    \end{align*} 
    Thus, up to a translation $I=(0,\pi)$ and $c(\theta)$ is a multiple of $\sin\theta$ on $I$. Thus, $\{c>0\}$ is a union of disjoint arcs, each one of length $\pi$. Therefore, these arcs can be at most two. Now, by the fact that $0\in\partial\{\psi_{0}>0\}$, we have from \eqref{Formula: BA(9)} that $\mathcal{L}^{2}(\{\psi_{0}>0\}\cap B_{1})<\mathcal{L}^{2}(B_{1})=\pi$. The assumption that $\psi_{0}$ is a one-homogeneous function gives that $\mathcal{H}^{1}(\{c>0\}\cap\partial B_{1})<2\pi$. Thus, $\{c>0\}$ is an arc of length $\pi$ and $\psi_{0}$ is of the form $a(x\cdot\nu)^{+}$ for some constant $a>0$. Since $\psi_{0}$ minimizes $\mathcal{J}(\phi)$, one has $a=x_{0}\sqrt{-y_{0}}$, and this concludes the proof.
\end{proof}
At this stage, we proved that every blow-up limit $\psi_{0}$ of $\psi$ is a non-negative and non-constantly vanishing harmonic function defined in $\mathbb{R}^{2}$. Next, we use the Weiss-type boundary adjusted energy defined in \eqref{Formula: D(3)} to prove that all the blow-up limits of $\psi$ are \emph{one-homogeneous functions}. The most significant result in this section is the following monotonicity formula.
\begin{lemma}[Monotonicity formula]\label{Lemma: monotonicity formula}
	Let $X_{0}\in N_{\psi}$ and let $\psi$ be a local minimizer of $J$ in $B_{R_{0}}(X_{0})$, and let $\mathcal{D}_{X_{0},\psi}(r)$ be defined as in \eqref{Formula: D(3)}, then for a.e. $r\in(0,\tfrac{R_{0}}{2})$,
	\begin{align}\label{Formula: Monotonicity formula}
		\begin{alignedat}{2}
			\frac{d\mathcal{D}_{X_{0},\psi}(r)}{dr}&=2r^{-2}\int_{\partial B_{r}(X_{0})}\frac{1}{x}\left(\nabla\psi\cdot\nu-\frac{\psi}{r}\right)^{2}\:d\mathcal{H}^{1}\\
			&+r^{-4}\int_{\partial B_{r}(X_{0})}\frac{x-x_{0}}{x^{2}}\psi^{2}\:d\mathcal{H}^{1}-r^{-3}J_{0}(r)-r^{-3}K_{1}(r),
		\end{alignedat}
	\end{align}
	where $\nu$ is the unit vector on $\partial B_{r}(X_{0})$ and
	\begin{align*}
		J_{0}(r):=\int_{B_{r}(X_{0})}\left[\frac{x-x_{0}}{x^{2}}|\nabla\psi|^{2}+((x-x_{0})y+(y-y_{0})x)I_{\{\psi>0\}}\right]\:dX,
	\end{align*}
	and
	\begin{align*}
		K_{1}(r)&:=\int_{B_{r}(X_{0})}(2F(\psi)(x-x_{0})+4xF(\psi))\:dX-r\int_{\partial B_{r}(X_{0})}(2xF(\psi)-x\psi f(\psi))\:d\mathcal{H}^{1}.
	\end{align*}
\end{lemma}
\begin{proof}
	\textbf{Step I}. In this step, we prove that if $\psi$ is a local minimizer of $J$ in $B_{R_{0}}(X_{0})$, then for any vector field $\eta(X)=(\eta_{1}(X),\eta_{2}(X))\in C_{0}^{1}(B_{R_{0}}(X_{0});\mathbb{R}^{2})$, one has
	\begin{align}\label{Formula: BA(13)}
		\begin{alignedat}{3}
			0&=\int_{B_{R}(X_{0})}\left(\frac{1}{x}|\nabla\psi|^{2}-2xF(\psi)-xyI_{\{\psi>0\}}\right)\operatorname{div}\eta\:dX\\
			&\quad-\int_{B_{R}(X_{0})}\frac{2}{x}\nabla\psi D\eta\nabla\psi\:dX\\
			&\quad+\int_{B_{R}(X_{0})}\left(-\frac{1}{x^{2}}|\nabla\psi|^{2}-2F(\psi)-yI_{\{\psi>0\}}\right)\eta_{1}\:dX\\
			&\quad-\int_{B_{R}(X_{0})}xI_{\{\psi>0\}}\eta_{2}\:dX.
		\end{alignedat}
	\end{align}
    To this end, we define $\tau_{\varepsilon}(X)=X+\varepsilon\eta(X):=\tilde{X}$ for any $\varepsilon>0$ and set $\psi_{\varepsilon}(\tilde{X}):=(\psi\circ\tau_{\varepsilon}^{-1})(\tilde{X})$. Denote
    \begin{align*}
    	J(\psi_{\varepsilon}(\tilde{X}))=\int_{B_{R}(X_{0})}\frac{1}{\tilde{x}}|\nabla\psi_{\varepsilon}(\tilde{X})|^{2}-2\tilde{x}F(\psi_{\varepsilon}(\tilde{X}))-\tilde{x}\tilde{y}I_{\{\psi_{\varepsilon}>0\}}\:d\tilde{X}:=J_{1}+J_{2}+J_{3}.
    \end{align*}
    It follows that
    \begin{align*}
    	J_{1}(\psi_{\varepsilon})&=\int_{B_{R}(X_{0})}\frac{1}{\tilde{x}}|\nabla\psi_{\varepsilon}(\tilde{X})|^{2}\:d\tilde{X}=\int_{B_{R}(X_{0})}\frac{1}{x+\varepsilon\eta_{1}}|\nabla\psi(\tau_{\varepsilon}(X))|^{2}J(\tau_{\varepsilon})\:dX\\
    	&=\int_{B_{R}(X_{0})}\frac{1}{x+\varepsilon\eta_{1}}|\nabla(\psi_{\varepsilon}\circ\tau_{\varepsilon})[D\tau_{\varepsilon}]^{-1}(X)|^{2}(1+\varepsilon\operatorname{div}\eta+o(\varepsilon))\:dX\\
    	&=\int_{B_{R}(X_{0})}\frac{1}{x+\varepsilon\eta_{1}}|\nabla\psi(\operatorname{Id}-\varepsilon D\eta)|^{2}(1+\varepsilon\operatorname{div}\eta+o(\varepsilon))\:dX\\
    	&=\int_{B_{R}(X_{0})}\frac{1}{x+\varepsilon\eta_{1}}|\nabla\psi|^{2}\:dX+\varepsilon\int_{B_{R}(X_{0})}\frac{1}{x+\varepsilon\eta_{1}}|\nabla\psi|^{2}\operatorname{div}\eta\:dX\\
    	&\quad-2\varepsilon\int_{B_{R}(X_{0})}\frac{1}{x+\varepsilon\eta_{1}}\nabla\psi D\eta(\nabla\psi)\:dX+o(\varepsilon),
    \end{align*}
    and
    \begin{align*}
    	J_{2}(\psi_{\varepsilon})&=\int_{B_{r}(X_{0})}-2\tilde{x}F(\psi_{\varepsilon}(\tilde{X}))\:d\tilde{X}\\
    	&=\int_{B_{r}(X_{0})}-2(x+\varepsilon\eta_{1})F(\psi_{\varepsilon}(\tau_{\varepsilon}(X)))J(\psi_{\varepsilon}(X))\:dX\\
    	&=\int_{B_{r}(X_{0})}-2(x+\varepsilon\eta_{1})F(\psi_{\varepsilon}(\tau_{\varepsilon}(X)))(1+\varepsilon\operatorname{div}\eta+o(\varepsilon))\:dX\\
    	&=\int_{B_{r}(X_{0})}-2xF(\psi(X))\:dX-\varepsilon\int_{B_{r}(X_{0})}2xF(\psi(X))\operatorname{div}\eta\:dX\\
    	&\quad-\varepsilon\int_{B_{r}(X_{0})}2\eta_{1}F(\psi(X))\:dX+o(\varepsilon),
    \end{align*}
    and
    \begin{align*}
    	J_{3}(\psi_{\varepsilon})&=\int_{B_{R}(X_{0})}-\tilde{x}\tilde{y}I_{\{\psi_{\varepsilon}>0\}}\:d\tilde{X}\\
    	&=\int_{B_{R}(X_{0})}-(x+\varepsilon\eta_{1})(y+\varepsilon\eta_{2})I_{\{\psi_{\varepsilon}>0\}}(\tau_{\varepsilon}(X))J(\tau_{\varepsilon}(X))\:dX\\
    	&=\int_{B_{R}(X_{0})}-xyI_{\{\psi>0\}}\:dX-\varepsilon\int_{B_{R}(X_{0})}(\eta_{1}y+\eta_{2}x)I_{\{\psi>0\}}\:dX\\
    	&\quad-\varepsilon\int_{B_{R}(X_{0})}xyI_{\{\psi>0\}}\operatorname{div}\eta\:dX+o(\varepsilon).
    \end{align*}
    Since $\psi$ is a local minimizer of $J$ in $B_{R_{0}}(X_{0})$, one has
    \begin{align*}
    	0=\lim_{\varepsilon\to0}\frac{J(\psi_{\varepsilon};B_{R}(X_{0}))-J(\psi;B_{R}(X_{0}))}{\varepsilon},
    \end{align*}
    this together with $J_{1}$, $J_{2}$ and $J_{3}$ gives \eqref{Formula: BA(13)}.
    
    \textbf{Step II}. In this step, we will prove the following Poho\v{z}aev-type identity:
    \begin{align}\label{Formula: BA(14)}
    	\begin{alignedat}{4}
    		&\int_{B_{r}(X_{0})}2xyI_{\{\psi>0\}}\:dX-r\int_{\partial B_{r}(X_{0})}xyI_{\{\psi>0\}}\:d\mathcal{H}^{1}\\
    		&=-r\int_{\partial B_{r}(X_{0})}\frac{1}{x}|\nabla\psi|^{2}\:d\mathcal{H}^{1}-\int_{B_{r}(X_{0})}\frac{x-x_{0}}{x^{2}}|\nabla\psi|^{2}\:dX+r\int_{\partial B_{r}(X_{0})}\frac{2}{x}(\nabla\psi\cdot\nu)^{2}\:d\mathcal{H}^{1}\\
    		&\quad-\int_{B_{r}(X_{0})}(2F(\psi)(x-x_{0})+4xF(\psi))\:dX+r\int_{\partial B_{r}(X_{0})}2xF(\psi)\:d\mathcal{H}^{1}\\
    		&\quad-\int_{B_{r}(X_{0})}(y(x-x_{0})+x(y-y_{0}))I_{\{\psi>0\}}\:dX.
    	\end{alignedat}
    \end{align}
    For any $\varepsilon>0$ define 
    \begin{align*}
    	\zeta_{\varepsilon}(t):=\max\{0,\min\{1,\tfrac{r-t}{\varepsilon}\}\}\quad r\in(0,R_{0}/2]
    \end{align*}
    and set $\eta_{\varepsilon}(X):=\zeta_{\varepsilon}(|X-X_{0}|)(X-X_{0})$. It follows that
    \begin{align*}
    	\operatorname{div}\eta_{\varepsilon}(X)&=2\zeta_{\varepsilon}(|X-X_{0}|)+\zeta'_{\varepsilon}\frac{(X-X_{0})}{|X-X_{0}|}\cdot(X-X_{0})\\
    	&=2\zeta_{\varepsilon}(|X-X_{0}|)+\zeta'_{\varepsilon}(|X-X_{0}|)|X-X_{0}|
    \end{align*}
    and
    \begin{align*}
    	D\eta_{\varepsilon}=\zeta_{\varepsilon}(|X-X_{0}|)\delta_{ij}+\zeta'_{\varepsilon}(|X-X_{0}|)\frac{(X_{i}-X_{0i})\cdot(X_{j}-X_{0j})}{|X-X_{0}|},
    \end{align*}
    for $i$, $j=1$, $2$. Here $X_{i}$ denotes the $i$th component of the point $X$. For instance, if $X=(x,y)$, then $X_{1}=x$ and $X_{2}=y$. Introducing $\eta_{\varepsilon}$ as test functions into the identity \eqref{Formula: BA(13)} and passing $\varepsilon\to0$ yields the following results
    \begin{align*}
    	&\int_{B_{R}(X_{0})}\left(\frac{1}{x}|\nabla\psi|^{2}-2xF(\psi)-xyI_{\{\psi>0\}}\:dX\right)\operatorname{div}\eta_{\varepsilon}\:dX\\
    	&=\int_{B_{R}(X_{0})}\left(\frac{1}{x}|\nabla\psi|^{2}-2xF(\psi)-xyI_{\{\psi>0\}}\right)(2\zeta_{\varepsilon}(|X-X_{0}|)+\zeta_{\varepsilon}'(|X-X_{0}|)|X-X_{0}|)\:dX\\
    	&\to2\int_{B_{r}(X_{0})}\left(\frac{1}{x}|\nabla\psi|^{2}-2xF(\psi)-xyI_{\{\psi>0\}}\right)\:dX\\
    	&\quad-r\int_{\partial B_{r}(X_{0})}\left(\frac{1}{x}|\nabla\psi|^{2}-2xF(\psi)-xyI_{\{\psi>0\}}\right)\:d\mathcal{H}^{1},
    \end{align*}
    and
    \begin{align*}
    	&\int_{B_{R}(X_{0})}\frac{2}{x}\nabla\psi D\eta_{\varepsilon}\nabla\psi\:dX=\int_{B_{R}(X_{0})}\frac{2}{x}\left[\partial_{i}\psi\left(\zeta_{\varepsilon}(|X-X_{0}|)\delta_{ij}\right)\partial_{j}\psi\right]\:dX\\
    	&\quad+\int_{B_{R}(X_{0})}\frac{2}{x}\left[\partial_{i}\psi\left(\zeta_{\varepsilon}'(|X-X_{0}|)\frac{(X_{i}-X_{0i})\cdot(X_{j}-X_{0j})}{|X-X_{0}|}\right)\partial_{j}\psi\right]\:dX\\
    	&=\int_{B_{R}(X_{0})}\frac{2}{x}(\partial_{1}\psi)^{2}\zeta_{\varepsilon}(|X-X_{0}|)\:dX+\int_{B_{R}(X_{0})}\frac{2}{x}(\partial_{2}\psi)^{2}\zeta_{\varepsilon}(|X-X_{0}|)\:dX\\
    	&\quad+\int_{B_{R}(X_{0})}\frac{2}{x}(\partial_{1}\psi)^{2}\zeta_{\varepsilon}'(|X-X_{0}|)\frac{(x-x_{0})^{2}}{|X-X_{0}|}\:dX\\
    	&\quad+\int_{B_{R}(X_{0})}\frac{2}{x}(\partial_{2}\psi)^{2}\zeta_{\varepsilon}'(|X-X_{0}|)\frac{(y-y_{0})^{2}}{|X-X_{0}|}\:dX\\
    	&\quad+\int_{B_{R}(X_{0})}\frac{2}{x}\left[(\partial_{1}\psi)(\partial_{2}\psi)2\zeta_{\varepsilon}'(|X-X_{0}|)\frac{(x-x_{0})(y-y_{0})}{|X-X_{0}|}\right]\:dX\\
    	&=\int_{B_{R}(X_{0})}-\frac{2}{x}|\nabla\psi|^{2}\zeta_{\varepsilon}(|X-X_{0}|)\:dX\\
    	&\quad+\int_{B_{R}(X_{0})}-\frac{2}{x}\left[\pd{\psi}{x}\cdot\frac{(x-x_{0})}{|X-X_{0}|}+\pd{\psi}{y}\cdot\frac{(y-y_{0})}{|X-X_{0}|}\right]^{2}|X-X_{0}|\zeta_{\varepsilon}'(|X-X_{0}|)\:dX\\
    	&\to\int_{B_{r}(X_{0})}-\frac{2}{x}|\nabla\psi|^{2}\:dX+r\int_{\partial B_{r}(X_{0})}\frac{2}{x}(\nabla\psi\cdot\nu)^{2}\:d\mathcal{H}^{1},
    \end{align*}
    and
    \begin{align*}
    	&\int_{B_{R}(X_{0})}\left(-\frac{1}{x^{2}}|\nabla\psi|^{2}-2F(\psi)-yI_{\{\psi>0\}}\right)\eta_{1}\:dX\\
    	&=\int_{B_{R}(X_{0})}\left(-\frac{1}{x^{2}}|\nabla\psi|^{2}-2F(\psi)-yI_{\{\psi>0\}}\right)\zeta_{\varepsilon}(|X-X_{0}|)(x-x_{0})\:dX\\
    	&\to\int_{B_{r}(X_{0})}-\frac{x-x_{0}}{x^{2}}|\nabla\psi|^{2}\:dX+\int_{B_{r}(X_{0})}-2(x-x_{0})F(\psi)\:dX+\int_{B_{r}(X_{0})}-y(x-x_{0})I_{\{\psi>0\}}\:dX,
    \end{align*}
    and
    \begin{align*}
    	&\int_{B_{R}(X_{0})}-xI_{\{\psi>0\}}\eta_{2}\:dX=\int_{B_{R}(X_{0})}-xI_{\{\psi>0\}}\zeta_{\varepsilon}(|X-X_{0}|)(y-y_{0})\:dX\\
    	&\to\int_{B_{r}(X_{0})}-x(y-y_{0})I_{\{\psi>0\}}\:dX,\qquad\text{ as }\varepsilon\to 0.
    \end{align*}
    Combining all the calculations above gives \eqref{Formula: BA(14)}.
    
    \textbf{Step III}. In this step, we finish the proof of this Lemma. Recalling the definition of $\mathcal{D}_{X_{0},\psi}$, a straightforward computation gives
    \begin{align*}
    	\frac{d(r^{-2}\mathcal{D}_{1,X_{0},\psi}(r))}{dr}&=-2r^{-3}\int_{B_{r}(X_{0})}\left(\frac{1}{x}|\nabla\psi|^{2}-xyI_{\{\psi>0\}}-x\psi f(\psi)\right)\:dX\\
    	&\quad+r^{-2}\int_{\partial B_{r}(X_{0})}\left(\frac{1}{x}|\nabla\psi|^{2}-xyI_{\{\psi>0\}}-x\psi f(\psi)\right)\:d\mathcal{H}^{1}\\
    	&=-2r^{-3}\int_{B_{r}(X_{0})}\left(\frac{1}{x}|\nabla\psi|^{2}-x\psi f(\psi)\right)\:dX\\
    	&\quad+r^{-3}\left(\int_{B_{r}(X_{0})}2xyI_{\{\psi>0\}}\:dX-r\int_{\partial B_{r}(X_{0})}xyI_{\{\psi>0\}}\:d\mathcal{H}^{1}\right)\\
    	&\quad+r^{-2}\int_{\partial B_{r}(X_{0})}\left(\frac{1}{x}|\nabla\psi|^{2}-x\psi f(\psi)\right)\:d\mathcal{H}^{1}.
    \end{align*} 
    Inserting \eqref{Formula: BA(14)} into the fourth line gives
    \begin{align}\label{Formula: BA(15)}
    	\begin{alignedat}{5}
    		&\frac{d(r^{-2}\mathcal{D}_{1,X_{0},\psi}(r))}{dr}\\
    		&=2r^{-2}\int_{\partial B_{r}(X_{0})}\frac{1}{x}(\nabla\psi\cdot\nu)^{2}\:d\mathcal{H}^{1}-2r^{-3}\int_{B_{r}(X_{0})}\left(\frac{1}{x}|\nabla\psi|^{2}-x\psi f(\psi)\right)\:dX\\
    		&\quad-r^{-3}\int_{B_{r}(X_{0})}\left[\frac{x-x_{0}}{x^{2}}|\nabla\psi|^{2}+(y(x-x_{0})+x(y-y_{0}))I_{\{\psi>0\}}\right]\:dX\\
    		&\quad-r^{-3}\int_{B_{r}(X_{0})}(2F(\psi)(x-x_{0})+4xF(\psi))\:dX\\
    		&\quad-r^{-3}\int_{\partial B_{r}(X_{0})}-r(2xF(\psi)-x\psi f(\psi))\:d\mathcal{H}^{1}.
    	\end{alignedat}
    \end{align}
    Since $\operatorname{div}\left(\frac{1}{x}\nabla\psi\right)+xf(\psi)=0$ in $B_{R_{0}}(X_{0})\cap\{\psi>0\}$, it is easy to deduce by an integration by parts that
    \begin{align}\label{Formula: BA(16)}
    	\int_{B_{r}(X_{0})}\left(\frac{1}{x}|\nabla\psi|^{2}-x\psi f(\psi)\right)\:dX=\int_{\partial B_{r}(X_{0})}\frac{1}{x}\psi\nabla\psi\cdot\nu\:d\mathcal{H}^{1}.
    \end{align} 
    On the other hand, a direct calculation gives
    \begin{align*}
    	&\frac{d}{dr}\left(r^{-3}\int_{\partial B_{r}(X_{0})}\frac{1}{x}\psi^{2}\:d\mathcal{H}^{1}\right)=\frac{d}{dr}\left(r^{-2}\int_{\partial B_{1}}\frac{1}{x_{0}+rx}\psi^{2}(x_{0}+rx)\:d\mathcal{H}^{1}\right)\\
    	&=-2r^{-3}\int_{\partial B_{1}}\frac{1}{x_{0}+rx}\psi^{2}(x_{0}+rx)\:d\mathcal{H}^{1}-r^{-2}\int_{\partial B_{1}}\frac{x}{(x_{0}+rx)^{2}}\psi^{2}(x_{0}+rx)\:d\mathcal{H}^{1}\\
    	&\quad+r^{-2}\int_{\partial B_{1}}\frac{2}{x}\psi(x_{0}+rx)\nabla\psi(x_{0}+rx)\cdot x\:d\mathcal{H}^{1}\\
    	&=-2r^{-4}\int_{\partial B_{r}(X_{0})}\frac{1}{x}\psi^{2}\:d\mathcal{H}^{1}-r^{-4}\int_{\partial B_{r}(X_{0})}\frac{x-x_{0}}{x^{2}}\psi^{2}\:d\mathcal{H}^{1}+2r^{-3}\int_{\partial B_{r}(X_{0})}\frac{1}{x}\psi\nabla\psi\cdot\nu\:d\mathcal{H}^{1}.
    \end{align*}
    Introducing \eqref{Formula: BA(16)} into \eqref{Formula: BA(15)}, and we have 
    \begin{align*}
    	&\frac{d\mathcal{D}_{X_{0},\psi}(r)}{dr}\\
    	&=2r^{-2}\int_{\partial B_{r}(X_{0})}\frac{1}{x}(\nabla\psi\cdot\nu)^{2}-2\frac{1}{x}\nabla\psi\cdot\nu\frac{\psi}{r}+\frac{1}{x}\left(\frac{\psi}{r}\right)^{2}\:d\mathcal{H}^{1}\\
    	&\quad+r^{-4}\int_{\partial B_{r}(X_{0})}\frac{x-x_{0}}{x^{2}}\psi^{2}\:d\mathcal{H}^{1}-r^{-3}J_{0}(r)-r^{-3}K_{1}(r)\\
    	&=2r^{-2}\int_{\partial B_{r}(X_{0})}\frac{1}{x}\left(\nabla\psi\cdot\nu-\frac{\psi}{r}\right)^{2}\:d\mathcal{H}^{1}\\
    	&\quad+r^{-4}\int_{\partial B_{r}(X_{0})}\frac{x-x_{0}}{x^{2}}\psi^{2}\:d\mathcal{H}^{1}-r^{-3}J_{0}(r)-r^{-3}K_{1}(r).
    \end{align*}
    This concludes the proof.
\end{proof}
We are now in a position to prove that for every $X_{0}\in N_{\psi}$ and every local minimizer of $J$, the blow-up limit of $\psi$ is a one-homogeneous function. This result highly depends on the fact that the limit $\lim_{r\to0^{+}}\mathcal{D}_{X_{0},\psi}(r)$ exists.
\begin{lemma}[Homogeneity of minimizers]\label{Lemma: Homogeneity of minimizers}
	Let $X_{0}\in N_{\psi}$ and let $\psi$ be a local minimizer of $J$ in $B_{R_{0}}(X_{0})$, then the following does hold:
	\begin{enumerate}
		\item The limit $\mathcal{D}_{X_{0},\psi}(0^{+})=\lim_{r\to0^{+}}\mathcal{D}_{X_{0},\psi}(r)$ exists and is finite.
		\item Every blow-up limit $\psi_{0}$ of $\psi_{n}$ is a one-homogeneous function for each $X\in\mathbb{R}^{2}$.
		\item The weighted energy $\mathcal{D}_{X_{0},\psi}(0^{+})$ satisfies
		\begin{align*}
			\mathcal{D}_{X_{0},\psi}(0^{+})=-x_{0}y_{0}\int_{B_{1}(0)}I_{\{\psi_{0}>0\}}\:dX.
		\end{align*}
	\end{enumerate}
\end{lemma}
\begin{proof}
	(1). In view of Lemma \ref{Lemma: monotonicity formula}, it suffices to prove that the terms $r^{-4}\int_{\partial B_{r}(X_{0})}\frac{x-x_{0}}{x^{2}}\psi^{2}\:d\mathcal{H}^{1}$, $r^{-3}J_{0}(r)$ and $r^{-3}K_{1}(r)$ are integrable on $(r,r_{0})$ with $r_{0}\leqslant\tfrac{R_{0}}{2}$. Observe first that
	\begin{align}\label{Formula: BA(17)}
		F(\psi)=\int_{0}^{\psi}f(s)\:ds\leqslant\max_{s\geqslant0}|f(s)|\psi\leqslant C\psi.
	\end{align}
    It follows from Proposition \ref{Proposition: Lipschitz regularity for local minimizers} and Remark \ref{Remark: Growth estimates of minimizers} that
    \begin{align}\label{Formula: BA(18)}
    	\frac{|\nabla\psi|^{2}}{x^{2}}\leqslant C\quad\text{ in }\quad B_{r}(X_{0}),
    \end{align}
    and
    \begin{align}\label{Formula: BA(19)}
    	\frac{\psi^{2}}{x^{2}}\leqslant Cr^{2}\quad\text{ in }\quad B_{r}(X_{0}),
    \end{align}
    where $C$ in \eqref{Formula: BA(18)} and \eqref{Formula: BA(19)} are constants depending only on $X_{0}$. With the aid of \eqref{Formula: BA(17)}-\eqref{Formula: BA(19)}, one has the following estimates
    \begin{align*}
    	r^{-3}J_{0}(r)&=r^{-3}\int_{B_{r}(X_{0})}\left[\frac{x-x_{0}}{x^{2}}|\nabla\psi|^{2}+((x-x_{0})y+(y-y_{0})x)I_{\{\psi>0\}}\right]\:dX\\
    	&\leqslant r^{-3}\int_{B_{r}(X_{0})}Cr\:dX\leqslant C,
    \end{align*}
    and
    \begin{align*}
    	r^{-4}\int_{\partial B_{r}(X_{0})}\frac{x-x_{0}}{x^{2}}\psi^{2}\:d\mathcal{H}^{1}&\leqslant r^{-4}\int_{\partial B_{r}(X_{0})}Cr^{3}\:d\mathcal{H}^{1}\leqslant C,
    \end{align*}
    and
    \begin{align*}
    	r^{-3}K_{1}(r)&=r^{-4}\int_{B_{r}(X_{0})}[2F(\psi)(x-x_{0})+4xF(\psi)]\:dX-r^{-2}\int_{\partial B_{r}(X_{0})}(2xF(\psi)-x\psi f(\psi))\:d\mathcal{H}^{1}\\
    	&\leqslant r^{-4}\int_{B_{r}(X_{0})}(C|x-x_{0}|\psi+C|x|\psi)\:dX+r^{-2}\int_{\partial B_{r}(X_{0})}C|x|\psi\:d\mathcal{H}^{1}\\
    	&\leqslant r^{-4}\int_{B_{r}(X_{0})}C(x_{0}+r)r^{2}\:dX+r^{-2}\int_{\partial B_{r}(X_{0})}C(x_{0}+r)r^{2}\:d\mathcal{H}^{1}\\
    	&\leqslant C(x_{0}+r)+C(x_{0}+r)r.
    \end{align*}
    The above arguments indicate that the term
    \begin{align*}
    	r^{-4}\int_{\partial B_{r}(X_{0})}\frac{x-x_{0}}{x^{2}}\psi^{2}\:d\mathcal{H}^{1}-r^{3}J_{0}(r)-r^{-3}K_{1}(r)
    \end{align*}
    is integrable on $(0,r_{0})$, as wanted.
    
    (2). Integrating the equality \eqref{Formula: Monotonicity formula} with respect to $r$ on $(\rho_{n}r_{1},\rho_{n}r_{2})$ for $0<r_{1}<r_{2}<\tfrac{R_{0}}{2}$ and $\rho_{n}\to0$ as $n\to\infty$ gives
    \begin{align*}
    	\int_{\rho_{n}r_{1}}^{\rho_{n}r_{2}}\frac{d\mathcal{D}_{X_{0},\psi}}{dr}\:dr&=\int_{\rho_{n}r_{1}}^{\rho_{n}r_{2}}2r^{-2}\int_{\partial B_{r}(X_{0})}\frac{1}{x}\left(\nabla\psi\cdot\nu-\frac{\psi}{r}\right)^{2}\:d\mathcal{H}^{1}dr\\
    	&\quad+\int_{\rho_{n}r_{1}}^{\rho_{n}r_{1}}r^{-4}\int_{\partial B_{r}(X_{0})}\frac{x-x_{0}}{x^{2}}\psi^{2}\:d\mathcal{H}^{1}-\int_{\rho_{n}r_{1}}^{\rho_{n}r_{2}}r^{-3}J_{0}(r)\:dr-\int_{\rho_{n}r_{1}}^{\rho_{n}r_{2}}r^{-3}K_{1}(r)\:dr
    \end{align*}
    A straightforward computation gives
    \begin{align*}
    	&\int_{\rho_{n}r_{1}}^{\rho_{n}r_{2}}\frac{d\mathcal{D}_{X_{0},\psi}}{dr}\:dr\\
    	&=\int_{\rho_{n}r_{1}}^{\rho_{n}r_{2}}2|X-X_{0}|^{-2}\int_{\partial B_{r}(X_{0})}\frac{1}{x}\left(\nabla\psi\cdot\nu-\frac{\psi}{|X-X_{0}|}\right)^{2}\:d\mathcal{H}^{1}dr\\
    	&=\int_{B_{\rho_{n}r_{2}}(X_{0})\setminus B_{\rho_{n}r_{1}}(X_{0})}2|X-X_{0}|^{-2}\frac{1}{x}\left(\nabla\psi\cdot\nu-\frac{\psi}{|X-X_{0}|}\right)^{2}\:dX\\
    	&=\int_{B_{r_{2}}(0)\setminus B_{r_{1}}(0)}2|\rho_{n}X'|^{-2}\frac{1}{x_{0}+\rho_{n}x'}\left(\nabla\psi(X_{0}+\rho_{n}X')\cdot\frac{X'}{|X|}-\frac{\psi(X_{0}+\rho_{n}X)}{\rho_{n}|X'|}\right)^{2}\rho_{n}^{2}\:dX'\\
    	&=\int_{B_{r_{2}}(0)\setminus B_{r_{1}}(0)}2|X'|^{-4}\frac{1}{x_{0}+\rho_{n}x'}\left(\nabla\psi_{n}\cdot X'-\psi_{n}\right)^{2}\:dX'.
    \end{align*}
    Thus,
    \begin{align*}
    	&2\int_{B_{r_{2}}(0)\setminus B_{r_{1}}(0)}|X'|^{-4}\frac{1}{x_{0}+\rho_{n}x'}\left(\nabla\psi_{n}\cdot X'-\psi_{n}\right)^{2}\:dX'\\
    	&\leqslant\mathcal{D}_{X_{0},\psi}(\rho_{n}r_{2})-\mathcal{D}_{X_{0},\psi}(\rho_{n}r_{1})+\int_{\rho_{n}r_{1}}^{\rho_{n}r_{2}}r^{-4}\int_{\partial B_{r}(X_{0})}\frac{|x-x_{0}|}{x^{2}}\psi^{2}\:d\mathcal{H}^{1}dr\\
    	&\quad+\int_{\rho_{n}r_{1}}^{\rho_{n}r_{2}}r^{-3}|J_{0}(r)|\:dr+\int_{\rho_{n}r_{1}}^{\rho_{n}r_{2}}r^{-3}|K_{1}(r)|\:dr\\
    	&\to0\quad\text{ as }n\to\infty.
    \end{align*}
    The desired homogeneity then follows from the strong convergence of $\psi_{n}$ to $\psi_{0}$.
    
    (3). It follows from \eqref{Formula: D(3)} and a direct calculation that
    \begin{align*}
    	\mathcal{D}_{X_{0},\psi}(0^{+})&=\lim_{n\to\infty}\rho_{n}^{-2}\int_{B_{\rho_{n}}(X_{0})}\frac{1}{x}|\nabla\psi|^{2}-xyI_{\{\psi>0\}}-x\psi f(\psi)\:dX\\
    	&\quad-\lim_{n\to\infty}\rho_{n}^{-3}\int_{\partial B_{\rho_{n}}(X_{0})}\frac{1}{x}\psi^{2}\:d\mathcal{H}^{1}\\
    	&=-\lim_{n\to\infty}\int_{B_{1}(0)}(x_{0}+\rho_{n}x)(y_{0}+\rho_{n}y)I_{\{\psi_{n}>0\}}\:dX\\
    	&\quad-\lim_{n\to\infty}\int_{B_{1}(0)}(x_{0}+\rho_{n}x)\rho_{n}\psi_{n}f(\rho_{n}\psi_{n})\:dX\\
    	&\quad+\lim_{n\to\infty}\int_{B_{1}(0)}\frac{1}{x_{0}+\rho_{n}x}|\nabla\psi_{n}|^{2}\:dX-\lim_{n\to\infty}\int_{\partial B_{1}(0)}\frac{1}{x_{0}+\rho_{n}x}\psi_{n}^{2}\:d\mathcal{H}^{1}\\
    	&=-\lim_{n\to\infty}\int_{B_{1}(0)}(x_{0}+\rho_{n}x)(y_{0}+\rho_{n}y)I_{\{\psi_{n}>0\}}\:dX\\
    	&\quad+\int_{B_{1}(0)}\frac{1}{x_{0}}|\nabla\psi_{0}|^{2}\:dX-\int_{\partial B_{1}(0)}\frac{1}{x_{0}}\psi_{0}^{2}\:d\mathcal{H}^{1}\\
    	&=-x_{0}y_{0}\int_{B_{1}(0)}I_{\{\psi_{0}>0\}}\:dX,
    \end{align*}
    as wanted.
\end{proof}
In view of the last property in the previous Lemma \ref{Lemma: Homogeneity of minimizers}, we have
\begin{align*}
	\mathcal{D}_{X_{0},\psi}(0^{+})=-x_{0}y_{0}\omega_{2}\frac{\mathcal{L}^{2}(B_{1}(0)\cap\{\psi_{0}>0\})}{\mathcal{L}^{2}(B_{1}(0))},
\end{align*}
where $\omega_{2}:=\mathcal{L}^{2}(B_{1}(0))$. In fact, we will obtain a stronger result that if $X_{0}\in N_{\psi}$, then
\begin{align*}
	\frac{\mathcal{L}^{2}(B_{1}(0)\cap\{\psi_{0}>0\})}{\mathcal{L}^{2}(B_{1}(0))}=\frac{1}{2}.
\end{align*}
However, this result depends heavily on the homogeneity of minimizers.
\begin{corollary}\label{Corollary: Density}
	Let $\psi$ be a local minimizer of $J$ in $B_{R_{0}}(X_{0})$, then for every $X_{0}\in N_{\psi}$, the limit
	\begin{align*}
		\lim_{r\to0^{+}}\frac{\mathcal{L}^{2}(B_{r}(X_{0})\cap\{\psi>0\})}{\mathcal{L}^{2}(B_{r}(X_{0}))}:=\gamma
	\end{align*}
	exists and $\gamma=\tfrac{1}{2}$.
\end{corollary}
\begin{proof}
	\textbf{Step I}. In this step, we prove that $\gamma$ exists and is finite. Notice that in view of the proof in the last property of Lemma \ref{Lemma: Homogeneity of minimizers}, one has
	\begin{align*}
		\lim_{r\to0^{+}}\mathcal{D}_{X_{0},\psi}(r)=\lim_{n\to\infty}\mathcal{D}_{0,\psi_{n}}(1)=-x_{0}y_{0}\mathcal{L}^{2}(B_{1}(0)\cap\{\psi_{0}>0\}).
	\end{align*}
    On the other hand, since $\psi_{0}$ is an one-homogeneous function, one has
    \begin{align*}
    	\lim_{n\to\infty}\int_{B_{1}(0)}\frac{1}{x_{0}+\rho_{n}x}|\nabla\psi_{n}|^{2}-(x_{0}+\rho_{n}x)\rho_{n}\psi_{n}f(\rho_{n}\psi_{n})\:dX-\lim_{n\to\infty}\int_{\partial B_{1}(0)}\frac{1}{x_{0}+\rho_{n}x}\psi_{n}^{2}\:d\mathcal{H}^{1}=0.
    \end{align*}
    Therefore,
    \begin{align*}
    	&\mathcal{L}^{2}(B_{1}(0)\cap\{\psi_{0}>0\})\\
    	&=\frac{1}{-x_{0}y_{0}}\lim_{n\to\infty}\mathcal{D}_{0,\psi_{n}}(1)\\
    	&=\frac{1}{-x_{0}y_{0}}\lim_{n\to\infty}\left[\int_{B_{1}(0)}\frac{1}{x_{0}+\rho_{n}x}|\nabla\psi_{n}|^{2}-(x_{0}+\rho_{n}x)\rho_{n}\psi_{n}f(\rho_{n}\psi_{n})\:dX-\int_{\partial B_{1}(0)}\frac{1}{x_{0}+\rho_{n}x}\psi_{n}^{2}\:d\mathcal{H}^{1}\right]\\
    	&\quad+\frac{1}{-x_{0}y_{0}}\lim_{n\to\infty}\int_{B_{1}(0)}I_{\{\psi_{n}>0\}}(x_{0}+\rho_{n}x)(-y_{0}+\rho_{n}y)\:dX\\
    	&=\lim_{n\to\infty}\int_{B_{1}(0)}I_{\{\psi_{n}>0\}}\:dX=\lim_{n\to\infty}\frac{\int_{B_{\rho_{n}}(X_{0})}I_{\{\psi>0\}}(X_{0}+\rho_{n}X)\:dX}{\rho_{n}^{2}}\\
    	&=\lim_{n\to\infty}\frac{\mathcal{L}^{2}(B_{\rho_{n}}(X_{0})\cap\{\psi>0\})}{\mathcal{L}^{2}(B_{\rho_{n}}(X_{0}))},
    \end{align*}
    as desired.
    
    \textbf{Step II}. We finish the proof in this step. It follows from Lemma \ref{Lemma: Structure of the blow-up limits}, the first property in Corollary \ref{Corollary: Properties of blow-up limits}, and Lemma \ref{Lemma: Homogeneity of minimizers} that $\psi_{0}\in W_{\mathrm{loc}}^{1,2}(\mathbb{R}^{2})$ is a continuous, non-negative, non-constantly vanishing, one-homogeneous function satisfying $\Delta\psi_{0}=0$ in $\{\psi_{0}>0\}$. Expressing the Laplacian of $\psi_{0}$ in polar coordinates and it follows from Theorem 9.21 in \cite{V2019} that $\mathcal{H}^{1}(\{\psi_{0}>0\}\cap\partial B_{1})\geqslant\omega_{2}$. This implies
    \begin{align*}
    	\mathcal{L}^{2}(\{\psi_{0}>0\}\cap B_{1}(0))\geqslant\frac{\mathcal{L}^{2}(B_{1}(0))}{2},
    \end{align*}
    by the homogeneity of $\psi_{0}$. Since $I_{\{\psi_{n}>0\}}$ converges to $I_{\{\psi_{0}>0\}}$ strongly in $L_{\mathrm{loc}}^{1}(\mathbb{R}^{2})$ ( recalling the second property in Lemma \ref{Lemma: Structure of the blow-up limits}), one has
    \begin{align*}
    	\gamma=\lim_{n\to\infty}\frac{\mathcal{L}^{2}(\{\psi>0\}\cap B_{\rho_{n}}(X_{0}))}{\mathcal{L}^{2}(B_{\rho_{n}}(X_{0}))}=\frac{\mathcal{L}^{2}(\{\psi_{0}>0\}\cap B_{1}(0))}{\mathcal{L}^{2}(B_{1}(0))}\geqslant\frac{1}{2}.
    \end{align*}
    In particular, in the case of equality $\gamma=\frac{1}{2}$, we have by Remark 4.8 in \cite{MTV2022} that $\psi_{0}|_{\partial B_{1}}$ is precisely the first eigenvalue on $\mathbb{S}^{1}_{+}$, whose one-homogeneous extension is precisely the one-plane solution $(X\cdot\nu)^{+}$. This finishes the proof.
\end{proof}
We are now in a position to finish the proof of the first part in Theorem \ref{Theorem: main(1)}. In fact, we have already proved \eqref{Formula: maintheorem(1)} in Corollary \ref{Corollary: Density}, we now show that minimizers satisfy the equation in the viscosity sense, which is another application of the homogeneity of blow-up limits.
\begin{proposition}[Minimizers are viscosity solutions]\label{Proposition: minimizers are viscosity solutions}
	Let $X_{0}\in N_{\psi}$ and let $\psi$ be a local minimizer of $J$ in $B_{R_{0}}(X_{0})$, then $\psi$ is a viscosity solution to \eqref{Formula: governed equation} in the sense of Definition \ref{Definition: viscosity solutions}.
\end{proposition}
\begin{proof}
	\textbf{Step I}. Assume first that $X_{1}=(x_{1},y_{1})\in B_{R_{0}}(X_{0})\cap\{\psi>0\}$. Let $\phi\in C^{\infty}(B_{R_{0}}(X_{0}))$ be a function touching $\psi$ from below at $X_{1}$. Then $\psi$ is $C^{2,\alpha}$ smooth at $X_{1}$ (recalling Corollary \ref{Corollary: higher regularity in the positive phase}). Moreover, by the definition of "touch",  $(\psi-\phi)(X_{1})=0$ and $(\psi-\phi)(X)\geqslant0$ in a small neighborhood of $X_{1}$. Therefore the function $(\psi-\phi)$ attains a local minimum at $X_{1}$, which implies that
	\begin{align*}
		\nabla(\psi-\phi)(X_{1})=0\qquad\text{ and }\qquad\Delta(\psi-\phi)(X_{1})\geqslant0.
	\end{align*}
    Consequently,
    \begin{align*}
    	&\operatorname{div}\left(\frac{1}{x}\nabla\phi\right)(X_{1})+x_{1}f(\phi(X_{1}))\\
    	&=\frac{1}{x_{1}}\Delta\phi(X_{1})-\frac{1}{x_{1}^{2}}\pd{\phi}{x}(X_{1})+x_{1}f(\phi(X_{1}))\\
    	&=\frac{1}{x_{1}}\Delta\phi(X_{1})-\frac{1}{x_{1}^{2}}\pd{\psi}{x}(X_{1})+x_{1}f(\psi(X_{1}))\\
    	&\leqslant\frac{1}{x_{1}}\Delta\psi(X_{1})-\frac{1}{x_{1}^{2}}\pd{\psi}{x}(X_{1})+x_{1}f(\psi(X_{1}))\\
    	&=\operatorname{div}\left(\frac{1}{x}\nabla\psi\right)(X_{1})+x_{1}f(\psi(X_{1}))=0,
    \end{align*}
    where we have used the fact that $\psi$ satisfies the equation $\operatorname{div}\left(\frac{1}{x}\nabla\psi\right)+xf(\psi)$ in the region $B_{R_{0}}(X_{0})\cap\{\psi>0\}=0$ in the last equality. We obtain 
    \begin{align*}
    	\operatorname{div}\left(\frac{1}{x}\nabla\phi\right)(X_{1})+x_{1}f(\phi(X_{1}))\leqslant0,
    \end{align*}
    where $\phi$ is the function which  touches $\psi$ from below. This verifies the first requirement of Definition \ref{Definition: viscosity solutions}. It is easy to check the case when $\phi$ touches $\psi$ from above at $X_{0}$ in a similar way.
    
    \textbf{Step II}. It remains to check that $\psi$ satisfies the free boundary condition
    \begin{align*}
    	\frac{1}{x^{2}}|\nabla\psi|^{2}=-y
    \end{align*}
    in the viscosity sense (Recall Definition \ref{Definition: viscosity solutions}). Suppose first that the function $\phi$ touches $\psi$ from below at $X_{0}\in\partial\{\psi>0\}$. Consider the blow-up sequence
    \begin{align*}
    	\psi_{n}(X)=\frac{\psi(X_{0}+\rho_{n}X)}{\rho_{n}}\qquad\text{ and }\qquad\phi_{n}(X)=\frac{\phi(X_{0}+\rho_{n}X)}{\rho_{n}},
    \end{align*}
    as $\rho_{n}\to0^{+}$. It follows from Lemma \ref{Lemma: Structure of the blow-up limits} that up to a subsequence,
    \begin{align*}
    	\psi_{0}(x)=\lim_{n\to\infty}\psi_{n}(x)\qquad\text{ and }\qquad\phi_{0}=\lim_{n\to\infty}\phi_{n}(X),
    \end{align*}
    the convergence being uniform in $B_{1}$. By Corollary \ref{Corollary: Properties of blow-up limits}, $\psi_{0}$ is harmonic in $\{\psi_{0}>0\}$. Since $\phi$ is a smooth function in $B_{R_{0}}(X_{0})$, we have $\phi_{0}=\xi_{0}\cdot X$, where the vector $\xi_{0}\in\mathbb{R}^{2}$ is precisely the gradient $\nabla\phi(0)$. Without loss of generality we may assume that $\xi_{0}=Ae_{2}$ for some non-negative constant $A\geqslant0$, therefore
    \begin{align*}
    	|\nabla\phi(X_{0})|=|\nabla\phi_{0}(0)|=A\quad\text{ and }\quad\phi_{0}(X)=Ay.
    \end{align*}
    Moreover, we can assume that $A>0$ since otherwise the inequality $|\nabla\phi(X_{0})|\leqslant x_{0}\sqrt{-y_{0}}$ holds trivially.
    
    Since $\psi_{0}\geqslant\phi_{0}$, we obtain that $\psi_{0}>0$ in the set $\{y>0\}$. Thus, $\psi_{0}$ is a one-homogeneous harmonic function defined in the cone $\{\psi_{0}>0\}\supset\{y>0\}$, it follows from Lemma 5.31 in \cite{RTV2018} that there are only two possibilities
    \begin{align*}
    	\psi_{0}(X)=\alpha y^{+}\quad\text{ or }\quad\psi_{0}(X)=\alpha y^{+}+\beta y^{-},
    \end{align*}
    where $\alpha$ and $\beta$ are positive constants. Note that in the second case, one has $\frac{\mathcal{L}^{2}(\{\psi_{0}>0\}\cap B_{1}(0))}{\mathcal{L}^{2}(B_{1}(0))}=1$, therefore it is ruled out since it contradicts with Corollary \ref{Corollary: Density}. Thus,
    \begin{align*}
    	\psi_{0}(X)=\alpha y^{+}\quad\text{ for every }X\in B_{1}(0).
    \end{align*}
    It then follows from Corollary \ref{Corollary: Properties of blow-up limits} that $\alpha=x_{0}\sqrt{-y_{0}}$. It then follows from the inequality $\psi_{0}\geqslant\phi_{0}$ that $x_{0}\sqrt{-y_{0}}\geqslant A$, and so one has $|\nabla\phi(X_{0})|=A\leqslant x_{0}\sqrt{-y_{0}}$. 
    
    Suppose now that $\phi$ touches $\psi$ from above at $X_{0}$. We still consider the blow-up limit $\psi_{0}$ and $\phi_{0}$ defined previously and we also assume that $\phi_{0}$ is given by $\phi_{0}=Ay$. In this case the set $\{\psi_{0}>0\}\subset\{y>0\}$, and Lemma 5.31 in \cite{RTV2018} tells us  $\{\psi_{0}>0\}=\{y>0\}$ and 
    \begin{align*}
    	\psi_{0}(X)=\alpha y^{+}\quad\text{ for every }x\in B_{1}(0).
    \end{align*}
    Along the same argument as in the first case, we have  $\alpha=x_{0}\sqrt{-y_{0}}$. Since $\psi_{0}\leqslant\phi_{0}$, we have that $|\nabla\phi(X_{0})|=|\nabla\phi_{0}(0)|=A\geqslant x_{0}\sqrt{-y_{0}}$, as wanted.
\end{proof}
At this stage, we have proved the former part of the Theorem \ref{Theorem: main(1)}. If $X_{0}\in N_{\psi}$, then $\psi$ is a Lipschitz viscosity solution to the problem \eqref{Formula: governed equation} in $B_{R_{0}}(X_{0})$. Moreover, $X_{0}\in N_{\psi}$ if and only if
\begin{align*}
	\mathcal{D}_{X_{0},\psi}(0^{+})=-x_{0}y_{0}\frac{\omega_{2}}{2}.
\end{align*}
In the concluding section, we will study the regularity of free boundary near non-degenerate points.
\section{Regularity of the non-degenerate points}\label{Section: regularity of the non-degenerate points}
This section aims to study the regularity of the free boundary near points in $N_{\psi}$. We begin with some primary settings. Let $X_{0}\in N_{\psi}$ and let $\psi$ be a local minimizer of $J$ in $B_{R_{0}}(X_{0})$. It follows from Proposition \ref{Proposition: Lipschitz regularity for local minimizers} that $\psi\in C^{0,1}(B_{R_{0}}(X_{0}))$. Define
\begin{align*}
	\tilde{f}(X):=-x^{2}(f\circ\psi)(X)\quad\text{ for every }X=(x,y)\in B_{R_{0}}(X_{0}),
\end{align*}
then after a direct calculation,
\begin{align*}
	|\tilde{f}(X)|=|x^{2}f(\psi(X))|\leqslant(x_{0}+R_{0})^{2}\left(|f(0)|+\max_{s\geqslant0}|f'(s)|\psi\right)\leqslant C(X_{0},F_{0},\|\psi\|_{L^{\infty}(B_{R_{0}}(X_{0}))}),
\end{align*}
where $F_{0}$ is defined in \eqref{Formula: assumptions on f}. We may assume that for every $X_{0}\in N_{\psi}$, $\tilde{f}(X)\in C^{0}(B_{R_{0}}(X_{0}))\cap L^{\infty}(B_{R_{0}}(X_{0}))$ (It should be stressed that the $L^{\infty}$ norm and the $C^{0}$ norm of $\tilde{f}$ depend on $X_{0}$, $F_{0}$ and $\|\psi\|_{L^{\infty}(B_{R_{0}}(X_{0}))}$). Moreover, it follows from Proposition \ref{Proposition: minimizers are viscosity solutions} that if $\psi$ is a local minimizer of $J$ in $B_{R_{0}}(X_{0})$, then $\psi$ solves the following problem
\begin{align}\label{Formula: R(1)}
	\begin{cases}
		\begin{alignedat}{2}
			\Delta\psi-\frac{1}{x}\pd{\psi}{x}&=\tilde{f}(X)\quad&&\text{ in }B_{R_{0}}(X_{0})\cap\{\psi>0\},\\
			|\nabla\psi|&=x\sqrt{-y}\quad&&\text{ on }B_{R_{0}}(X_{0})\cap\partial\{\psi>0\},
		\end{alignedat}
	\end{cases}
\end{align}
in the viscosity sense. Instead of directly studying \eqref{Formula: R(1)}, we begin by analyzing the following more general free boundary problem. 

Let $\Omega$ be a bounded Lipschitz domain in $\mathbb{R}^{2}$, let $a_{ij}(X)\in C^{0,\beta}(\Omega)$ be uniformly elliptic with constants $0<\lambda\leqslant\Lambda<\infty$, let $\tilde{f}(X)\in C^{0}(\Omega)\cap L^{\infty}(\Omega)$ and let $\mathbf{b}(X)=(b_{1}(X),b_{2}(X))\in C^{0}(\Omega)\cap L^{\infty}(\Omega)$. Consider the free boundary problem
\begin{align}\label{Formula: R(2)}
	\begin{cases}
		\displaystyle\sum_{i,j=1}^{2}a_{ij}(X)D_{ij}u+\mathbf{b}(X)\cdot\nabla u=\tilde{f}(X)\quad&\text{ in }\varOmega^{+}(u):=\{u>0\}\cap\Omega,\\
		|\nabla u|=Q(X)\quad&\text{ on }\varGamma(u):=\partial\{u>0\}\cap\Omega.
	\end{cases}
\end{align}
Here $Q(X)\in C^{0,\beta}(\Omega)$, $Q(X)>0$. In this section, we will study the viscosity solutions to the problem \eqref{Formula: R(2)}. In particular, we will show how the improvement of flatness implies the $C^{1,\gamma}$ regularity of the free boundary.
\begin{remark}
	On the one hand, if we choose $a_{ij}=\delta_{ij}$, $\mathbf{b}(X)=(\tfrac{1}{x},0)$, $Q(X)=x\sqrt{-y}$ in \eqref{Formula: R(2)}, then the problem \eqref{Formula: R(2)} and the problem \eqref{Formula: R(1)} are governed by the same equation and free boundary condition, respectively. On the other hand, if $X_{0}\in N_{\psi}$, then it follows from \eqref{Formula: R0} that $\frac{1}{x}\in L^{\infty}(B_{R_{0}}(X_{0}))\cap C^{0}(B_{R_{0}}(X_{0}))$ and $x\sqrt{-y}\in C^{0,\beta}(B_{R_{0}}(X_{0}))$ with $x\sqrt{-y}>0$ in $B_{R_{0}}(X_{0})$. This means that the regularity assumptions we require in \eqref{Formula: R(2)} is automatically satisfied by local minimizers of problem \eqref{Formula: R(1)} in $B_{R_{0}}(X_{0})$ with $X_{0}\in N_{\psi}$. 
\end{remark}
For the abuse of notation, we denote
\begin{align*}
	\mathfrak{L}:=\sum_{i,j=1}^{2}a_{ij}(X)D_{ij}+\mathbf{b}(X)\cdot\nabla,
\end{align*}
and we stress that a constant depending on the Lipschitz norm of $u$, $\lambda$, $\Lambda$, $[a_{ij}]_{C^{0,\beta}}$, $\|\tilde{f}\|_{L^{\infty}}$ and $\|\mathbf{b}\|_{L^{\infty}}$ will be called universal. Our first main result in this section is the following “flatness implies $C^{1,\gamma}$-regularity”. We begin with the definition of “flatness”
\begin{definition}[Flatness]
	Let $u\colon B_{1}(0)\to\mathbb{R}$ be a given function. Let $\varepsilon>0$ be a fixed real number and let $\nu\in\mathbb{R}^{2}$ a unit vector. We say that $u$ is $\varepsilon$-flat, in the direction $\nu$, in $B_{1}(0)$ (please see Fig \ref{Fig: Flatness}), provided that
	\begin{align}\label{Formula: flatness expression}
		(X\cdot\nu-\varepsilon)^{+}\leqslant u(X)\leqslant(X\cdot\nu+\varepsilon)^{+}\quad\text{ in }B_{1}(0).
	\end{align}
	In other words, \eqref{Formula: flatness expression} says that $\{X\cdot\nu<-\varepsilon\}\subset\{u=0\}\subset\{X\cdot\nu<\varepsilon\}$ in $B_{1}(0)$.
\end{definition}
\begin{figure}[ht]
	\tikzset{every picture/.style={line width=0.75pt}} 
	\begin{tikzpicture}[x=0.75pt,y=0.75pt,yscale=-0.9,xscale=0.9]
		
		\draw  [line width=0.75]  (203.54,111.2) .. controls (219.18,76.07) and (260.34,60.27) .. (295.47,75.91) .. controls (330.6,91.55) and (346.4,132.71) .. (330.76,167.84) .. controls (315.12,202.97) and (273.96,218.76) .. (238.83,203.12) .. controls (203.7,187.48) and (187.9,146.33) .. (203.54,111.2) -- cycle ;
		\draw [line width=0.75]    (267.15,139.52) ;
		\draw [shift={(267.15,139.52)}, rotate = 0] [color={rgb, 255:red, 0; green, 0; blue, 0 }  ][fill={rgb, 255:red, 0; green, 0; blue, 0 }  ][line width=0.75]      (0, 0) circle [x radius= 3.35, y radius= 3.35]   ;
		\draw [color={rgb, 255:red, 0; green, 0; blue, 255 }  ,draw opacity=1 ][line width=0.75]    (167.02,74) .. controls (200.77,83.01) and (201.71,128.85) .. (267.15,139.52) ;
		\draw [color={rgb, 255:red, 0; green, 0; blue, 255 }  ,draw opacity=1 ][line width=0.75]    (267.15,139.52) .. controls (316.58,144.97) and (315.73,200.97) .. (391.09,174.31) ;
		\draw [line width=0.75]  [dash pattern={on 4.5pt off 4.5pt}]  (166.6,84.76) -- (397.48,187.01) ;
		\draw [line width=0.75]  [dash pattern={on 4.5pt off 4.5pt}]  (160.75,104.05) -- (391.63,206.3) ;
		\draw [line width=0.75]    (267.15,139.52) -- (303.05,57.39) ;
		\draw [shift={(303.85,55.56)}, rotate = 113.61] [color={rgb, 255:red, 0; green, 0; blue, 0 }  ][line width=0.75]    (10.93,-3.29) .. controls (6.95,-1.4) and (3.31,-0.3) .. (0,0) .. controls (3.31,0.3) and (6.95,1.4) .. (10.93,3.29)   ;

		\draw (401.29,163.69) node [anchor=north west][inner sep=0.75pt]  [rotate=-24]  {$\textcolor[rgb]{0,0,1}{\varGamma }\textcolor[rgb]{0,0,1}{(}\textcolor[rgb]{0,0,1}{u}\textcolor[rgb]{0,0,1}{)}$};
		\draw (394.22,185.72) node [anchor=north west][inner sep=0.75pt]  [rotate=-24]  {$2\varepsilon $};
		\draw (274.13,143.2) node [anchor=north west][inner sep=0.75pt]  [rotate=-24]  {$0$};
		\draw (293.4,98.69) node [anchor=north west][inner sep=0.75pt]  [rotate=-24]  {$\nu $};
		\draw (300.13,204.04) node [anchor=north west][inner sep=0.75pt]  [rotate=-24]  {$B_{1}( 0)$};
		\draw (232.44,86.87) node [anchor=north west][inner sep=0.75pt]  [rotate=-24]  {$u >0$};
		\draw (212.1,137.47) node [anchor=north west][inner sep=0.75pt]  [rotate=-24]  {$u=0$};	
	\end{tikzpicture}
	\caption{In $B_{1}(0)$, the free boundary is trapped between two parallel hyperplanes at $\varepsilon$-distance from each other}
	\label{Fig: Flatness}
\end{figure}
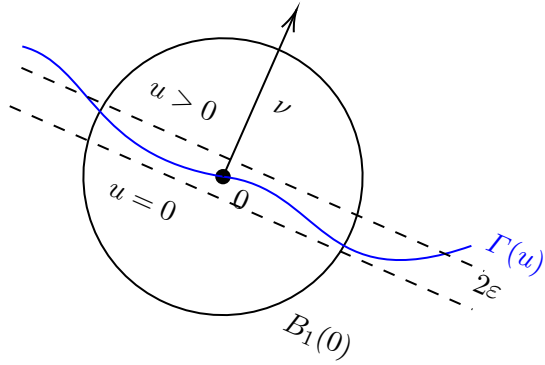
\begin{proposition}[Flatness implies $C^{1,\gamma}$]\label{Proposition: Flatness implies C1alpha}
	Let $u$ be a Lipschitz viscosity solution to \eqref{Formula: R(2)} in $B_{1}$. Assume that $0\in\varGamma(u)$, $Q(0)=1$ and $a_{ij}(0)=\delta_{ij}$. There exists a universal constant $\bar{\varepsilon}>0$ such that, if $u$ is $\varepsilon$-flat in $B_{1}(0)$, that is, 
	\begin{align*}
		(y-\varepsilon)^{+}\leqslant u(X)\leqslant(y+\varepsilon)^{+}\quad\text{ in }B_{1}(0),
	\end{align*}
    with $0\leqslant\varepsilon\leqslant\bar{\varepsilon}$, then $\varGamma(u)$ is $C^{1,\gamma}$ in $B_{1/2}(0)$.
\end{proposition}
\begin{remark}
	Let $u\in C^{0}(\Omega)$, we say that $u$ is a \emph{viscosity solution} to \eqref{Formula: R(2)} in $\Omega$ if the following conditions are satisfied:
	\begin{enumerate}
		\item $\displaystyle\sum_{i,j=1}^{2}a_{ij}(X)D_{ij}u+\mathbf{b}(X)\cdot\nabla u=\tilde{f}(X)$ in $\varOmega^{+}(u)$ in the viscosity sense. i.e., if $\phi\in C^{2}(\varOmega^{+}(u))$ touches $u$ from below (resp. above) at $X_{1}\in\varOmega^{+}(u)$ then
		\begin{align*}
			\sum_{i,j=1}^{2}a_{ij}(X_{1})D_{ij}\phi(X_{1})+\mathbf{b}(X_{1})\nabla\phi(X_{1})\leqslant(\text{ resp. }\geqslant)\tilde{f}(X_{1}).
		\end{align*}
		\item If $\phi\in C^{2}(\Omega)$ touches $u$ from below (resp. above) at $X_{0}\in\varGamma(u)$ then
		\begin{align*}
			|\nabla\phi(X_{0})|\leqslant(\text{ resp. }\geqslant)Q(X_{0}).
		\end{align*}
	\end{enumerate}
\end{remark}
\begin{remark}
	The method “flatness implies $C^{1,\gamma}$” was originated from the celebrated work \cite{AC1981} of \caps{Alt} and \caps{Caffarelli}, who gave the basic fact that if the free boundary is “flat” enough at some point, then it is regular near that point. It was then due to \caps{De Silva} \cite{S2011}, who for the first time developed the viscosity approach of \caps{Caffarelli}. Compared to \cite{S2011}, the main difference is that our equation \eqref{Formula: R(2)} involves the lower order term $\mathbf{b}\cdot\nabla u$. However, we will prove in this section that if $\mathbf{b}$ is bounded and continuous (correspond to the case that $X_{0}\in N_{\psi}$ is a non-degenerate point), then we still obtain the regularity of the free boundary. 
\end{remark}
\begin{remark}
	Proposition \ref{Proposition: Flatness implies C1alpha} can be easily extended to $d$-dimensions, for $d\geqslant3$, with a slight change in the proof. 
\end{remark}
To prove Proposition \ref{Proposition: Flatness implies C1alpha}, we need the following lemma.
\begin{lemma}\label{Lemma: Flatness implies C1gamma}
	Let $u$ be a viscosity solution to \eqref{Formula: R(2)} in $B_{1}$. Assume that $0\in\varGamma(u)$, $Q(0)=1$ and $a_{ij}(0)=\delta_{ij}$. There exists a universal constant $\bar{\varepsilon}$ such that, if $u$ is $\bar{\varepsilon}$-flat in $B_{1}(0)$, that is,
	\begin{align*}
		(y-\bar{\varepsilon})^{+}\leqslant u(X)\leqslant(y+\bar{\varepsilon})^{+},\quad\text{ in }B_{1},
	\end{align*}
    and
	\begin{align}\label{Formula: R(2(2))}
		[a_{ij}]_{C^{0,\beta}(B_{1})}\leqslant\bar{\varepsilon},\quad\|\tilde{f}\|_{L^{\infty}(B_{1})}\leqslant\bar{\varepsilon},\quad\|\mathbf{b}\|_{L^{\infty}(B_{1})}\leqslant\bar{\varepsilon},\quad[Q]_{C^{0,\beta}(B_{1})}\leqslant\bar{\varepsilon},
	\end{align}
    then $\varGamma(u)$ is $C^{1,\gamma}$ in $B_{1/2}(0)$.
\end{lemma}
\begin{remark}
	Compared to \cite{S2011}, which considered the case $\mathbf{b}=0$. We found that for a free boundary problem with the first order term $\mathbf{b}\cdot\nabla u$, the order of $\mathbf{b}$ we require should be the same to the order of the right-hand side. This observation helps us to establish a partial Harnack inequality to the  problem \eqref{Formula: R(2)}.
\end{remark}
In order to prove Lemma \ref{Lemma: Flatness implies C1gamma}, we have to establish a Harnack type inequality and an improvement of flatness procedure as in \cite{S2011}. We begin by proving a Harnack type inequality for a solution $u$ to the problem
\begin{align}\label{Formula: R(3)}
	\begin{cases}
		\displaystyle\sum_{i,j=1}^{2}a_{ij}(X)D_{ij}u+\mathbf{b}\cdot\nabla u=\tilde{f}(X)\quad&\text{ in }\varOmega^{+}(u),\\
		|\nabla u|=Q(X)\quad&\text{ on }\varGamma(u),
	\end{cases}
\end{align}
under the assumption $(0<\varepsilon<1)$
\begin{align}\label{Formula: R(4)}
	\|\tilde{f}\|_{L^{\infty}(B_{1})}\leqslant\varepsilon^{2},\quad\|\mathbf{b}\|_{L^{\infty}(B_{1})}\leqslant\varepsilon^{2},\quad\|Q-1\|_{L^{\infty}(\Omega)}\leqslant\varepsilon^{2},\quad\|a_{ij}-\delta_{ij}\|_{L^{\infty}(\Omega)}\leqslant\varepsilon.
\end{align}
\begin{proposition}[Partial Harnack inequality]\label{Proposition: Harnack inequality}
	There exists a universal constant $\bar{\varepsilon}$ such that if $u$ solves \eqref{Formula: R(3)}-\eqref{Formula: R(4)}, and for some point $X_{1}\in\varOmega^{+}(u)\cup\varGamma(u)$,
	\begin{align*}
		(y+a_{0})^{+}\leqslant u(X)\leqslant(y+b_{0})^{+}\quad\text{ in }B_{r}(X_{1})
	\end{align*}
	with 
	\begin{align*}
		b_{0}-a_{0}\leqslant\varepsilon r,\qquad\varepsilon\leqslant\bar{\varepsilon},
	\end{align*}
	then
	\begin{align*}
		(y+a_{1})^{+}\leqslant u(X)\leqslant(y+b_{1})^{+}\quad\text{ in }B_{r/20}(X_{1})
	\end{align*}
	with
	\begin{align*}
		a_{0}\leqslant a_{1}\leqslant b_{1}\leqslant b_{0},\qquad b_{1}-a_{1}\leqslant(1-c)\varepsilon r,
	\end{align*}
	and $c\in(0,1)$ universal.
\end{proposition}
\begin{remark}
	In fact, Proposition \ref{Proposition: Harnack inequality} can be viewed as a weak version of improvement of flatness, it merely says that the flatness is improved in some smaller ball $B_{r}$ in the same direction $\nu$ (please see Fig \ref{Fig: partial boundary harnack}). In other words, the flatness direction $\nu$ does not change.
\end{remark}
\begin{figure}[ht]
	\tikzset{every picture/.style={line width=0.75pt}} 
	\begin{tikzpicture}[x=0.75pt,y=0.75pt,yscale=-0.9,xscale=0.9]
		
		\draw [line width=0.75]    (185,119.25) -- (374.28,232.72) ;
		\draw [shift={(376,233.75)}, rotate = 210.94] [color={rgb, 255:red, 0; green, 0; blue, 0 }  ][line width=0.75]    (10.93,-3.29) .. controls (6.95,-1.4) and (3.31,-0.3) .. (0,0) .. controls (3.31,0.3) and (6.95,1.4) .. (10.93,3.29)   ;
		\draw [line width=0.75]    (223.38,236.07) -- (309.11,86.35) ;
		\draw [shift={(310.11,84.62)}, rotate = 119.8] [color={rgb, 255:red, 0; green, 0; blue, 0 }  ][line width=0.75]    (10.93,-3.29) .. controls (6.95,-1.4) and (3.31,-0.3) .. (0,0) .. controls (3.31,0.3) and (6.95,1.4) .. (10.93,3.29)   ;
		\draw  [line width=0.75]  (215.72,138.28) .. controls (230.46,112.75) and (263.46,104.2) .. (289.43,119.19) .. controls (315.4,134.18) and (324.5,167.04) .. (309.75,192.58) .. controls (295.01,218.12) and (262,226.66) .. (236.04,211.67) .. controls (210.07,196.68) and (200.97,163.82) .. (215.72,138.28) -- cycle ;
		\draw [line width=0.75]  [dash pattern={on 4.5pt off 4.5pt}]  (178.13,84.88) -- (398.43,212.24) ;
		\draw [line width=0.75]  [dash pattern={on 4.5pt off 4.5pt}]  (144.14,127.61) -- (370.53,259.43) ;
		\draw [color={rgb, 255:red, 255; green, 0; blue, 0 }  ,draw opacity=1 ][line width=0.75]  [dash pattern={on 4.5pt off 4.5pt}]  (174,99.75) -- (366.13,214.21) ;
		\draw [color={rgb, 255:red, 255; green, 0; blue, 0 }  ,draw opacity=1 ][line width=0.75]  [dash pattern={on 4.5pt off 4.5pt}]  (160.5,117.25) -- (354.62,227.64) ;
		\draw [color={rgb, 255:red, 0; green, 0; blue, 255 }  ,draw opacity=1 ][line width=0.75]    (220.76,151.42) .. controls (237.41,143.06) and (281.32,184.27) .. (296.94,175.84) ;
		\draw [color={rgb, 255:red, 0; green, 0; blue, 255 }  ,draw opacity=1 ][line width=0.75]    (177.72,175.71) .. controls (223.69,170.54) and (202.21,156.56) .. (220.76,151.42) ;
		\draw [color={rgb, 255:red, 0; green, 0; blue, 255 }  ,draw opacity=1 ][line width=0.75]    (296.94,175.84) .. controls (312.71,164.34) and (351.24,150.65) .. (382.5,164.46) ;
		\draw  [color={rgb, 255:red, 255; green, 0; blue, 0 }  ,draw opacity=1 ][line width=0.75]  (233.55,148.58) .. controls (242.7,132.73) and (263.19,127.42) .. (279.31,136.73) .. controls (295.42,146.03) and (301.07,166.43) .. (291.92,182.28) .. controls (282.77,198.13) and (262.28,203.44) .. (246.16,194.13) .. controls (230.04,184.83) and (224.4,164.43) .. (233.55,148.58) -- cycle ;
		\draw [line width=0.75]    (388,231.25) -- (397.47,214) ;
		\draw [shift={(398.43,212.24)}, rotate = 118.77] [color={rgb, 255:red, 0; green, 0; blue, 0 }  ][line width=0.75]    (10.93,-3.29) .. controls (6.95,-1.4) and (3.31,-0.3) .. (0,0) .. controls (3.31,0.3) and (6.95,1.4) .. (10.93,3.29)   ;
		\draw [line width=0.75]    (377.85,246.75) -- (371.53,257.7) ;
		\draw [shift={(370.53,259.43)}, rotate = 300] [color={rgb, 255:red, 0; green, 0; blue, 0 }  ][line width=0.75]    (10.93,-3.29) .. controls (6.95,-1.4) and (3.31,-0.3) .. (0,0) .. controls (3.31,0.3) and (6.95,1.4) .. (10.93,3.29)   ;
		
		\draw (380.3,228.52) node [anchor=north west][inner sep=0.75pt]  [rotate=-30]  {$2\varepsilon $};
		\draw (387.25,160.6) node [anchor=north west][inner sep=0.75pt]  [rotate=-30]  {$\textcolor[rgb]{0,0,1}{\varGamma }\textcolor[rgb]{0,0,1}{(}\textcolor[rgb]{0,0,1}{u}\textcolor[rgb]{0,0,1}{)}$};
		\draw (279.99,215.36) node [anchor=north west][inner sep=0.75pt]  [rotate=-30]  {$B_{1}( 0)$};
		\draw (141.48,88.06) node [anchor=north west][inner sep=0.75pt]  [font=\scriptsize,rotate=-30]  {$\textcolor[rgb]{1,0,0}{2}\textcolor[rgb]{1,0,0}{(}\textcolor[rgb]{1,0,0}{1-c}\textcolor[rgb]{1,0,0}{)}\textcolor[rgb]{1,0,0}{\varepsilon }$};
		\draw (287.98,188.8) node [anchor=north west][inner sep=0.75pt]  [rotate=-30]  {$\textcolor[rgb]{1,0,0}{B}\textcolor[rgb]{1,0,0}{_{r}}\textcolor[rgb]{1,0,0}{(}\textcolor[rgb]{1,0,0}{0}\textcolor[rgb]{1,0,0}{)}$};
		\draw (318.77,81.42) node [anchor=north west][inner sep=0.75pt]  [rotate=-30]  {$\nu $};		
	\end{tikzpicture}
    \caption{Partial improvement of flatness.}
    \label{Fig: partial boundary harnack}
\end{figure}
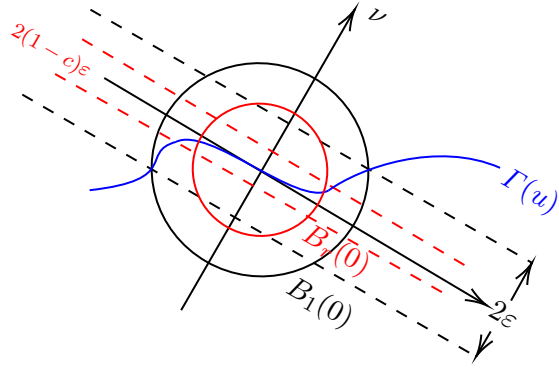
The proof if Proposition \ref{Proposition: Harnack inequality} relies on the following lemma.
\begin{lemma}\label{Lemma: Harnack inequality}
	There exists a universal constant $\bar{\varepsilon}>0$ such that if $u$ is a solution to \eqref{Formula: R(3)}-\eqref{Formula: R(4)} in $B_{1}(0)$ with $0<\varepsilon\leqslant\bar{\varepsilon}$ and $u$ satisfies
	\begin{align}\label{Formula: R(5)}
		(y+\sigma)^{+}\leqslant u(X)\leqslant(y+\sigma+\varepsilon)^{+},\quad x\in B_{1}(0),|\sigma|<1/10,
	\end{align}
	then if at $\bar{X}=(0,\tfrac{1}{5})$,
	\begin{align}\label{Formula: R(5(1))}
		u(\bar{X})\geqslant\left(\tfrac{1}{5}+\sigma+\tfrac{\varepsilon}{2}\right)^{+}
	\end{align}
	then
	\begin{align*}
		u(X)\geqslant(y+\sigma+c\varepsilon)^{+}\quad\text{ in }\bar{B}_{1/2}(0)
	\end{align*}
	for some $c\in(0,1)$ universal. Analogously, if $u(\bar{X})\leqslant\left(\tfrac{1}{5}+\sigma+\tfrac{\varepsilon}{2}\right)^{+}$, then
	\begin{align*}
		u(X)\leqslant(y+(1-c)\varepsilon)^{+}\quad\text{ in }\bar{B}_{1/2}(0).
	\end{align*}
\end{lemma}
\begin{proof}
	The proof is similar to Lemma 3.3 in \cite{S2011}, and the main differences involve in the term $\mathbf{b}\cdot\nabla u$. We only prove the first statement and the proof of the second one follows from a similar argument. Denote $p(X)=y+\sigma$. It follows from \eqref{Formula: R(5)} that
	\begin{align}\label{Formula: R(8)}
		u\geqslant p\quad\text{ in }B_{1}.
	\end{align}
	Let
	\begin{align*}
		w(X):=c(|X-\bar{X}|^{-\gamma}-(\tfrac{3}{4})^{-\gamma}).
	\end{align*}
    be defined in the closure of the annulus
    \begin{align}\label{Formula: R(8(1(1))}
    	A:=B_{3/4}(\bar{X})\setminus\bar{B}_{1/20}(\bar{X}).
    \end{align}
    Here $c$ is a constant so that $w$ satisfies the boundary conditions
    \begin{align*}
    	\begin{cases}
    		w=0\quad&\text{ on }\partial B_{3/4}(\bar{X}),\\
    		w=1\quad&\text{ on }\partial B_{1/20}(\bar{X}).
    	\end{cases}
    \end{align*}
    A straightforward computation gives that
    \begin{align*}
    	&\sum_{i,j=1}^{2}a_{ij}(X)D_{ij}w(X)\\
    	&=\gamma(\gamma+2)|X-\bar{X}|^{-\gamma-4}\sum_{i,j=1}^{2}a_{ij}(x_{j}-\bar{x}_{j})(x_{i}-\bar{x}_{i})-\gamma|X-\bar{X}|^{-\gamma-2}\sum_{i,j=1}^{2}a_{ij}\delta_{ij}\\
    	&\geqslant\gamma(\gamma+2)|X-\bar{X}|^{-\gamma-2}(2\lambda)-\gamma|X-\bar{X}|^{-\gamma-2}(2\Lambda)\\
    	&=\gamma|X-\bar{X}|^{-\gamma-2}2((\gamma+2)\lambda-\Lambda).
    \end{align*}
    Thus,
    \begin{align*}
    	\mathfrak{L}w&\geqslant\gamma|X-\bar{X}|^{-\gamma-2}2((\gamma+2)\lambda-\Lambda)-\gamma\|\mathbf{b}\|_{L^{\infty}}|X-\bar{X}|^{-\gamma-1}\\
    	&=\gamma|X-\bar{X}|^{-\gamma-2}(2((\gamma+2)\lambda-\Lambda)-\|\mathbf{b}\|_{L^{\infty}}|X-\bar{X}|)\\
    	&\geqslant\gamma|X-\bar{X}|^{-\gamma-2}(2((\gamma+2)\lambda-\Lambda)-\|\mathbf{b}\|_{L^{\infty}})>0,
    \end{align*}
    provided that the constant $\gamma$ satisfies
    \begin{align*}
    	2((\gamma+2)\lambda-\Lambda)-\|\mathbf{b}\|_{L^{\infty}}>0.
    \end{align*}
    The above argument shows that there exists a constant $\gamma=\gamma(\lambda,\Lambda,\|\mathbf{b}\|_{L^{\infty}})$ so that
    \begin{align}\label{Formula: R(8(1))}
    	\mathfrak{L}w\geqslant\delta>0\quad\text{ in }A,
    \end{align}
    with a positive universal constant $\delta$. Extend $w$ to be equal to $1$ in $\bar{B}_{1/20}(\bar{X})$. Due to the fact that  $|\sigma|<\tfrac{1}{10}$, we have from \eqref{Formula: R(8)} that
    \begin{align*}
    	B_{1/10}(\bar{X})\subset B_{1}(0)\cap\{u>0\}.
    \end{align*}
    Also $B_{1/2}\subset\subset B_{3/4}(\bar{X})\subset\subset B_{1}$. $u-p\geqslant0$ solves
    \begin{align*}
    	\mathfrak{L}(u-p)=\tilde{f}-b_{2}.
    \end{align*}
    By the Harnack inequality, we obtain that in $\bar{B}_{1/20}(\bar{X})$,
    \begin{align*}
    	\begin{alignedat}{2}
    		u(X)-p(X)&\geqslant c(u(\bar{X})-p(\bar{X}))-C\|\tilde{f}-b_{2}\|_{L^{\infty}}\\
    		&\geqslant c(u(\bar{X})-p(\bar{X}))-C(\|\tilde{f}\|_{L^{\infty}}+\|\mathbf{b}\|_{L^{\infty}}).
    	\end{alignedat}
    \end{align*}
    Recalling \eqref{Formula: R(5(1))} and the first and the second inequality in \eqref{Formula: R(4)} we conclude that (for $\varepsilon$ small enough)
    \begin{align}\label{Formula: R(11)}
    	u-p\geqslant c\varepsilon-C\varepsilon^{2}\geqslant c_{0}\varepsilon\quad\text{ in }\bar{B}_{1/20}(\bar{X}).
    \end{align}
    Define now a function
    \begin{align}\label{Formula: R(12)}
    	v(X)=p(X)+c_{0}\varepsilon(w(X)-1)\quad\text{ in }\bar{B}_{3/4}(\bar{X}),
    \end{align}
    and for $t\geqslant0$,
    \begin{align*}
    	v_{t}:=v(x)+t\quad\text{ in }\bar{B}_{3/4}(\bar{X}).
    \end{align*}
    Recalling \eqref{Formula: R(8(1))}, one gets
    \begin{align*}
    	\mathfrak{L}v_{t}=\mathfrak{L}p(X)+c_{0}\varepsilon\mathfrak{L}w(X)\geqslant b_{2}+c_{0}\varepsilon\delta>\varepsilon^{2}\quad\text{ in }A,
    \end{align*}
    provided that $\varepsilon$ is small enough. According to \eqref{Formula: R(8)} and the definition of $v_{t}$ one has
    \begin{align*}
    	v_{0}(X)=v(X)\leqslant p(X)\leqslant u(X)\quad\text{ in }\bar{B}_{3/4}(\bar{X}).
    \end{align*}
    Let now $\bar{t}$ be the largest $t\geqslant0$ so that
    \begin{align}\label{Formula: R(12(1))}
    	v_{t}(X)\leqslant u(X)\quad\text{ in }\bar{B}_{3/4}(\bar{X}).
    \end{align}
    We claim that
    \begin{align*}
    	\bar{t}\geqslant c_{0}\varepsilon.
    \end{align*}
    Suppose not, then $\bar{t}<c_{0}\varepsilon$. Then there must exist $\tilde{X}\in\bar{B}_{3/4}(\bar{X})$ so that $v_{\bar{t}}(\tilde{X})=u(\tilde{X})$. We consider the following cases:
    \begin{enumerate}
    	\item If $\tilde{X}\in\partial B_{3/4}(\bar{X})$, then since $w\equiv0$ on $\partial B_{3/4}(\bar{X})$ and $v_{\bar{t}}=p-c_{0}\varepsilon<p$ on $\partial B_{3/4}(\bar{X})$. This together with $u\geqslant p$ gives
    	\begin{align*}
    		v_{\bar{t}}<u\quad\text{ on }\partial B_{3/4}(\bar{X}),
    	\end{align*}
        a contradiction to \eqref{Formula: R(12(1))}. Thus $\tilde{X}\notin\partial B_{3/4}(\bar{X})$.
        \item If $\tilde{X}\in A$ where $A$ is the annulus defined in \eqref{Formula: R(8(1(1))}. Observe that  $\mathfrak{L}v_{t}>\varepsilon^{2}$ in $A$ and 
        \begin{align}\label{Formula: R(14)}
        	|\nabla v_{\bar{t}}|\geqslant\left|\pd{v_{\bar{t}}}{y}\right|=\left|1+c_{0}\varepsilon\pd{w}{y}\right|\quad\text{ in }A.
        \end{align}
        Since $w$ is radially symmetric, we have
        \begin{align*}
        	\pd{w}{y}=|\nabla w|(\nu_{X}\cdot e_{2})\quad\text{ in }A,
        \end{align*}
        where $\nu_{X}$ is the unit direction of $X-\bar{X}$. It is easy to read from the definition of $w$ that $|\nabla w|>c$ in $A$ and $\nu_{X}\cdot e_{2}$ is bounded below in $\{v_{\bar{t}}\leqslant0\}\cap A$. Thus, from \eqref{Formula: R(14)} we conclude that $|\nabla v_{\bar{t}}|\geqslant 1+c_{2}\varepsilon$ in $\{v_{\bar{t}}\leqslant0\}\cap A$. For $\varepsilon$ small enough and in view of the third inequality in \eqref{Formula: R(4)}, $|\nabla v_{\bar{t}}|>1+\varepsilon^{2}\geqslant Q(X)$ in $A\cap\varGamma(v_{\bar{t}}>0)$. According to Lemma 2.4 in \cite{S2011}, we conclude that  $\tilde{X}$ cannot belong to $A$.
    \end{enumerate}
    The above discussion implies that $\tilde{X}$ can only belong to $\bar{B}_{1/20}(\bar{X})$. Therefore,
    \begin{align*}
    	u(\tilde{X})=v_{\bar{t}}(\tilde{X})\leqslant p(\tilde{X})+\bar{t}<p(\tilde{X})+c_{0}\varepsilon,
    \end{align*}
    which implies 
    \begin{align*}
    	u(\tilde{X})-p(\tilde{X})<c_{0}\varepsilon,
    \end{align*}
    contradicting with \eqref{Formula: R(11)}. In conclusion, $\bar{t}\geqslant c_{0}\varepsilon$. Using the definition of $v$ in \eqref{Formula: R(12)} we obtain
    \begin{align*}
    	u(X)\geqslant v(X)+\bar{t}\geqslant p(X)+c_{0}\varepsilon w(X),
    \end{align*}
    and hence, since in $\bar{B}_{1/2}\subset B_{3/4}(\bar{X})$ one has $w(X)\geqslant c_{2}$ for some universal constant $c_{2}$, we get
    \begin{align*}
    	u(X)-p(X)\geqslant c\varepsilon\quad\text{ in }\bar{B}_{1/2}(0),
    \end{align*}
    as wanted.
\end{proof}
With the aid of Lemma \ref{Lemma: Harnack inequality} are now in a position to prove Proposition \ref{Proposition: Harnack inequality}.
\begin{proof}
	Without loss of generality, assume that $X_{1}=0$ and $r=1$. In view of
	\begin{align*}
		p(X)^{+}\leqslant u(X)\leqslant(p(X)+\varepsilon)^{+}\quad\text{ in }B_{1}(0),
	\end{align*}
	with $p(X)=y+a_{0}$, we consider three cases. If $|a_{0}|<\tfrac{1}{10}$, then we can apply the previous Lemma \ref{Lemma: Harnack inequality}, and the desired statement follows directly. 
	
	If $a_{0}<-\tfrac{1}{10}$, then $0$ belongs to the zero phase of $(p(X)+\varepsilon)^{+}$, which implies that $0$ also belongs to the zero phase of $u$ for sufficiently small $\varepsilon$, a contradiction. 
	
	If $a_{0}>\tfrac{1}{10}$. Then $B_{1/10}\subset B_{1}\cap\varGamma(u)$, and the conclusion follows by the classic Harnack inequality.
\end{proof}
Let now $r=1$, and let $m>0$ be such that
\begin{align*}
	\frac{1}{2}(1/20)^{m+1}\leqslant\frac{\varepsilon}{\bar{\varepsilon}}<\frac{1}{2}(1/20)^{m},
\end{align*}
where $\varepsilon$ and $\bar{\varepsilon}$ are universal constants in Proposition \ref{Proposition: Harnack inequality}. Let $r_{j}:=\frac{1}{2}(1/20)^{j}$. Then for every $j=0$, $1$, $\dots$, $m$, one has $\varepsilon\leqslant\bar{\varepsilon}r_{j}$. Therefore, we can apply the partial boundary Harnack in $B_{r_{j}}(X_{1})$ repeatedly, and we get that there are $a_{0}\leqslant a_{1}\leqslant\cdots\leqslant a_{j}\leqslant\cdots\leqslant a_{n}\leqslant b_{n}\leqslant\cdots\leqslant b_{j}\leqslant\cdots\leqslant b_{1}\leqslant b_{0}$ so that
\begin{align*}
	|b_{j}-a_{j}|\leqslant(1-c)^{j}|b_{0}-a_{0}|,
\end{align*}
and
\begin{align*}
	(y+a_{j})^{+}\leqslant u(X)\leqslant(y+b_{j})^{+}\quad\text{ in }B_{r_{j}}(X_{1}),
\end{align*}
therefore, $|(u(X)-y)-a_{j}|\leqslant(1-c)^{j}\varepsilon$ for $x\in B_{r_{j}}(X_{0})\cap\overline{\varOmega^{+}(u)}$. The triangle inequality implies that $|\tilde{u}(X)-\tilde{u}(X_{1})|\leqslant2(1-c)^{j}$ for $x\in B_{r_{j}}(X_{0})\cap\overline{\varOmega^{+}(u)}$. Choose $j$ so that $r_{j+1}<|X-X_{1}|\leqslant r_{j}$ and let $\gamma$ be the number which satisfies $(1/20)^{\gamma}=1-c$, we immediately obtain the following result.
\begin{corollary}\label{Corollary: Compactness}
	Let $u$ be a solution to \eqref{Formula: R(3)}-\eqref{Formula: R(4)} satisfying \eqref{Formula: R(4)} for $r=1$. Define
	\begin{align*}
		\tilde{u}_{\varepsilon}(X):=\frac{u(X)-y}{\varepsilon}\quad\text{ in }B_{1}(X_{1})\cap\overline{\{u>0\}},
	\end{align*}
	Then in $B_{1}(X_{1})$, $\tilde{u}_{\varepsilon}$ has a H\"{o}lder modulus of continuity at $X_{1}$, outside the ball of radius $\varepsilon/\bar{\varepsilon}$. That is,
    \begin{align*}
    	|\tilde{u}_{\varepsilon}(X)-\tilde{u}_{\varepsilon}(X_{1})|\leqslant C|X-X_{1}|^{\gamma},
    \end{align*}
    for all $X\in B_{1}(X_{1})\cap\overline{\{u>0\}}$ with $|X-X_{1}|\geqslant\tfrac{\varepsilon}{\bar{\varepsilon}}$.
\end{corollary}
We are now in a position to prove the improvement of flatness, from which Proposition \ref{Proposition: Flatness implies C1alpha} follows by a standard iteration argument. Roughly speaking (please see Fig \ref{Fig: improvement of flatness}), it says that if $u$ is $\varepsilon$-flat in some direction $e_{2}$ for instance, then there must exist a new direction $\nu$, which is close enough to $e_{2}$, so that $u$ is $\sigma\varepsilon$-flat for some $\sigma\in(0,1)$ in a smaller ball.
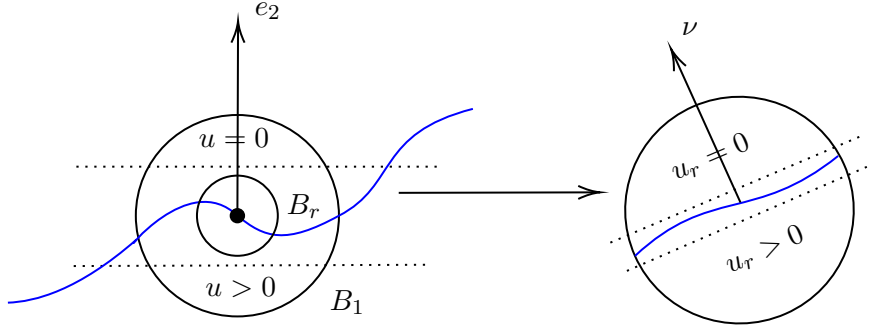
\begin{figure}[ht]

	\tikzset{every picture/.style={line width=0.75pt}} 
	
	\begin{tikzpicture}[x=0.75pt,y=0.75pt,yscale=-0.9,xscale=0.9]
		
		\draw  [line width=0.75]  (122.83,151.5) .. controls (122.83,123.52) and (145.52,100.83) .. (173.5,100.83) .. controls (201.48,100.83) and (224.17,123.52) .. (224.17,151.5) .. controls (224.17,179.48) and (201.48,202.17) .. (173.5,202.17) .. controls (145.52,202.17) and (122.83,179.48) .. (122.83,151.5) -- cycle ;
		\draw [color={rgb, 255:red, 0; green, 0; blue, 255 }  ,draw opacity=1 ][line width=0.75]    (173.5,151.5) .. controls (187.17,162.83) and (197.5,166.25) .. (224.17,151.5) ;
		\draw [color={rgb, 255:red, 0; green, 0; blue, 255 }  ,draw opacity=1 ][line width=0.75]    (123.17,163.5) .. controls (115,173.25) and (149,127.75) .. (173.5,151.5) ;
		\draw [line width=0.75]    (173.5,151.5) ;
		\draw [shift={(173.5,151.5)}, rotate = 0] [color={rgb, 255:red, 0; green, 0; blue, 0 }  ][fill={rgb, 255:red, 0; green, 0; blue, 0 }  ][line width=0.75]      (0, 0) circle [x radius= 3.35, y radius= 3.35]   ;
		\draw  [line width=0.75]  (153.33,151.5) .. controls (153.33,140.36) and (162.36,131.33) .. (173.5,131.33) .. controls (184.64,131.33) and (193.67,140.36) .. (193.67,151.5) .. controls (193.67,162.64) and (184.64,171.67) .. (173.5,171.67) .. controls (162.36,171.67) and (153.33,162.64) .. (153.33,151.5) -- cycle ;
		\draw [line width=0.75]  [dash pattern={on 0.84pt off 2.51pt}]  (93.17,126.83) -- (273.83,126.83) ;
		\draw [line width=0.75]  [dash pattern={on 0.84pt off 2.51pt}]  (91.83,176.83) -- (269.17,176.17) ;
		\draw  [line width=0.75]  (371.88,171.75) .. controls (359.14,142.96) and (372.16,109.3) .. (400.94,96.56) .. controls (429.73,83.82) and (463.39,96.83) .. (476.13,125.62) .. controls (488.87,154.41) and (475.86,188.07) .. (447.07,200.81) .. controls (418.29,213.55) and (384.62,200.54) .. (371.88,171.75) -- cycle ;
		\draw [color={rgb, 255:red, 0; green, 0; blue, 255 }  ,draw opacity=1 ][line width=0.75]    (371.88,171.75) .. controls (409.35,136.47) and (431.23,155.38) .. (473.25,121.43) ;
		\draw [line width=0.75]  [dash pattern={on 0.84pt off 2.51pt}]  (359.58,165.05) -- (484.28,111.62) ;
		\draw [line width=0.75]  [dash pattern={on 0.84pt off 2.51pt}]  (367.85,179.27) -- (488.48,126.51) ;
		\draw [line width=0.75]    (424.5,144.83) -- (390.65,69.33) ;
		\draw [shift={(389.83,67.5)}, rotate = 65.85] [color={rgb, 255:red, 0; green, 0; blue, 0 }  ][line width=0.75]    (10.93,-3.29) .. controls (6.95,-1.4) and (3.31,-0.3) .. (0,0) .. controls (3.31,0.3) and (6.95,1.4) .. (10.93,3.29)   ;
		\draw [line width=0.75]    (173.5,151.5) -- (173.83,55.5) ;
		\draw [shift={(173.83,53.5)}, rotate = 90.19] [color={rgb, 255:red, 0; green, 0; blue, 0 }  ][line width=0.75]    (10.93,-3.29) .. controls (6.95,-1.4) and (3.31,-0.3) .. (0,0) .. controls (3.31,0.3) and (6.95,1.4) .. (10.93,3.29)   ;
		\draw [color={rgb, 255:red, 0; green, 0; blue, 255 }  ,draw opacity=1 ][line width=0.75]    (59,195.25) .. controls (72.5,194.75) and (94,189.25) .. (123.17,163.5) ;
		\draw [line width=0.75]    (254,140) -- (352.5,139.75) ;
		\draw [shift={(354.5,139.75)}, rotate = 179.86] [color={rgb, 255:red, 0; green, 0; blue, 0 }  ][line width=0.75]    (10.93,-3.29) .. controls (6.95,-1.4) and (3.31,-0.3) .. (0,0) .. controls (3.31,0.3) and (6.95,1.4) .. (10.93,3.29)   ;
		\draw [color={rgb, 255:red, 0; green, 0; blue, 255 }  ,draw opacity=1 ][line width=0.75]    (224.17,151.5) .. controls (253.5,136.25) and (242.5,110.25) .. (291,97.75) ;
		
		\draw (218.17,187.23) node [anchor=north west][inner sep=0.75pt]    {$B_{1}$};
		\draw (156.17,181.23) node [anchor=north west][inner sep=0.75pt]    {$u >0$};
		\draw (152.83,105.9) node [anchor=north west][inner sep=0.75pt]    {$u=0$};
		\draw (196.67,139.07) node [anchor=north west][inner sep=0.75pt]    {$B_{r}$};
		\draw (393.83,52.57) node [anchor=north west][inner sep=0.75pt]    {$\nu $};
		\draw (180.67,42.73) node [anchor=north west][inner sep=0.75pt]    {$e_{2}$};
		\draw (413.48,170.67) node [anchor=north west][inner sep=0.75pt]  [rotate=-336.81]  {$u_{r}>0$};
		\draw (386.35,124.28) node [anchor=north west][inner sep=0.75pt]  [rotate=-336.81]  {$u_{r}=0$};
	\end{tikzpicture}
    \caption{Improvement of flatness in the ball $B_{1}$}
    \label{Fig: improvement of flatness}
\end{figure}
\begin{proposition}[Improvement of flatness]\label{Proposition: improvement of flatness}
	Let $u$ be a solution to \eqref{Formula: R(3)}-\eqref{Formula: R(4)} in $B_{1}(0)$ satisfying
	\begin{align*}
		(y-\varepsilon)^{+}\leqslant u(X)\leqslant(y+\varepsilon)^{+}\quad\text{ in }B_{1}(0),
	\end{align*}
    with $0\in\varGamma(u)$. Then if $0<r<r_{0}$ for $r_{0}$ universal, $0<\varepsilon\leqslant\varepsilon_{0}$ depending on $r_{0}$, then 
    \begin{align}\label{Formula: R(16)}
    	(X\cdot\nu-r\tfrac{\varepsilon}{2})^{+}\leqslant u(X)\leqslant(X\cdot\nu+r\tfrac{\varepsilon}{2})^{+}\quad\text{ in }B_{r}(0),
    \end{align}
    with $|\nu|=1$ and $|\nu-e_{2}|\leqslant C_{0}\varepsilon^{2}$ for a universal constant $C_{0}$.
\end{proposition}
\begin{proof}
	The proof follows closely Lemma 4.1 in \cite{S2011}, at which considered the case when $\mathbf{b}\equiv0$, we adapt it to our settings. The proof is divided into three steps.
	
	\textbf{Step I}. \emph{Compactness}. Let $r_{0}$ be fixed which is up to be determined. Assume for the sake of contradiction that there is a sequence $\varepsilon_{k}\to0$ and a sequence $u_{k}$ of solutions to \eqref{Formula: R(3)} in $B_{1}(0)$ with coefficient $a_{ij}^{k}$, $\mathbf{b}^{k}$, right-hand side $\tilde{f}^{k}$, and free boundary condition $Q_{k}$ satisfying \eqref{Formula: R(4)}, so that $u_{k}$ satisfies
    \begin{align}\label{Formula: R(17)}
    	(y-\varepsilon_{k})^{+}\leqslant u_{k}(X)\leqslant(y+\varepsilon_{k})^{+}\quad\text{ in }B_{1}(0)\quad\text{ and }0\in\varGamma(u_{k}),
    \end{align}
    but it does not satisfy the conclusion \eqref{Formula: R(16)} of the Lemma. Set
    \begin{align*}
    	\tilde{u}_{k}(X)=\frac{u_{k}(X)-y}{\varepsilon_{k}}\quad\text{ in }B_{1}(0)\cap\overline{\varOmega^{+}(u_{k})}.
    \end{align*}
    Then, it follows from \eqref{Formula: R(17)} that
    \begin{align}\label{Formula: R(17(1))}
    	-1\leqslant\tilde{u}_{k}\leqslant1\quad\text{ in }B_{1}(0)\cap\overline{\varOmega^{+}(u_{k})}.
    \end{align}
    It then follows from Corollary \ref{Corollary: Compactness} that $\tilde{u}_{k}$ satisfies
    \begin{align*}
    	|\tilde{u}_{k}(X)-\tilde{u}_{k}(Y)|\leqslant C|X-Y|^{\gamma}
    \end{align*}
    for a universal constant $C$ with
    \begin{align*}
    	|X-Y|\geqslant\frac{\varepsilon_{k}}{\bar{\varepsilon}}\quad\text{ in }B_{1/2}(0)\cap\overline{\varOmega^{+}(u_{k})}.
    \end{align*}
    It follows from \eqref{Formula: R(17)} and the H\"{o}lder convergence of $\tilde{u}_{k}$ that $\varGamma(u_{k})$ converges to $B_{1}(0)\cap\{y=0\}$ locally in the Hausdorff distance. This together with Ascoli-Arzela Theorem give that up to a subsequence the graph of  $\tilde{u}_{k}$ over $B_{1/2}(0)\cap\overline{\varOmega^{+}(u_{k})}$ converges in the Hausdorff distance to the graph of a H\"{o}lder continuous function $\tilde{u}$ over $B_{1/2}(0)\cap\{y\geqslant0\}$.
    
    \textbf{Step II}. \emph{Limiting solution}. In this step, we aim to show that $\tilde{u}$ solves
    \begin{align}\label{Formula: R(18)}
    	\begin{cases}
    		\Delta\tilde{u}=0\quad&\text{ in }B_{1/2}(0)\cap\{y>0\},\\
    		\pd{\tilde{u}}{y}=0\quad&\text{ on }B_{1/2}(0)\cap\{y=0\},
    	\end{cases}
    \end{align}
    in the viscosity sense, i.e., if $P(X)$ is a quadratic polynomial touching $\tilde{u}$ at $\bar{X}=(\bar{x},\bar{y})\in B_{1/2}(X_{1})\cap\{y\geqslant0\}$ strictly from below, and 
    \begin{enumerate}
    	\item if $\bar{X}\in B_{1/2}(0)\cap\{y>0\}$, then $\Delta P(\bar{X})\leqslant0$;
    	\item if $\bar{X}\in B_{1/2}(0)\cap\{y=0\}$, then $\pd{P}{y}(\bar{X})\leqslant0$.
    \end{enumerate}
    Since $\tilde{u}_{k}\to\tilde{u}$ in $C^{\gamma}$, there exist points $X_{k}\in B_{1/2}(0)\cap\overline{\{u_{k}>0\}}$, $X_{k}\to\bar{X}$ so that
    \begin{align}\label{Formula: R(19)}
    	P(X_{k})+c_{k}=\tilde{u}_{k}(X_{k}),
    \end{align}
    where $c_{k}\to0$ and
    \begin{align}\label{Formula: R(20)}
    	\tilde{u}_{k}(X)\geqslant P(X)+c_{k},
    \end{align}
    in a small neighborhood of $X_{k}$. Set
    \begin{align*}
    	Q(X)=\varepsilon_{k}(P(X)+c_{k})+y,
    \end{align*}
    It follows from the definition of $\tilde{u}_{k}$, \eqref{Formula: R(19)} and \eqref{Formula: R(20)} read
    \begin{align*}
    	u_{k}(X_{k})=Q(X_{k})
    \end{align*}
    and
    \begin{align*}
    	u_{k}(X)\geqslant Q(X)\quad\text{ in a neighborhood of }X_{k}.
    \end{align*}
    We consider the following two cases.
    \begin{enumerate}
    	\item if $\bar{X}\in B_{1/2}(0)\cap\{x_{d}>0\}$, then $X_{k}\in B_{1/2}(0)\cap\{u_{k}>0\}$ and hence, since $Q$ touches $u_{k}$ from below at $X_{k}$, one has
    	\begin{align*}
    		\sum_{i,j=1}^{2}a_{ij}^{k}(X_{k})D_{ij}Q&=\varepsilon_{k}\sum_{i,j=1}^{2}a_{ij}^{k}(X_{k})D_{ij}P\leqslant\tilde{f}_{k}(X_{k})-\varepsilon_{k}\mathbf{b}^{k}\cdot\nabla P-b_{2}:=\tilde{F}_{k},
    	\end{align*}
        and $|\tilde{F}_{k}|\leqslant c\varepsilon_{k}^{2}$, where we have used the assumption that $\|f\|_{L^{\infty}}\leqslant\varepsilon^{2}$ and $\|\mathbf{b}\|_{L^{\infty}}\leqslant\varepsilon^{2}$. This implies that
        \begin{align*}
        	\sum_{i,j=1}^{2}a_{ij}^{k}(X_{k})D_{ij}P\leqslant c\varepsilon_{k}.
        \end{align*}
        Thus, in view of the last inequality in \eqref{Formula: R(4)},
        \begin{align*}
        	\Delta P=\sum_{i,j=1}^{2}(\delta_{ij}-a_{ij}^{k}(X_{k}))D_{ij}P+\sum_{i,j=1}^{k}a_{ij}^{k}(X_{k})D_{ij}P\leqslant C\varepsilon_{k}.
        \end{align*}
        Passing to the limit as $k\to+\infty$ yields $\Delta P\leqslant0$, as wanted.
        \item If $\bar{X}\in B_{1/2}(0)\cap\{y=0\}$, it follows from Remark [Page 226 in \cite{S2011}] that we can assume that $\Delta P>0$. We assume without loss of generality that $\Delta P>0$ and claim that for $k$ large enough, $X_{k}\in\varGamma(u_{k})$. If this is not the case, then $X_{k_{n}}\in B_{1}(0)\cap\{u_{k_{n}}>0\}$ for a subsequence $k_{n}\to\infty$ and as in the previous case,
        \begin{align*}
        	\Delta P\leqslant C\varepsilon_{k_{n}},
        \end{align*}
        which implies that $\Delta p\leqslant0$ by letting $k_{n}\to\infty$. This leads to a contradiction. Thus, $X_{k}\in\varGamma(u_{k})$ for $k$ large enough. Now notice that
        \begin{align*}
        	\nabla Q=\varepsilon_{k}\nabla P+e_{2},
        \end{align*}
        since $Q$ touches $u_{k}$ from below,
        \begin{align*}
        	|\nabla Q(X_{k})|\leqslant Q_{k}(X_{k})\leqslant1+\varepsilon_{k}^{2},
        \end{align*}
        which yields
        \begin{align*}
        	|\nabla Q(X_{k})|^{2}=\varepsilon_{k}|\nabla P(X_{k})|^{2}+1+2\varepsilon_{k}\pd{P}{y}(X_{k})\leqslant1+3\varepsilon_{k}^{2},
        \end{align*}
        dividing by $\varepsilon_{k}$ gives
        \begin{align*}
        	\varepsilon_{k}|\nabla P(X_{k})|^{2}-3\varepsilon_{k}+2\pd{P}{y}(X_{k})\leqslant0.
        \end{align*}
        Passing to the limit as $k\to\infty$ gives $\pd{P}{y}(\bar{X})\leqslant0$, as desired.
    \end{enumerate}
    \textbf{Step III}. \emph{Improvement of flatness}. It follows from \eqref{Formula: R(18)} and \eqref{Formula: R(17(1))} that
    \begin{align*}
    	-1\leqslant\tilde{u}\leqslant1\quad\text{ in }B_{1/2}(0)\cap\{y\geqslant0\}.
    \end{align*}
    It follows from Lemma 2.6 in \cite{S2011} that for the given $r$,
    \begin{align*}
    	|\tilde{u}(X)-\tilde{u}(0)-\nabla\tilde{u}(0)\cdot X|\leqslant C_{0}r^{2}\quad\text{ in }B_{r}(0)\cap\{y\geqslant0\},
    \end{align*}
    for a constant $C_{0}$. Since $0\in\varGamma(\tilde{u})$ and also $\pd{\tilde{u}}{y}(0)=0$, we have
    \begin{align*}
    	X'\cdot\tilde{\nu}-C_{0}r^{2}\leqslant\tilde{u}\leqslant X'\cdot\tilde{\nu}+C_{0}r^{2}\quad\text{ in }B_{r}(0)\cap\{y\geqslant0\},
    \end{align*}
    with $\tilde{\nu}_{1}=\pd{u}{x}(0)$, $|\tilde{\nu}|\leqslant\tilde{C}$. Thus, for $k$ large enough one has
    \begin{align*}
    	X'\cdot\tilde{\nu}-C_{1}r^{2}\leqslant\tilde{u}_{k}\leqslant X'\cdot\tilde{\nu}+C_{1}r^{2}\quad\text{ in }B_{r}(0)\cap\overline{\{u_{k}>0\}}.
    \end{align*}
    It follows from the definition of $\tilde{u}_{k}$ that
    \begin{align}\label{Formula: R(21)}
    	\varepsilon_{k}X'\cdot\tilde{\nu}+y-\varepsilon_{k}C_{1}r^{2}\leqslant u_{k}\leqslant\varepsilon_{k}X'\cdot\tilde{\nu}+y+\varepsilon_{k}C_{1}r^{2}\quad\text{ in }B_{r}(0)\cap\overline{\{u_{k}>0\}}.
    \end{align}
    Set
    \begin{align*}
    	\nu=\left(\frac{\varepsilon_{k}\tfrac{\partial u}{\partial x}(0)}{\sqrt{\varepsilon_{k}^{2}+1}},\frac{1}{\sqrt{\varepsilon_{k}^{2}+1}}\right).
    \end{align*}
    Since, for $k$ large, $1\leqslant\sqrt{\varepsilon_{k}^{2}+1}\leqslant 1+\tfrac{\varepsilon_{k}^{2}}{2}$, we deduce from \eqref{Formula: R(21)} that
    \begin{align*}
    	X\cdot\nu-\varepsilon_{k}^{2}\frac{r}{2}-C_{1}r^{2}\varepsilon_{k}\leqslant u_{k}\leqslant X\cdot\nu+\varepsilon_{k}^{2}\frac{r}{2}+C_{1}r^{2}\varepsilon_{k}^{2}\quad\text{ in }B_{r}\cap\overline{\{u_{k}>0\}}.
    \end{align*}
    If in particular we  choose $r_{0}$ so that $C_{1}r_{0}\leqslant\tfrac{1}{4}$, then $\varepsilon_{k}\leqslant\tfrac{1}{2}$ for $k$ large, one has
    \begin{align*}
    	X\cdot\nu-\varepsilon_{k}\frac{r}{2}\leqslant u_{k}\leqslant X\cdot\nu+\varepsilon_{k}\frac{r}{2}\quad\text{ in }B_{r}\cap\overline{\{u_{k}>0\}},
    \end{align*}
    which together with \eqref{Formula: R(17)} implies that
    \begin{align*}
    	(X\cdot\nu-\varepsilon_{k}\tfrac{r}{2})^{+}\leqslant u_{k}\leqslant(X\cdot\nu+\varepsilon_{k}\tfrac{r}{2})^{+}\quad\text{ in }B_{r}(0).
    \end{align*}
    Thus, $u_{k}$ satisfy \eqref{Formula: R(16)}, and this is a contradiction.
\end{proof}
To prove Lemma \ref{Lemma: Flatness implies C1gamma}, we need the following iteration Lemma.
\begin{lemma}[Iteration]\label{Lemma: iteration}
	Assume that $u$ satisfies the assumption of Lemma \ref{Lemma: Flatness implies C1gamma}, then there exists $\bar{r}$ and a sequence of unit vectors $\{\nu_{k}\}_{k=1}^{\infty}$ so that
	\begin{align}\label{Formula: R(21(2))}
		\left(X\cdot\nu_{k}-\varepsilon_{k}\right)^{+}\leqslant u(X)\leqslant\left(X\cdot\nu_{k}+\varepsilon_{k}\right)^{+}\quad\text{ in }B_{\bar{r}^{k}}(0),
	\end{align}
    with $\varepsilon_{k}:=2^{-k}\varepsilon_{0}(\bar{r})^{2}$ and  $|\nu_{k}-\nu_{k-1}|\leqslant C_{0}\varepsilon_{k-1}$. Here $C_{0}$ and $\varepsilon_{0}$ are universal constants in Proposition \ref{Proposition: improvement of flatness}.
\end{lemma}
\begin{proof}
	Let $u$ be a viscosity solution to \eqref{Formula: R(2)} in $B_{1}$ with $0\in\varGamma(u)$, $Q(0)=1$ and $a_{ij}(0)=\delta_{ij}$. Fix a universal constant $\bar{r}>0$ such that
	\begin{align*}
		\bar{r}\leqslant\min\{r_{0},(\tfrac{1}{4})^{1/\beta}\},
	\end{align*}
    where $r_{0}$ is the universal constant in the improvement of flatness Proposition \ref{Proposition: improvement of flatness}. Also, let us fix a universal constant $\tilde{\varepsilon}>0$ such that
    \begin{align*}
    	\tilde{\varepsilon}\leqslant\min\{\varepsilon_{0}(\bar{r})^{2},C_{0}^{-1}\},
    \end{align*}
    where $\varepsilon_{0}$ and $C_{0}$ are constants in Proposition \ref{Proposition: improvement of flatness}. Let now $\bar{\varepsilon}=\tilde{\varepsilon}^{2}$. By our choice of $\tilde{\varepsilon}$, one has
    \begin{align*}
    	(y-\tilde{\varepsilon})^{+}\leqslant u(x)\leqslant(y+\tilde{\varepsilon})^{+}\quad\text{ in }B_{1}(0).
    \end{align*}
    On the other hand, we see that under the assumption of Lemma \ref{Lemma: Flatness implies C1gamma} and the choice of $\tilde{\varepsilon}$, one has
    \begin{align*}
    	&\|\tilde{f}\|_{L^{\infty}(B_{1})}\leqslant\bar{\varepsilon}\leqslant\tilde{\varepsilon}^{2},\qquad\|\mathbf{b}\|_{L^{\infty}(B_{1})}\leqslant\bar{\varepsilon}\leqslant\tilde{\varepsilon}^{2}\\
    	&\|Q(X)-1\|_{L^{\infty}(B_{1})}\leqslant[Q]_{C^{0,\beta}}\leqslant\bar{\varepsilon}\leqslant\tilde{\varepsilon}^{2},\\
    	&\|a_{ij}-\delta_{ij}\|_{L^{\infty}(B_{1})}\leqslant[a_{ij}]_{C^{0,\beta}}\leqslant\bar{\varepsilon}\leqslant\tilde{\varepsilon}^{2}.
    \end{align*}
    Thus, one can conclude from Proposition \ref{Proposition: improvement of flatness} that
    \begin{align*}
    	\left(X\cdot\nu_{1}-\bar{r}\tfrac{\tilde{\varepsilon}}{2}\right)^{+}\leqslant u(X)\leqslant\left(X\cdot\nu_{1}+\bar{r}\tfrac{\tilde{\varepsilon}}{2}\right)^{+}\quad\text{ in }B_{\bar{r}}(0),
    \end{align*}
    with $|\nu_{1}|=1$ and $|\nu_{1}-e_{2}|\leqslant C_{0}\tilde{\varepsilon}$. We therefore rescale and iterate the argument above. For $k=0$, $1$, $\ldots$, set
	\begin{align*}
		\rho_{k}=\bar{r}^{k},\qquad \varepsilon_{k}=2^{-k}\varepsilon_{0}(\bar{r})^{2},\qquad u_{k}(X):=\frac{u(\rho_{k}X)}{\rho_{k}},\qquad\mathbf{b}_{k}(X)=\rho_{k}\mathbf{b}(\rho_{k}X),\qquad\tilde{f}_{k}=\rho_{k}\tilde{f}(\rho_{k}X),
	\end{align*}
    Notice that each $u_{k}$ satisfies
    \begin{align*}
    	\begin{cases}
    		\begin{alignedat}{2}
    			\displaystyle\sum_{i,j=1}^{2}a_{ij}(\rho_{k}X)D_{ij}u_{k}+\rho_{k}\mathbf{b}(\rho_{k}X)\nabla u_{k}&=\rho_{k}\tilde{f}(\rho_{k}X)\quad&&\text{ in }B_{1}(0)\cap\{u_{k}>0\},\\
    			|\nabla u_{k}|&=Q(\rho_{k}X)\quad&&\text{ on }B_{1}(0)\cap\partial\{u_{k}>0\}.
    		\end{alignedat}
    	\end{cases}
    \end{align*}
    Note that for $k$ large enough, the assumption of Proposition \ref{Proposition: improvement of flatness} is satisfied in $B_{1}(0)$. Indeed,
    \begin{align*}
    	|\rho_{k}\tilde{f}(\rho_{k}X)|&\leqslant\|f\|_{L^{\infty}}\rho_{k}\leqslant\bar{\varepsilon}\bar{r}^{k}\leqslant\varepsilon_{k}^{2},\\
    	|\rho_{k}\mathbf{b}(\rho_{k}X)\|_{L^{\infty}}&\leqslant\|\mathbf{b}\|_{L^{\infty}}\rho_{k}\leqslant\bar{\varepsilon}\bar{r}^{k}\leqslant\varepsilon_{k}^{2},\\
    	|Q(\rho_{k}X)-Q(0)|&\leqslant[Q]_{C^{0,\beta}}\rho_{k}^{\beta}\leqslant\bar{\varepsilon}\bar{r}^{k\beta}\leqslant\varepsilon_{k}^{2},
    \end{align*}
    and
    \begin{align*}
    	|a_{ij}(\rho_{k}X)-a_{ij}(0)|&\leqslant[a_{ij}]_{C^{0,\beta}}\rho_{k}^{\beta}\leqslant\bar{\varepsilon}\bar{r}^{k\beta}\leqslant\varepsilon_{k},
    \end{align*}
    with
    \begin{align*}
    	\left(X\cdot\nu_{k}-\varepsilon_{k}\right)^{+}\leqslant u(X)\leqslant\left(X\cdot\nu_{k}+\varepsilon_{k}\right)^{+}\quad\text{ in }B_{\bar{r}^{k}}(0),
    \end{align*}
    as desired.
\end{proof}
We now prove Proposition \ref{Proposition: Flatness implies C1alpha}.
\begin{proof}
	Notice that \eqref{Formula: R(21(2))} implies that $\varGamma(u)$ is $C^{1,\gamma}$ at $0$. Then repeating the procedure for points in a small neighborhood of $0$, we have that there exists vector $\nu_{\infty}=\lim_{k\to\infty}\nu_{k}$, $C>0$ and $\gamma'\in(0,1)$ so that up to a rotation, in the new coordinate system $e_{1}$, $\nu_{\infty}$, one has $y=h(x)$, $h(0)=0$ and
	\begin{align*}
		|h(x)-\nu_{\infty}\cdot X|\leqslant C|X|^{1+\gamma'}
	\end{align*}
    in a small neighborhood of $0$.
\end{proof}
In fact, as a direct corollary of Proposition \ref{Proposition: Flatness implies C1alpha} and can be proved via a standard blow-up argument, we can prove that Lipschitz free boundaries are $C^{1,\gamma}$. The argument is standard, so we give it into the appendix B for the sake of completeness.

With Proposition \ref{Proposition: Flatness implies C1alpha} at hand, we are now in a position to study the regularity of the non-degenerate points. We begin with the following observation.
\begin{lemma}[Uniqueness of blow-up limits]\label{Lemma: Uniqueness of blow-up limits}
	Let $X_{0}=(x_{0},y_{0})\in N_{\psi}$, let $\psi$ be a local minimizer of $J$ in $B_{R_{0}}(X_{0})$, and let $\psi_{X_{0},r}$ be the rescaling defined by
	\begin{align*}
		\psi_{X_{0},r}(X):=\frac{\psi(X_{0}+rX)}{r}\quad X\in B_{1}(0).
	\end{align*}
    Then there is a unique vector $\nu_{X_{0}}\in\partial B_{1}(0)$ so that
    \begin{align}\label{Formula: R(21(3))}
    	\|\psi_{X_{0},r}-\psi_{0}\|_{L^{\infty}(B_{1})}\leqslant C_{1}r^{\gamma}\quad\text{ for every }r\leqslant\frac{1}{2},
    \end{align}
    where
    \begin{align}\label{Formula: R(21(4))}
    	\psi_{0}(X):=x_{0}\sqrt{-y_{0}}(\nu_{X_{0}}\cdot X)^{+}\quad\text{ for every }X=(x,y)\in\mathbb{R}^{2}.
    \end{align}
    Here $C_{1}$ and $\gamma$ are positive constants depending on $X_{0}$, $r_{0}$, $\varepsilon_{0}$ and $C_{0}$ (recall Proposition \ref{Proposition: improvement of flatness}).
\end{lemma}
\begin{proof}
	\textbf{Step I}. Since $X_{0}\in N_{\psi}$, it follows from the third property in Corollary \ref{Corollary: Properties of blow-up limits} that there exists a vanishing sequence $\rho_{k}\to0$ so that the rescaling $\psi_{k}(X):=\frac{\psi(X_{0}+\rho_{k} X)}{\rho_{k}}$ converges uniformly in $B_{1}(0)$ to a function $\psi_{0}$ of the form
    \begin{align*}
    	\psi_{0}(X)=x_{0}\sqrt{-y_{0}}(X\cdot\nu)^{+},
    \end{align*}
    for some $\nu\in\mathbb{R}^{2}$. Since $\psi_{k}$ converges to $\psi_{0}$ uniformly, there exists $\varepsilon>0$ so that for $k$ large enough
    \begin{align*}
    	\|\psi_{k}-\psi_{0}\|_{L^{\infty}(B_{1})}<\varepsilon,
    \end{align*}
    this implies that
    \begin{align}\label{Formula: R(22)}
    	\{\psi_{k}>0\}\quad\text{ in }\{X\cdot\nu>\varepsilon\}\qquad\text{ and }\qquad\psi_{k}=0\quad\text{ in }\{X\cdot\nu<-\varepsilon\}.
    \end{align}
    Namely, $\psi_{k}$ satisfies
    \begin{align*}
    	x_{0}\sqrt{-y_{0}}(X\cdot\nu-2\varepsilon)^{+}\leqslant\psi_{k}\leqslant x_{0}\sqrt{-y_{0}}(X\cdot\nu+2\varepsilon)^{+},
    \end{align*}
    for some $\nu\in\partial B_{1}(0)$. Let now $r\in(0,1)$ be arbitrary, then there exists $k\in\mathbb{N}$ so that
    \begin{align*}
    	\rho_{k+1}\leqslant r\leqslant\rho_{k}
    \end{align*}
    Therefore, one can choose $\varrho\in(0,1]$ so that $r=\varrho\rho_{k}$. It follows from \eqref{Formula: R(22)} that $\psi_{X_{0},r}=\psi_{\varrho\rho_{k}}$ satisfies
    \begin{align*}
    	x_{0}\sqrt{-y_{0}}\left(X\cdot\nu-\frac{2\varepsilon}{\varrho}\right)^{+}\leqslant\psi_{X_{0},r}\leqslant x_{0}\sqrt{-y_{0}}\left(X\cdot\nu+\frac{2\varepsilon}{\varrho}\right)^{+}\quad\text{ in }B_{1}(0).
    \end{align*}
    For $\varepsilon$ small enough, we can assume that $0<\frac{2\varepsilon}{\varrho}\leqslant\bar{\varepsilon}$, where $\bar{\varepsilon}$ is determined in the previous Lemma \ref{Lemma: iteration}. Thus, for every $r$ small enough, $\psi_{X_{0},r}$ satisfies
    \begin{align*}
    	x_{0}\sqrt{-y_{0}}(X\cdot\nu-\bar{\varepsilon})^{+}\leqslant\psi_{X_{0},r}\leqslant x_{0}\sqrt{-y_{0}}(X\cdot\nu+\bar{\varepsilon})^{+}\quad\text{ in }B_{1}(0),
    \end{align*}
    for some $\nu\in\partial B_{1}(0)$, thus $\psi_{X_{0},r}$ satisfies the flatness assumption in Lemma \ref{Lemma: Flatness implies C1gamma}.
    
    \textbf{Step II}. Since $\psi$ minimizes $J$ in $B_{R_{0}}(X_{0})$ with $X_{0}\in N_{\psi}$, then $\psi_{X_{0},r}$ is a viscosity solution to the problem
    \begin{align*}
    	\begin{cases}
    		\begin{alignedat}{2}
    			\Delta\psi_{X_{0},r}-\frac{r}{x_{0}+rx}\pd{\psi_{X_{0},r}}{x}&=r\tilde{f}(X_{0}+rX)\quad&&\text{ in }B_{1}(0)\cap\{\psi_{r}>0\},\\
    			|\nabla\psi_{X_{0},r}|&=(x_{0}+rx)\sqrt{-y_{0}-rx}\quad&&\text{ on }B_{1}(0)\cap\partial\{\psi_{r}>0\}.
    		\end{alignedat}
    	\end{cases}
    \end{align*}
    Notice that for $r$ small enough, one has $|\tfrac{r}{x_{0}+rx}|\leqslant\bar{\varepsilon}$ and $|r\tilde{f}(X_{0}+rX)|\leqslant\bar{\varepsilon}$, where $\bar{\varepsilon}$ is defined in Lemma \ref{Lemma: iteration}. Then it follows that the assumption of Lemma \ref{Lemma: Flatness implies C1gamma} is satisfied for $\psi_{X_{0},r}$. Combined with step I, we apply Lemma \ref{Lemma: iteration} that there exist $\bar{r}>0$ and a sequence of unit vectors $\{\nu_{k}\}_{k=1}^{\infty}$ so that
    \begin{align}\label{Formula: R(23)}
    	x_{0}\sqrt{-y_{0}}\left(X\cdot\nu_{k}-\varepsilon_{k}\right)^{+}\leqslant\psi_{X_{0},r}(X)\leqslant x_{0}\sqrt{-y_{0}}\left(X\cdot\nu_{k}+\varepsilon_{k}\right)^{+}\quad\text{ in }B_{1}(0),
    \end{align}
    with
    \begin{align}\label{Formula: R(24)}
    	|\nu_{k}-\nu_{k-1}|\leqslant C_{0}\varepsilon_{k}.
    \end{align}
    This implies that $\{\nu_{k}\}$ is a Cauchy sequence so there is a unit vector $\nu_{X_{0}}$ such that
    \begin{align}\label{Formula: R(25)}
    	\nu_{X_{0}}=\lim_{k\to\infty}\nu_{k}.
    \end{align}
    It follows from the definition of $\varepsilon_{k}$ that there exists $\sigma\in(0,1)$ so that  $\varepsilon_{k}=\tfrac{\varepsilon_{0}(\bar{r})^{2}}{2^{k}}\leqslant\varepsilon_{0}(\bar{r})^{2}\sigma^{k}<1$. Thus it follows from \eqref{Formula: R(24)} and \eqref{Formula: R(25)} that
    \begin{align*}
    	|\nu_{X_{0}}-\nu_{k}|\leqslant\sum_{k=2}^{\infty}|\nu_{k}-\nu_{k-1}|\leqslant\frac{C_{0}\varepsilon_{0}(\bar{r})^{2}}{1-\sigma}\sigma^{k}.
    \end{align*}
    Define now $\psi_{0}=x_{0}\sqrt{-y_{0}}(\nu_{X_{0}}\cdot X)^{+}$ and it follows from \eqref{Formula: R(23)} that
    \begin{align}\label{Formula: R(26)}
    	|\psi_{0}-\psi_{X_{0},r}|\leqslant x_{0}\sqrt{-y_{0}}\left(\left|(X\cdot\nu_{X_{0}})^{+}-(X\cdot\nu_{k})^{+}\right|+\left|\frac{\varepsilon_{0}(\bar{r})^{2}}{2^{k}}\right|\right)\leqslant C_{0}\varepsilon_{0}(\bar{r})^{2}\left(1+\frac{C_{0}}{1-\sigma}\right)\sigma^{k}.
    \end{align}
    Notice that for $k$ large, $(\bar{r})^{k+1}\leqslant2r$. Choose $\gamma$ so that $(\bar{r})^{\gamma}=\sigma$, this gives
    \begin{align*}
    	\sigma^{k}=(\bar{r})^{k\gamma}=\frac{1}{(\bar{r})^{\gamma}}((\bar{r})^{k+1})^{\gamma}\leqslant\frac{2^{\gamma}}{(\bar{r})^{\gamma}}r^{\gamma}.
    \end{align*}
    This together with \eqref{Formula: R(26)} gives the desired result.
\end{proof}
As a direct application of the uniqueness of blow-up limit, we have
\begin{corollary}\label{Corollary: free boundary Lipschitz}
	Let $X_{0}\in N_{\psi}$ and let $\psi$ be a local minimizer of $J$ in $B_{R_{0}}(X_{0})$, then the free boundary $B_{R_{0}}(X_{0})\cap\partial\{\psi>0\}$ is the graph of a Lipschitz function in a small neighborhood of $X_{0}$.
\end{corollary}
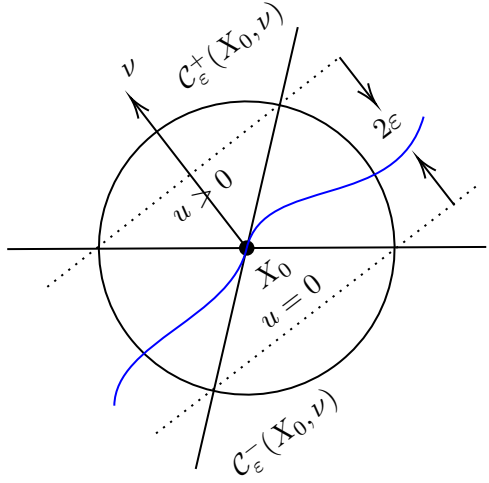
\begin{figure}[ht]
	\tikzset{every picture/.style={line width=0.75pt}} 
	
	\begin{tikzpicture}[x=0.75pt,y=0.75pt,yscale=-0.9,xscale=0.9]
		
		\draw  [line width=0.75]  (207.35,176.44) .. controls (182.49,144.18) and (188.5,97.87) .. (220.76,73.01) .. controls (253.03,48.15) and (299.34,54.16) .. (324.2,86.43) .. controls (349.05,118.69) and (343.05,165) .. (310.78,189.86) .. controls (278.51,214.72) and (232.21,208.71) .. (207.35,176.44) -- cycle ;
		\draw [line width=0.75]    (146.27,131.98) -- (385.27,130.89) ;
		\draw [line width=0.75]    (291.24,20.3) -- (240.31,242.57) ;
		\draw [line width=0.75]  [dash pattern={on 0.84pt off 2.51pt}]  (166.27,150.97) -- (312.09,35.48) ;
		\draw [line width=0.75]  [dash pattern={on 0.84pt off 2.51pt}]  (220.52,224.68) -- (380.73,100) ;
		\draw [line width=0.75]    (265.77,131.43) ;
		\draw [shift={(265.77,131.43)}, rotate = 0] [color={rgb, 255:red, 0; green, 0; blue, 0 }  ][fill={rgb, 255:red, 0; green, 0; blue, 0 }  ][line width=0.75]      (0, 0) circle [x radius= 3.35, y radius= 3.35]   ;
		\draw [line width=0.75]    (369.45,109.94) -- (352.06,87.37) ;
		\draw [shift={(350.84,85.78)}, rotate = 52.39] [color={rgb, 255:red, 0; green, 0; blue, 0 }  ][line width=0.75]    (10.93,-3.29) .. controls (6.95,-1.4) and (3.31,-0.3) .. (0,0) .. controls (3.31,0.3) and (6.95,1.4) .. (10.93,3.29)   ;
		\draw [line width=0.75]    (312.09,35.48) -- (329.17,57.66) ;
		\draw [shift={(330.39,59.24)}, rotate = 232.39] [color={rgb, 255:red, 0; green, 0; blue, 0 }  ][line width=0.75]    (10.93,-3.29) .. controls (6.95,-1.4) and (3.31,-0.3) .. (0,0) .. controls (3.31,0.3) and (6.95,1.4) .. (10.93,3.29)   ;
		\draw [line width=0.75]    (265.77,131.43) -- (208.9,57.22) ;
		\draw [shift={(207.69,55.63)}, rotate = 52.54] [color={rgb, 255:red, 0; green, 0; blue, 0 }  ][line width=0.75]    (10.93,-3.29) .. controls (6.95,-1.4) and (3.31,-0.3) .. (0,0) .. controls (3.31,0.3) and (6.95,1.4) .. (10.93,3.29)   ;
		\draw [color={rgb, 255:red, 0; green, 0; blue, 255 }  ,draw opacity=1 ][line width=0.75]    (199.69,210.74) .. controls (200.28,178.73) and (257.83,166.59) .. (265.77,131.43) ;
		\draw [color={rgb, 255:red, 0; green, 0; blue, 255 }  ,draw opacity=1 ][line width=0.75]    (265.77,131.43) .. controls (274.69,94.27) and (340.65,113.52) .. (354.03,65.34) ;
		
		\draw (226.3,42.11) node [anchor=north west][inner sep=0.75pt]  [rotate=-322.39]  {$\mathcal{C}_{\varepsilon }^{+}( X_{0} ,\nu )$};
		\draw (253.43,237.07) node [anchor=north west][inner sep=0.75pt]  [rotate=-322.39]  {$\mathcal{C}_{\varepsilon }^{-}( X_{0} ,\nu )$};
		\draw (265.43,144.51) node [anchor=north west][inner sep=0.75pt]  [rotate=-322.39]  {$X_{0}$};
		\draw (325.46,69.23) node [anchor=north west][inner sep=0.75pt]  [rotate=-322.39]  {$2\varepsilon $};
		\draw (224.4,111.74) node [anchor=north west][inner sep=0.75pt]  [rotate=-322.39]  {$u >0$};
		\draw (268.49,164.87) node [anchor=north west][inner sep=0.75pt]  [rotate=-322.39]  {$u=0$};
		\draw (200.48,40.54) node [anchor=north west][inner sep=0.75pt]  [rotate=-322.39]  {$\nu $};
	\end{tikzpicture}
    \caption{$\mathcal{C}_{\varepsilon}^{+}(X_{0},\nu_{X_{0}})$ and $\mathcal{C}_{\varepsilon}^{-}(X_{0},\nu_{X_{0}})$}
    \label{Fig: cones}
\end{figure}
\begin{proof}
	The proof is based on the Lipschitz regularity of $\psi$ at $X_{0}$, the non-degenerate of $\psi$ near $X_{0}$ and the uniqueness of blow-up limit (Lemma \ref{Lemma: Uniqueness of blow-up limits}), and the idea is mainly borrowed from Proposition 8.6 in \cite{V2019}. We divide the proof into three steps.
	
	\textbf{Step I}. In this step, we show that for every $\varepsilon>0$, there exists $R>0$ so that
	\begin{align}\label{Formula: R(27)}
		\begin{cases}
			\psi>0&\quad\text{ in }\mathcal{C}_{\varepsilon}^{+}(X_{0},\nu_{X_{0}})\cap B_{R}(X_{0}),\\
			\psi=0&\quad\text{ in }\mathcal{C}_{\varepsilon}^{-}(X_{0},\nu_{X_{0}})\cap B_{R}(X_{0}),
		\end{cases}
	\end{align}
    where $\mathcal{C}_{\varepsilon}^{\pm}(X_{0},\nu)$ are the cones (please see Fig \ref{Fig: cones}) defined by
    \begin{align*}
    	\mathcal{C}_{\varepsilon}^{\pm}(X_{0},\nu)=\{X=(x,y)\in\mathbb{R}^{2}\colon\pm\nu\cdot(X-X_{0})>\varepsilon|X-X_{0}|\},\quad\varepsilon>0.
    \end{align*}
    To this end, let $\gamma$ and $C_{1}$ be defined as in Lemma \ref{Lemma: Uniqueness of blow-up limits}, then it follows from \eqref{Formula: R(21(3))} and \eqref{Formula: R(21(4))} that
    \begin{align*}
    	\psi_{X_{0},r}\geqslant(X\cdot\nu_{X_{0}}-C_{1}r^{\gamma})^{+}\quad\text{ in }B_{1}(0).
    \end{align*}
    Choose $C\geqslant C_{1}$ be any positive constant, we have
    \begin{align}\label{Formula: R(28)}
    	\{X\cdot\nu_{X_{0}}>Cr^{\gamma}\colon X\in B_{1}(0)\}\subset\{\psi_{X_{0},r}>0\}\cap B_{1}(0).
    \end{align}
    On the other hand, we claim
    \begin{align}\label{Formula: R(29)}
    	\{\psi_{X_{0},r}>0\}\cap\{X\cdot\nu_{X_{0}}<-Cr^{\gamma}\}=\varnothing.
    \end{align}
    We suppose that there is a point $X'\in B_{1}(0)$ so that
    \begin{align*}
    	\psi_{X_{0},r}(X')>0\qquad\text{ and }\qquad X'\cdot\nu_{X_{0}}<-Cr^{\gamma}.
    \end{align*}
    This implies that $\tfrac{X'}{2}\in B_{1/2}$ satisfies
    \begin{align*}
    	\psi_{X_{0},2r}(\tfrac{X'}{2})>0\qquad\text{ and }\qquad(\tfrac{X'}{2})\cdot\nu_{X_{0}}<-\frac{1}{2}Cr^{\gamma}.
    \end{align*}
    Since $\psi_{X_{0},2r}$ is also a local minimizer of $J$, we have that
    \begin{align*}
    	\left\Vert\psi_{X_{0},2r}\right\Vert_{L^{\infty}(B_{\rho}(X'/2))}\geqslant\eta\rho\qquad\text{ where }\qquad\rho:=\frac{1}{2}Cr^{\gamma}.
    \end{align*}
    Notice that $\psi_{0}=0$ on $B_{\rho}(\tfrac{X'}{2})$. On the other hand, choosing $r_{0}$ so that
    \begin{align*}
    	Cr_{0}^{\gamma}\leqslant1,
    \end{align*}
    one has $\rho\leqslant\tfrac{1}{2}$ and thus $B_{\rho}(\tfrac{X'}{2})\subset B_{1}$ and
    \begin{align*}
    	\frac{\eta}{2}Cr^{\gamma}\leqslant\|\psi_{X_{0},2r}-\psi_{0}\|_{L^{\infty}(B_{1})}\leqslant C_{1}(2r)^{\gamma},
    \end{align*}
    which yields a contradiction, provided that $C\geqslant\tfrac{2}{\eta}C_{1}$. Therefore the estimates \eqref{Formula: R(28)} and \eqref{Formula: R(29)} imply \eqref{Formula: R(27)} by taking $R>0$ so that $CR^{\gamma}\leqslant\varepsilon$.
    
    \textbf{Step II}. In this step, we show that the free boundary near the origin can be expressed by the graph of a function $g$. Precisely, there exists $\delta>0$ such that
    \begin{align*}
    	(-\delta,\delta)\times(-\delta,\delta)\cap\{\psi>0\}=\{(x,t)\in(-\delta,\delta)\times(-\delta,\delta)\colon t=g(x)\}.
    \end{align*}
    Without loss of generality, we assume that $\nu_{X_{0}}=e_{2}$ and thus $\psi_{0}=x_{0}\sqrt{-y_{0}}y^{+}$ is the unique blow-up limit of $\psi$ at $X_{0}$. Thus $\{\psi_{0}>0\}=\{y>0\}$. Let now $\varepsilon\in(0,1)$ and $R$ be as in \eqref{Formula: R(27)}, it follows that
    \begin{align*}
    	\begin{cases}
    		\begin{alignedat}{2}
    			\psi&>0\quad&&\text{ in }\mathcal{C}_{\varepsilon}^{+}(X_{0},e_{d})\cap B_{R}(X_{0}),\\
    			\psi&=0\quad&&\text{ in }\mathcal{C}_{\varepsilon}^{-}(X_{0},e_{d})\cap B_{R}(X_{0}),
    		\end{alignedat}
    	\end{cases}
    \end{align*}
    Let $x_{1}\in(-\delta,\delta)$ where $\delta\leqslant R\sqrt{1-\varepsilon^{2}}$, it follows that the segment $\{(x_{1},t)\colon t>\varepsilon R\}$ is contained in $\{\psi>0\}$ and the segment $\{(x_{1},t)\colon t<-\varepsilon R\}$ is contained in $\{\psi=0\}$. Therefore, $g(x_{1}):=\inf\{t\colon\psi(x_{1},t)>0\}$ is well defined and $X_{1}:=(x_{1},g(x_{1}))\in\partial\{\psi>0\}\cap B_{R}(X_{0})$. Moreover, one has
    \begin{align*}
    	-\varepsilon|x_{1}|\leqslant g(x_{1})\leqslant\varepsilon|x_{1}|,
    \end{align*}
    which implies that $|X_{1}|\leqslant|x_{1}|\sqrt{1+\varepsilon^{2}}\leqslant\sqrt{2}\delta$. In order to finish the claim, it suffices to show that if $\delta>0$ is small enough, then
    \begin{align}\label{Formula: R(30)}
    	\begin{cases}
    		\begin{alignedat}{2}
    			\psi&>0\quad&&\text{ in }\mathcal{C}_{2\varepsilon}^{+}(X_{1},e_{d})\cap B_{R}(X_{1}),\\
    		    \psi&=0\quad&&\text{ in }\mathcal{C}_{2\varepsilon}^{-}(X_{1},e_{d})\cap B_{R}(X_{1}).
    		\end{alignedat}
    	\end{cases}
    \end{align}
    Indeed, as a direct consequence of \eqref{Formula: R(30)}, one can easily deduce that
    \begin{align*}
    	&((-\delta,\delta)\times(-\delta,\delta))\cap\{\psi>0\}=\{(x',t)\in(-\delta,\delta)\times(-\delta,\delta)\colon t>g(x')\},\\
    	&((-\delta,\delta)\times(-\delta,\delta))\cap\partial\{\psi>0\}=\{(x',t)\in(-\delta,\delta)\times(-\delta,\delta)\colon t\leqslant g(x')\}.
    \end{align*}
    We now prove \eqref{Formula: R(30)} to conclude this step. Applying \eqref{Formula: R(27)} for the free boundary point $X_{1}$, we have
    \begin{align*}
    	\begin{cases}
    		\begin{alignedat}{2}
    			\psi&>0\quad&&\text{ in }\mathcal{C}_{\varepsilon}^{+}(X_{1},\nu_{X_{1}})\cap B_{R}(X_{1})\\
    			\psi&=0\quad&&\text{ in }\mathcal{C}_{\varepsilon}^{-}(X_{1},\nu_{X_{1}})\cap B_{R}(X_{1}).
    		\end{alignedat}
    	\end{cases}
    \end{align*}
    In view of \eqref{Formula: R(30)}, we only need to show $\mathcal{C}_{2\varepsilon}^{\pm}(X_{1},e_{d})\subset\mathcal{C}_{\varepsilon}^{\pm}(X_{1},\nu_{X_{1}})$. Since $\psi$ is a Lipschitz function in $B_{1}$ and $0$, $X_{1}\in\partial\{\psi>0\}\cap B_{R}(X_{0})$, let $\alpha:=\frac{\gamma}{1+\gamma}$ and $r:=|X_{1}-X_{0}|^{1-\alpha}$,
    \begin{align*}
    	|\psi_{X_{0},r}-\psi_{X_{1},r}|=\frac{1}{r}|\psi(X_{0}+rX)-\psi(X_{1}+rX)|\leqslant\frac{L}{r}|X_{1}-X_{0}|:=L|X_{1}-X_{0}|^{\alpha},
    \end{align*}
    where $L$ is the Lipschitz constant of $\psi$ in $B_{1}(0)$. This implies $\|\psi_{X_{0},r}-\psi_{X_{1},r}\|_{L^{\infty}(B_{1})}\leqslant L|X_{1}-X_{0}|^{\alpha}$. On the other hand, if we denote $\psi_{X_{0}}$ and $\psi_{X_{1}}$ the unique blow-up limit of $\psi$ at $X_{0}$ and $X_{1}$ respectively. It follows from \eqref{Formula: R(21(3))} that
    \begin{align*}
    	\|\psi_{X_{0}}-\psi_{X_{0},r}\|_{L^{\infty}(B_{1})}\leqslant C_{1}r^{\gamma}\quad\text{ and }\quad\|\psi_{X_{1}}-\psi_{X_{1},r}\|_{L^{\infty}(B_{1})}\leqslant C_{1}r^{\gamma}.
    \end{align*}
    Choose $R$ such that $(2R)^{1-\alpha}\leqslant\tfrac{1}{2}$, we have by $r^{\gamma}=|X_{1}-X_{0}|^{\alpha}$ that
    \begin{align*}
    	\|\psi_{X_{0}}-\psi_{0}\|_{L^{\infty}(B_{1})}\leqslant(L+2C_{1})|X_{1}-X_{0}|^{\alpha}.
    \end{align*}
    This implies $|\nu_{X_{0}}-e_{2}|\leqslant C|X_{1}-X_{0}|^{\alpha}$ by a direct calculation. Let $X\in\mathcal{C}_{2\varepsilon}^{\pm}(X_{1},e_{2})$. Then,
    \begin{align*}
    	\nu_{X_{1}}\cdot(X-X_{0})&=e_{2}\cdot(X-X_{1})+(\nu_{X_{1}}-e_{2})\cdot(X-X_{1})\\
    	&>2\varepsilon|X-X_{1}|-C(\sqrt{2}\delta)^{\alpha}|X-X_{1}|>\varepsilon|X-X_{1}|,
    \end{align*}
    where we choose $\delta$ such that $C(\sqrt{2}\delta)^{\alpha}\leqslant\varepsilon$. This finishes $\mathcal{C}_{2\varepsilon}^{\pm}(X_{1},e_{2})\subset\mathcal{C}_{\varepsilon}^{\pm}(X_{1},\nu_{X_{1}})$ and thus \eqref{Formula: R(30)} is concluded.
    
    \textbf{Step III.} We now prove that $g\colon(-\delta,\delta)\to\mathbb{R}$ is Lipschitz continuous on $(-\delta,\delta)$. Let $x_{1}$, $x_{2}\in(-\delta,\delta)$ where $X_{1}=(x_{1},g(x_{1}))$ and $X_{2}=(x_{2},g(x_{2}))$. Since $X_{1}\notin\mathcal{C}_{2\varepsilon}^{+}(X_{2},e_{2})$, we have that
    \begin{align*}
    	g(x_{1})-g(x_{2})=(x_{1}-x_{2})\cdot e_{2}\leqslant2\varepsilon|x_{1}-x_{2}|+2\varepsilon|g(x_{1})-g(x_{2})|.
    \end{align*}
    A same inequality can be derived since $X_{2}\notin\mathcal{C}_{2\varepsilon}^{+}(X_{1},e_{2})$. Two estimates give that  $(1-2\varepsilon)|g(x_{1})-g(x_{2})|\leqslant2\varepsilon|x_{1}-x_{2}|$. Choosing $\varepsilon\leqslant\frac{1}{4}$, we obtain
    \begin{align*}
    	|g(x_{1})-g(x_{2})|\leqslant4\varepsilon|x_{1}-x_{2}|,
    \end{align*}
    which concludes the proof of the Lipschitz continuity of $g$.
\end{proof}
We now provide the proof that near all non-degenerate points, the free boundary is $C^{1,\gamma}$ smooth for some $\gamma\in(0,1)$.
\begin{proof}[Proof of the second conclusion in Theorem \ref{Theorem: main(1)}]
	Let $X_{0}=(x_{0},y_{0})\in N_{\psi}$ and let us consider the blow-up sequence
	\begin{align*}
		\psi_{X_{0},k}(X)=\frac{\psi(X_{0}+\delta_{k}X)}{\delta_{k}},
	\end{align*}
    with $\delta_{k}\to0$ as $k\to\infty$. Notice that for each $k$, $\psi_{X_{0},k}$ is a solution to
    \begin{align}\label{Formula: R(32)}
    	\begin{cases}
    		\begin{alignedat}{2}
    			\Delta\psi_{X_{0},k}-\frac{\delta_{k}}{x_{0}+\delta_{k}x}\pd{\psi_{X_{0},k}}{x}&=\delta_{k}\tilde{f}(X_{0}+\delta_{k}X)\quad&&\text{ in }B_{1}(0)\cap\{\psi_{X_{0},k}>0\},\\
    			|\nabla\psi_{X_{0},k}|&=(x_{0}+\delta_{k}x)\sqrt{(-y_{0}-\delta_{k}y)}\quad&&\text{ on }B_{1}(0)\cap\partial\{\psi_{X_{0},r}>0\}.
    		\end{alignedat}
    	\end{cases}
    \end{align}
    Since $\psi$ is a local minimizer of $J$ in $B_{R_{0}}(X_{0})$, by Proposition \ref{Proposition: minimizers are viscosity solutions}, $\psi_{X_{0},k}$ satisfied \eqref{Formula: R(32)} in the viscosity sense. Let $\bar{\varepsilon}$ be the universal constant in Lemma \ref{Lemma: Flatness implies C1gamma}. For $k$ large enough, the assumption \eqref{Formula: R(2(2))} is satisfied for $\bar{\varepsilon}$. In fact, in $B_{1}(0)$, we have 
    \begin{align*}
    	&|\delta_{k}\tilde{f}(X_{0}+\delta_{k}X)|\leqslant\delta_{k}\|\tilde{f}\|_{L^{\infty}}\leqslant\bar{\varepsilon},\\
    	&\left|\frac{\delta_{k}}{x_{0}+\delta_{k}x}\right|\leqslant\delta_{k}\left|\frac{1}{x_{0}+\delta_{k}x}\right|\leqslant\bar{\varepsilon},
    \end{align*}
    and
    \begin{align*}
    	|(x_{0}+\delta_{k}x)\sqrt{(-y_{0}-\delta_{k}y)}-x_{0}\sqrt{-y_{0}}|\leqslant\bar{\varepsilon}.
    \end{align*}
    Thus, using the nondegeneracy (Lemma \ref{Lemma: Non-degeneracy of minimizers}) and uniform Lipschitz continuity (Proposition \ref{Proposition: Lipschitz regularity for local minimizers}) of $\psi_{X_{0},k}$, we have from Lemma \ref{Lemma: Structure of the blow-up limits} that up to extract a further subsequence
    \begin{enumerate}
    	\item $\psi_{X_{0},k}\to\psi_{0}$ locally and uniformly in $\mathbb{R}^{2}$,
    	\item $\partial\{\psi_{X_{0},k}>0\}\to\partial\{\psi_{0}>0\}$ locally in the Hausdorff distance in $\mathbb{R}^{2}$,
    \end{enumerate}
    for a globally defined function $\psi_{0}\colon\mathbb{R}^{2}\to\mathbb{R}$. Moreover, by the first property of Corollary \ref{Corollary: Properties of blow-up limits}, $\psi_{0}$ is a global solution to the free boundary problem
    \begin{align}\label{Formula: R(31)}
    	\begin{cases}
    		\Delta\psi_{0}=0\quad&\text{ in }\{\psi_{0}>0\}\cap B_{1}(0),\\
    		|\nabla\psi_{0}|=x_{0}\sqrt{-y_{0}}\quad&\text{ on }\{\psi_{0}>0\}\cap B_{1}(0).
    	\end{cases}
    \end{align}
    It follows from Corollary \ref{Corollary: free boundary Lipschitz} that $\partial\{\psi>0\}$ is a Lipschitz graph in a small neighborhood of $X_{0}$. We read from \eqref{Formula: R(31)} that $\partial\{\psi_{0}>0\}$ is Lipschitz continuous. Thus, it follows from \cite{C1987} that up to a rotation, $\psi_{0}=x_{0}\sqrt{-y_{0}}y^{+}$. Consequently, we conclude that for all $k$ large enough, $\psi_{X_{0},k}$ is $\bar{\varepsilon}$-flat in $B_{1}(0)$, i.e., 
    \begin{align*}
    	x_{0}\sqrt{-y_{0}}(y-\bar{\varepsilon})^{+}\leqslant\psi_{X_{0},k}(X)\leqslant x_{0}\sqrt{-y_{0}}(y+\bar{\varepsilon})^{+}\quad\text{ in }B_{1}(0).
    \end{align*}
    Thus, $\psi_{X_{0},k}$ satisfies the assumption required by Proposition \ref{Proposition: Flatness implies C1alpha} and we conclude that $\partial\{\psi_{X_{0},k}>0\}$ and hence $\partial\{\psi>0\}$ is $C^{1,\gamma}$ in a small neighborhood of $X_{0}$.
\end{proof}
\section{Appendix A}
We exhibit the proof of Proposition \ref{Proposition: properties for local minimizers}.
\begin{proof}
	(1). Let $\phi\in W^{1,2}( B_{R_{0}}(X_{0}))$ be a function which satisfies
	\begin{align*}
		\operatorname{div}\left(\frac{1}{x}\nabla\phi\right)+xf(\phi)=0\quad\text{ in }B_{r}(X_{0})\qquad\phi=\psi\quad\text{ on }\partial B_{r}(X_{0}).
	\end{align*} 
	Note that the existence of $\phi$ can be achieved by minimizing the following functional
	\begin{align*}
		\mathcal{I}(\phi):=\int_{B_{R_{0}}(X_{0})}\frac{1}{x}|\nabla\phi|^{2}-2xF(\phi)dX,
	\end{align*}
	among all functions $\phi\in W^{1,2}(B_{R_{0}}(X_{0}))$ with $\phi-\psi\in W_{0}^{1,2}(B_{R_{0}}(X_{0}))$. Since $\psi$ is a local minimizer of $J$ in $B_{r}(X_{0})$, one has
	\begin{align*}
		0&\leqslant J(\phi)-J(\psi)\\
		&=\int_{B_{r}(X_{0})}-\frac{1}{x}|\nabla(\phi-\psi)|^{2}+\frac{2}{x}\nabla\phi\cdot\nabla(\phi-\psi)-2x(F(\phi)-F(\psi))\:dX\\
		&\quad+\int_{B_{r}(X_{0})}-xy\left(I_{\{\phi>0\}}-I_{\{\psi>0\}}\right)\:dX\\
		&\leqslant\int_{B_{r}(X_{0})}-\frac{1}{x}|\nabla(\phi-\psi)|^{2}+\frac{2}{x}\nabla\phi\cdot\nabla(\phi-\psi)-2xF'(\phi)(\phi-\psi)\:dX\\
		&\quad+\int_{B_{r}(X_{0})}-xy\left(I_{\{\phi>0\}}-I_{\{\psi>0\}}\right)\:dX\\
		&\leqslant\int_{B_{r}(X_{0})}-\frac{1}{x}|\nabla(\phi-\psi)|^{2}\:dX-\frac{9}{4}x_{0}y_{0}\int_{B_{R_{0}}(X_{0})}I_{\{\phi>0\}}\:dX.
	\end{align*}
	Consequently,
	\begin{align*}
		\int_{B_{r}(X_{0})}|\nabla(\phi-\psi)|^{2}\:dX\leqslant C\int_{B_{r}(X_{0})}I_{\{\phi>0\}}\:dX\leqslant CR_{0}^{2},
	\end{align*}
	where $C$ depends only on $X_{0}$. Then the H\"{o}lder regularity follows from a similar argument as in \cite[Theorem 2.1]{ACF1984TMS}. 
	
	(2). For every $\varepsilon\in(0,1)$ we define $\phi:=\psi-\varepsilon\min\{\psi,0\}$. It follows from $\psi\geqslant0$ on $\partial B_{R_{0}}(X_{0})$ that $\phi=\psi$ on $\partial B_{R_{0}}(X_{0})$. Since $\psi$ is a local minimizer of $J$ in $B_{R_{0}}(X_{0})$, we have
	\begin{align*}
		0&\leqslant J(\phi)-J(\psi)\\
		&=\int_{B_{R_{0}}(X_{0})}\frac{[(1-\varepsilon)^{2}-1]}{x}|\nabla\min\{\psi,0\}|^{2}-2x(F((1-\varepsilon)\psi)-F(\psi))\:dX\\
		&\leqslant\int_{B_{R_{0}}(X_{0})}\frac{-2\varepsilon}{x}|\nabla\min\{\psi,0\}|^{2}-2xF'((1-\varepsilon)\min\{\psi,0\})(-\varepsilon\min\{\psi,0\})\:dX.
	\end{align*}
	Therefore,
	\begin{align*}
		\int_{B_{R_{0}}(X_{0})}|\nabla\min\{\psi,0\}|^{2}\:dX\leqslant C_{0}\int_{B_{R_{0}}(X_{0})}F'((1-\varepsilon)\min\{\psi,0\})\min\{\psi,0\}\:dX,
	\end{align*}
	where $C_{0}$ is a constant depending only on $X_{0}$. Since $F'(z)\geqslant0$ for any $z\leqslant0$, we have that
	\begin{align*}
		\int_{B_{R_{0}}(X_{0})}|\nabla\min\{\psi,0\}|^{2}\:dX\leqslant0,
	\end{align*}
	this implies that $\psi(x,y)\geqslant0$ in $B_{R_{0}}(X_{0})$.
	
	(3). Let $\xi\in C_{0}^{\infty}(B_{R_{0}}(X_{0}))$ be any non-negative test function, it follows from $J(\psi)\leqslant J(\psi-\varepsilon\xi)$ that
	\begin{align*}
		0\leqslant\lim_{\varepsilon\to0}\frac{J(\psi-\varepsilon\xi)-J(\psi)}{\varepsilon}=-\int_{B_{R_{0}}(X_{0})}\frac{1}{x}\nabla\psi\cdot\nabla\xi-xF'(\psi)\xi\:dX,
	\end{align*}
	as desired.
	
	(4). It follows from the continuity of $\psi$ that $B_{R_{0}}(X_{0})\cap\{\psi>0\}$ is open. For any $\xi\in C_{0}^{\infty}(B_{R_{0}}(X_{0})\cap\{\psi>0\})$, it is easy to check that $\psi-\varepsilon\xi=\psi$ on $\partial B_{R_{0}}(X_{0})$. Since $J(\psi)\leqslant J(\psi-\varepsilon\xi)$, we have
	\begin{align*}
		0\leqslant\int_{B_{R_{0}}(X_{0})}\frac{-2\varepsilon}{x}\nabla\psi\cdot\nabla\xi-2xF'(\psi-\varepsilon\xi)(-\varepsilon\xi)\:dX+o(\varepsilon),
	\end{align*}
	which implies that
	\begin{align*}
		-\varepsilon\int_{B_{R_{0}}(X_{0})}\frac{1}{x}\nabla\psi\cdot\nabla\xi-xF'(\psi-\varepsilon\xi)\xi\:dX+o(\varepsilon)\geqslant0,
	\end{align*}
	and therefore,
	\begin{align*}
		\int_{B_{R_{0}}(X_{0})}\frac{1}{x}\nabla\psi\cdot\nabla\xi-xF'(\psi-\varepsilon\xi)\xi\:dX=0\quad\text{ for any }\varepsilon.
	\end{align*}
	Passing to the limit as $\varepsilon\to0$ gives the desired result.
\end{proof}
\section{Appendix B}
In this appendix, we display that for the free boundary problem \eqref{Formula: R(2)}, we can also prove that Lipschitz free boundaries are $C^{1,\gamma}$. Namely,
\begin{proposition}[Lipschitz implies $C^{1,\gamma}$]
	Let $u$ be a Lipschitz viscosity solution to \eqref{Formula: R(2)} in $B_{1}(0)$. Assume that $0\in\varGamma(u)$ and $Q(0)>0$. If $\varGamma(u)$ is a Lipschitz graph in a small neighborhood of $0$, then $\varGamma(u)$ is $C^{1,\gamma}$ in a small neighborhood of $B_{\rho}$ of $0$ (the size of which depending on a universal constant).
\end{proposition}
The proof is similar to the proof of (2) in Theorem \ref{Theorem: main(1)}, we write it for the sake of completeness.
\begin{proof}
	Consider the blow-up sequence
	\begin{align*}
		u_{k}(X):=\frac{u(\delta_{k}X)}{\delta_{k}},
	\end{align*}
    with $\delta_{k}\to0$ as $k\to\infty$. Note that for each $k$, $u_{k}$ is a solution to the problem
    \begin{align*}
    	\begin{cases}
    		\begin{alignedat}{2}
    			\displaystyle\sum_{i,j=1}^{2}a_{ij}(\delta_{k}X)D_{ij}u_{k}+\delta_{k}\mathbf{b}(\delta_{k}X)\nabla u_{k}&=\delta_{k}\tilde{f}(\delta_{k}X)\quad&&\text{ in }B_{1}(0)\cap\{u_{k}>0\},\\
    			|\nabla u_{k}|&=Q(\delta_{k}X)\quad&&\text{ on }B_{1}(0)\cap\partial\{u_{k}>0\}.
    		\end{alignedat}
    	\end{cases}
    \end{align*}
    Thus, the assumption \eqref{Formula: R(2(2))} is satisfied for the universal constant $\bar{\varepsilon}$. Indeed, in $B_{1}(0)$,
    \begin{align*}
    	|\delta_{k}\tilde{f}(\delta_{k}X)|&\leqslant\delta_{k}|\tilde{f}(\delta_{k}X)|\leqslant\delta_{k}\|\tilde{f}\|_{L^{\infty}}\leqslant\bar{\varepsilon},\\
    	|\delta_{k}\mathbf{b}(\delta_{k}X)|&\leqslant\delta_{k}|\mathbf{b}(\delta_{k}X)|\leqslant\delta_{k}\|\mathbf{b}\|_{L^{\infty}}\leqslant\bar{\varepsilon},\\
    	|a_{ij}(\delta_{k}X)-\delta_{ij}|&=|a_{ij}(\delta_{k}X)-a_{ij}(0)|\leqslant[a_{ij}(\delta_{k}X)]_{C^{0,\beta}}\delta_{k}^{\beta}\leqslant\bar{\varepsilon},
    \end{align*}
    and
    \begin{align*}
    	|Q(\delta_{k}X)-1|=|Q(\delta_{k}X)-Q(0)|\leqslant[Q(\delta_{k}X)]_{C^{0,\beta}}\delta_{k}^{\beta}\leqslant\bar{\varepsilon}.
    \end{align*}
    Thus, using the uniform Lipschitz continuity of $u_{k}$'s, we have
    \begin{align*}
    	u_{k}\to u_{0}\quad\text{ uniformly on compact sets},
    \end{align*}
    and
    \begin{align*}
    	\varGamma(u_{k})\to\varGamma(u_{0})\quad\text{ in the Hausdorff distance},
    \end{align*}
    where $u_{0}$ is the global solution to the problem
    \begin{align*}
    	\begin{cases}
    		\Delta u_{0}=0\quad&\text{ in }\varOmega^{+}(u_{0}),\\
    		|\nabla u_{0}|=1\quad&\text{ on }\varGamma(u_{0}).
    	\end{cases}
    \end{align*}
    We see from the above equation that $\varGamma(u_{0})$ is a Lipschitz graph, by assumption that $\varGamma(u)$ is a Lipschitz graph, it follows from \cite{AC1981} that $u_{0}=y^{+}$ is a one-plane solution. Thus, for all $k$ large enough, $u_{k}$ is $\bar{\varepsilon}$-flat in $B_{1}$, that is,
    \begin{align*}
    	(y-\bar{\varepsilon})^{+}\leqslant u_{k}(X)\leqslant(y+\bar{\varepsilon})^{+}\quad\text{ in }B_{1}(0).
    \end{align*}
    Thus, $u_{k}$ satisfies the assumption of Proposition \ref{Proposition: Flatness implies C1alpha}, and our conclusion follows.
\end{proof}
\subsection*{Acknowledgement}
The authors would like to express their gratitude to the anonymous referees for their careful reading and valuable suggestions on the initial version of the paper. We are also thankful to the editor for providing helpful suggestions, which have improved our paper.
\subsection*{Conflicts of interest}
The authors declare that they have no conflict of interest.

\end{document}